\documentclass[12pt]{amsart}
\usepackage{amssymb,amscd,verbatim,
mathrsfs, latexsym,amsmath,amsthm}
\usepackage{a4wide}
\usepackage{color}
\numberwithin{equation}{section}
\newtheorem{theorem}{Theorem}[section]
\newtheorem{proposition}[theorem]{Proposition}
\newtheorem{lemma}[theorem]{Lemma}
\newtheorem{corollary}[theorem]{Corollary}
\theoremstyle{definition}
\newtheorem{definition}[theorem]{Definition}
\newtheorem{prob}[theorem]{Problem}
\newtheorem{assumption}[theorem]{Assumption}

\def\Lra{\Leftrightarrow}

\def\tr{\text{tr}}
\DeclareMathOperator{\IM}{Im} \DeclareMathOperator{\RE}{Re}

\DeclareMathOperator{\Res}{Res}
\DeclareMathOperator{\Ker}{Ker}

\DeclareMathOperator{\sgn}{sgn}
\DeclareMathOperator{\Li}{Li}
\DeclareMathOperator{\Log}{Log}

\def\sR{\hbox{I\kern-.1667em\hbox{R}}}

\newcommand{\R}{\mathbb R}

\newcommand{\C}{\mathbb C}

\newcommand{\Z}{\mathbb Z}

\newcommand{\N}{\mathbb N}
\newcommand{\Q}{\mathbb Q}
\newcommand{\bH}{\mathbb H}

\newcommand{\mm}{\underline{m}}
\newcommand{\ve}{\varepsilon}

\def\tr{\hbox{Tr}}

\def\tr{\mathrm{tr}}

\begin{document}
\title[Selberg type zeta functions for Hilbert modular surfaces]
{Differences of the Selberg trace formula 
and Selberg type zeta functions for Hilbert modular surfaces}

\author[Y. Gon]{Yasuro Gon}
\email{ygon@math.kyushu-u.ac.jp}
\address{Faculty of Mathematics\\ Kyushu University\\
Motooka \\ Fukuoka 819-0395\\ Japan}

\thanks{
This work is partially supported by Grant-in-Aid for Scientific Research (C) no. 23540020.}

\date{\today}

\begin{abstract}
{We present the first example of the Selberg type zeta function
for noncompact higher rank locally symmetric spaces.
We study certain Selberg type zeta
functions and Ruelle type zeta functions 
attached to the Hilbert modular 
group of a real quadratic field. We show that they have meromorphic extensions
to the whole complex plane and satisfy functional equations.
The method is based on considering the differences among several Selberg 
trace formulas with different weights
for the Hilbert modular group. 
Besides as an application of the differences
of the Selberg trace formula, we also obtain an asymptotic average 
of the class numbers of indefinite binary quadratic forms over the real quadratic integer ring.}
\end{abstract}


\maketitle

\setcounter{tocdepth}{2}
\tableofcontents

%%%%%%%%%%%%%%%%%%%%%%%%%%%%%%%%%%%%%%%%%%%%%%%%%%%%%%%%%%%%%%%%%%%%%%%
\section{Introduction}
%%%%%%%%%%%%%%%%%%%%%%%%%%%%%%%%%%%%%%%%%%%%%%%%%%%%%%%%%%%%%%%%%%%%%%%

In this article, we consider Selberg type zeta functions attached to
the Hilbert modular group of a real quadratic field.
First of all, we recall the original Selberg zeta function constructed by Selberg in 1956.
Let $\Gamma$ be a co-finite discrete subgroup
of  $\mathrm{PSL}(2,\R)$ acting on the upper half plane $\bH$. 
Take a hyperbolic element $\gamma \in \Gamma$, that is $|\tr(\gamma)| > 2$,
then the centralizer of $\gamma$ in $\Gamma$ is infinite cyclic  
and $\gamma$ is conjugate in $\mathrm{PSL}(2,\R)$ to 
$\Bigl( 
\begin{array}{cc}
N(\gamma)^{1/2} & 0 \\
0 & N(\gamma)^{-1/2}
\end{array} \Bigr) $
with $N(\gamma)>1$. 
Put $\text{Prim}(\Gamma)$ be the set of $\Gamma$-conjugacy classes of 
the primitive hyperbolic elements in $\Gamma$. 
The Selberg zeta function for $\Gamma$ is defined by 
the following Euler product:
\[
Z_\Gamma(s):= \prod_{p \in \text{Prim}(\Gamma)} \prod_{k=0}^{\infty}
\Bigl( 1- N(p)^{-(k+s)} \Bigr) \qquad \mbox{for $\RE(s)> 1$}.
\]
Selberg defined this zeta function 
and proved (Cf. Selberg \cite{Sel1,Sel2}) :
\begin{enumerate}
\item $Z_{\Gamma}(s)$ defined for $\RE(s)> 1$ extends meromorphically 
over the whole complex plane.
\item $Z_{\Gamma}(s)$ has ``non-trivial" zeros at $s=\frac{1}{2} \pm i r_{n}$
of order equal to the multiplicity of the eigenvalue $1/4+r_n^2$
of  the Laplacian $\Delta_0 = -y^2(\frac{\partial^2}{\partial x^2}
+\frac{\partial^2}{\partial y^2})$ acting on 
$L^2(\Gamma \backslash \bH)$.
\item $Z_{\Gamma}(s)$ satisfies a functional equation between $s$ and $1-s$.
\end{enumerate}
The theory of Selberg zeta functions for locally symmetric
spaces of rank one is evolved by Gangolli \cite{Gang} 
(compact case) and Gangolli-Warner \cite{GW} (noncompact case).
For higher rank cases, Deitmar  \cite{D} defined and studied 
``generalized Selberg zeta functions" for compact higher rank locally symmetric spaces.  
(See also Kelmer-Sarnak \cite{KS}).
Therefore, our concern is to define and study ``Selberg type zeta functions" 
for {\it noncompact} higher rank locally symmetric spaces such as
Hilbert modular surfaces.  

Let us explain our main results on Selberg type zeta functions 
for Hilbert modular surfaces in more detail.
Let
$K/\Q$ be a real quadratic field with class number one and
$\mathcal{O}_K$ be the ring of integers of $K$. Put $D$ be the
discriminant of $K$ and $\varepsilon > 1 $ be the fundamental
unit of $K$. 
We denote the generator of $\mathrm{Gal}(K/\Q)$
by $\sigma$ and put $a' := \sigma(a)$ and $N(a) :=a a'$ 
for $a \in K$. 
We also put
$\gamma' =  \Bigl( 
\begin{array}{cc}
a' & b' \\
c' & d'
\end{array} \Bigr)$ for 
$\gamma =  \Bigl( 
\begin{array}{cc}
a & b \\
c & d
\end{array} \Bigr) \in \mathrm{PSL}(2, \mathcal{O}_K)$.
Let $\Gamma_{K} = \{ (\gamma, \gamma') \, | \, 
\gamma \in \mathrm{PSL}(2, \mathcal{O}_K) \}$ 
be the Hilbert modular group of $K$.
It is known that $\Gamma_{K}$ is a
co-finite ({\it non-cocompact}) irreducible discrete subgroup 
of $\mathrm{PSL}(2, \R) \times \mathrm{PSL}(2, \R)$ and $\Gamma_K$ acts on 
the product $\bH^2$ of two copies of the upper half plane $\bH$ 
by component-wise linear fractional transformation.
$\Gamma_K$ have only one cusp $(\infty,\infty)$, i.e. 
$\Gamma_K$-inequivalent parabolic fixed point.
$X_K := \Gamma_K \backslash \bH^2$ 
is called the Hilbert modular surface.

Let $(\gamma, \gamma') \in \Gamma_{K}$ be hyperbolic-elliptic,
i.e, $|\tr(\gamma)|>2$ and $|\tr(\gamma')|<2$. 
Then the centralizer of hyperbolic-elliptic 
$(\gamma, \gamma')$ in $\Gamma_K$ is infinite cyclic.

\begin{definition}[Selberg type zeta function for $\Gamma_{K}$ with the weight $(0,m)$]
For an even integer $m \ge 2$, we define
\label{def:zeta}
\[
Z_{K}(s;m) := \prod_{(p,p')} 
\prod_{k=0}^{\infty} \Bigl( 1-e^{i (m-2) \omega} \, N(p)^{-(k+s)} 
\Bigr)^{-1} 
\quad \mbox{for $\RE(s) >1$} 
\]
\end{definition}

Here, $(p,p')$ run through the set of primitive hyperbolic-elliptic 
$\Gamma_K$-conjugacy classes of $\Gamma_K$, and $(p,p')$ is conjugate 
in $\mathrm{PSL}(2,\R)^2$ to
\[ (p, p') \sim \Bigl( 
\Bigl( 
\begin{array}{cc}
N(p)^{1/2} & 0 \\
0 & N(p)^{-1/2}
\end{array} \Bigr), 
\, 
\Bigl( 
\begin{array}{cc}
\cos \omega & - \sin \omega \\
\sin \omega & \cos \omega
\end{array} \Bigr)\Bigr).
\]
Here, $N(p)>1$, $\omega \in (0,\pi)$ and $\omega \notin \pi \Q$.
The product is absolutely convergent for $\RE(s)>1$.

Our main theorems on analytic properties of $Z_K(s;m)$
are followings. 
\begin{theorem}[Theorems \ref{th:a1} and \ref{th:b1}] \label{th:s1} 
For an even integer $m \ge 2$, 
$Z_K(s;m)$ a priori defined for 
$\RE(s) >1$ has a meromorphic extension over the whole complex plane. 
\end{theorem}

Our Selberg zeta functions $Z_{K}(s;m)$ have also ``non-trivial" zeros or poles
and they have connections with the eigenvalues of two Laplacians.
Let $\Delta_{0}^{(1)} := -y_1^2(\frac{\partial^2}{\partial x_1^2}
   +\frac{\partial^2}{\partial y_1^2})$
and 
$\Delta_{m}^{(2)} := -y_2^2(\frac{\partial^2}{\partial x_2^2}
   +\frac{\partial^2}{\partial y_2^2}) 
+ i  m \, y_2 \frac{\partial}{\partial x_2}$
be the Laplacians of weight $0$ and $m$ 
for $(z_1,z_2) \in \bH^2$.
Two Laplacians $\Delta_0^{(1)}$ and $\Delta_{m}^{(2)}$ act on 
$L^2_\text{dis} ( \Gamma_{K} \backslash \bH^2 
; (0,m) )$, the space of Hilbert Maass forms 
of weight $(0,m)$. (See Definition \ref{def:HM} for definition).

\begin{theorem}[Theorem \ref{th:a1}] \label{th:s2}
For an even integer $m \ge 4$, 
\begin{enumerate}
\item $Z_{K}(s;m)$ has ``non-trivial" zeros at  
$s=\frac{1}{2} \pm i \rho_j(m)$ of order equal to the multiplicity 
of the eigenvalue $\frac{1}{4}+\rho_j(m)^2$ of $\Delta_0^{(1)}$ acting
on $\Ker(\Lambda_m^{(2)})$, 
\item $Z_{K}(s;m)$ has ``non-trivial" poles at  
$s=\frac{1}{2} \pm i \rho_j(m-2)$ of order equal to the multiplicity 
of the eigenvalue $\frac{1}{4}+\rho_j(m-2)^2$ of $\Delta_0^{(1)}$ acting
on $\Ker(\Lambda_{m-2}^{(2)})$.
\end{enumerate}
Here, 
\[ \Ker(\Lambda_{q}^{(2)}) =
\biggl\{  f \in  L^2_{dis} \Bigl( \Gamma_{K} \backslash \bH^2 
\, ; \, (0,q) \Bigr)  \, \Bigl| \,  \Delta_{q}^{(2)} f  = \frac{q}{2} \Bigl( 1-\frac{q}{2} \Bigr) \, f  
\biggr\} \]
for $q=m$ and $m-2$ 
and $\Lambda_q^{(2)}  \colon L^2_{dis} ( \Gamma_{K} \backslash \bH^2 
; (0,q) ) \to L^2_{dis} ( \Gamma_{K} \backslash \bH^2 
; (0,q-2) )$ is a ``weight down" Maass operator. 
For ``trivial zeros" of $Z_{K}(s;m)$, see Theorem \ref{th:a1}. 
\end{theorem}

On the contrary to the case of $m \ge 4$, $Z_{K}(s;2)$ has no ``non-trivial" poles.
\begin{theorem}[Theorem \ref{th:b1}]  \label{th:s3}
$Z_{K}(s;2)$ has ``non-trivial" zeros at  
\begin{enumerate}
\item $s=\frac{1}{2} \pm i \rho_j(2)$ of order equal to the multiplicity 
of the eigenvalue $\frac{1}{4}+\rho_j(2)^2$ of $\Delta_0^{(1)}$ acting
on $\Ker(\Lambda_{2}^{(2)})=
\Bigl\{  f \in  L^2_{dis} \Bigl( \Gamma_{K} \backslash \bH^2 
\, ; \, (0,2) \Bigr)  \, \Bigl| \,  \Delta_{2}^{(2)}  f  = 0
\Bigr\}$, 
\item $s=\frac{1}{2} \pm i \mu_j(-2)$ of order equal to the multiplicity 
of the eigenvalue $\frac{1}{4}+\mu_j(-2)^2$ of $\Delta_0^{(1)}$ acting
on $\Ker(K_{-2}^{(2)})=
\Bigl\{  f \in  L^2_{dis} \Bigl( \Gamma_{K} \backslash \bH^2 
\, ; \, (0,-2) \Bigr)  \, \Bigl| \,  \Delta_{-2}^{(2)} \, f  = 0
\Bigr\}$.
\end{enumerate}
Here, $K_{-2}^{(2)} \colon L^2_{dis} ( \Gamma_{K} \backslash \bH^2 
; (0,-2) ) \to L^2_{dis} ( \Gamma_{K} \backslash \bH^2 
; (0,0) )$ is a ``weight up" Maass operator. 
For ``trivial zeros" of $Z_{K}(s;2)$, see Theorem \ref{th:b1}. 
\end{theorem}
Actually $Z_{K}(s;m)$ has infinite ``non-trivial" zeros by the following ``Weyl's law".
\begin{theorem}[Theorem \ref{th:weyl2}]
For an even integer $m \ge 2$, let
\[
N_m^{+}(T) := \# \bigl\{ j \, \big| \,  1/4 + \rho_j(m)^2 \le T \bigr\}
\]
for $T>0$. Then we have
\[
N_m^{+}(T) \sim (m-1) \frac{\mathrm{vol}(\Gamma_K \backslash \bH^2)}{16 \pi^2}
\, T \quad (T \to \infty).
\]
\end{theorem}
Our $Z_{K}(s;m)$ also satisfy a symmetric functional equation.
\begin{theorem} [Theorems \ref{th:a2} and \ref{th:b2}] \label{th:s4}
\label{th:3}
The zeta function $Z_K(s;m)$ satisfies the 
functional equation
\[ \hat{Z}_K(s;m) = \hat{Z}_K(1-s;m). \]
Here the completed zeta function $\hat{Z}_K(s,m)$ is given by
\[ \hat{Z}_K(s;m) := Z_K(s;m)  
\, Z_{\mathrm{id}}(s) \, Z_{\mathrm{ell}}(s;m)
\, Z_{\mathrm{par/sct}}(s;m) 
\, Z_{\mathrm{hyp2/sct}}(s;m).  \]
Each local Selberg zeta functions corresponding to each $\Gamma_K$-conjugacy classes
of $\Gamma_K$ are
 explicitly given. See Theorems \ref{th:a2} and \ref{th:b2}
for details.
\end{theorem}
We also consider the Ruelle type zeta function.
\begin{definition}[Ruelle type zeta function for $\Gamma_{K}$]
For $\RE(s)>1$, the Ruelle type zeta function for 
$\Gamma_K$ is defined by the following absolutely convergent
Euler product:
\[
R_{K}(s) := \prod_{(p,p')} 
\bigl( 1 - N(p)^{-s} \bigr)^{-1}. 
\]
Here, $(p,p')$ run through the set of primitive hyperbolic-elliptic 
$\Gamma_K$-conjugacy classes of $\Gamma_K$, and $(p,p')$ is conjugate 
in $\mathrm{PSL}(2,\R)^2$ to
\[ (p, p') \sim \Bigl( 
\Bigl( 
\begin{array}{cc}
N(p)^{1/2} & 0 \\
0 & N(p)^{-1/2}
\end{array} \Bigr), 
\, 
\Bigl( 
\begin{array}{cc}
\cos \omega & - \sin \omega \\
\sin \omega & \cos \omega
\end{array} \Bigr)\Bigr).
\]
Here, $N(p)>1$, $\omega \in (0,\pi)$ and $\omega \notin \pi \Q$.
\end{definition}
By the relation
\[ R_{K}(s) = \frac{Z_{K}(s;2)}{Z_{K}(s+1;2)}, \]
we have
\begin{theorem}[Theorem \ref{th:Ruelle}] 
The function $R_{K}(s)$ has a meromorphic continuation to the 
whole $\C$. $R_{K}(s)$ has a double pole at $s=1$ and nonzero for 
$\RE(s) \ge 1$.
\end{theorem}

As a byproduct of Theorem  \ref{th:s4}, 
we obtain a simple functional equation for $R_{K}(s)$
and an explicit formula of the coefficient of the leading term of $R_{K}(s)$ at $s=0$.  
\begin{theorem}[Corollary \ref{cor:ruelle}]
Let $D$ be the discriminant of $K$ and $D \ge 13$. Then,   
the function $R_{K}(s)$ satisfy the functional equation
\begin{equation*}
\begin{split}
R_{K}(s) \, R_{K}(-s)=&(-1)^{E(X_K)} \, 2^{2E(X_K)}
\sin (\pi s)^{2E(X_K) -2a_2(\Gamma) -2 a_3(\Gamma)} \\
& \cdot \sin \Bigl( \frac{\pi s}{2} \Bigr)^{2a_2(\Gamma)} 
 \sin \Bigl( \frac{\pi s}{3} \Bigr)^{2 a_3(\Gamma)} 
 \biggl( \frac{\zeta_{\ve}(s-1) \, \zeta_{\ve} (s+1)}{\zeta_{\ve}(s)^2} \biggr)^2
 \end{split}
\end{equation*}
and the absolute value of the coefficient of the leading term of $R_{K}(s)$ at $s=0$ is 
given by
\[ |R_{K}^{*}(0)|
  = \frac{(2 \pi)^{E(X_{K})}}{2^{a_2(\Gamma)} \, 3^{a_3(\Gamma)}} 
\frac{(2 \ve \log \ve)^2}{(\ve^2-1)^2}.
\]
Here, $E(X_{K})$ denotes the Euler characteristic of $X_{K}$,
$\ve$ is the fundamental unit of $K$, 
$\zeta_\ve(s)=(1-\ve^{-2s})^{-1}$ and
$a_r(\Gamma)$ is the number of elliptic fixed points
in $X_K$ for which  corresponding points have isotropy  groups
of order $r$.  
For $D=5,8$ or $12$, See Theorem \ref{th:ruelle-fe}
and Corollary \ref{cor:5812}.
\end{theorem}

These analytic properties and functional equations of $Z_K(s;m)$ and $R_{K}(s)$ are 
obtained by using the ``differences" of the Selberg trace formula for Hilbert modular 
surfaces. The key point is considering the differences between two Selberg trace formulas with
different weights. For this we shall extend the Selberg trace formula for Hilbert modular group
$\Gamma_{K}$ with trivial weight (Cf. Efrat \cite{E} and  Zograf \cite{Z}) to that with non-trivial weights 
(Theorem \ref{th:trf}).
Based on our Selberg trace formula for $\Gamma_{K}$ with weight $(0,m)$, 
we can treat and obtain the differences and double differences of the Selberg trace formula
(Theorems \ref{th:dtrf} and \ref{th:ddtrf}).

As an application of ``Double differences of the Selberg trace formula"
(Theorem \ref{th:ddtrf}), 
we obtain a prime geodesic type theorem 
(Theorem \ref{th:pgt}) and
a generalization of Sarnak's theorem \cite{S1} on
class numbers of indefinite binary quadratic forms over $\Z$ 
to that for class numbers of indefinite binary quadratic forms 
over $\mathcal{O}_{K}$.  
Put
$ \mathcal{D}_{+-} := \{ d \in \mathcal{O}_{K} \, | \,  
\exists b \in \mathcal{O}_K \, \mbox{ s.t. } \, d \equiv b^2
\pmod{4}, \, d \mbox{ not a square in $\mathcal{O}_{K}$}, \, 
d > 0, \, d' < 0 \}$.
For each $d \in \mathcal{D}_{+-}$, let $h_{K}(d)$ denote the number of 
inequivalent primitive binary quadratic forms over $\mathcal{O}_{K}$ 
of discriminant $d$,  
and let $(x_d,y_d) \in \mathcal{O}_{K} \times \mathcal{O}_{K}$ 
be the fundamental solution of the Pellian equation $x^2-dy^2=4$. 
Put 
$\varepsilon_K(d): = (x_d+\sqrt{d} \, y_d)/2$.
\begin{theorem}[Theorem \ref{th:class-number}]
For $x \ge 2$, we have
%\begin{equation*}
%\begin{split}
%\sum_{\substack{d \in \mathcal{D}_{+-} \\
%    \varepsilon_K(d) \le x}}
%h_{K}(d) \log \varepsilon_K(d) = & \, x^2
%- \frac{1}{2} \sum_{1/2< s_j(2) <1} \frac{X^{2 s_j(2)}}{s_j(2)} 
%- \frac{1}{2} \sum_{1/2< s_j(2) <1} \frac{X^{2 s_j(-2)}}{s_j(-2)} \\
%& + O(x^{3/2}) \quad \quad (x \to \infty), 
%\end{split}
%\end{equation*}
\begin{equation*}
\begin{split}
\sum_{\substack{d \in \mathcal{D}_{+-} \\
    \varepsilon_K(d) \le x}}
h_{K}(d)
=& \, 2 \Li(x^2) 
- \sum_{1/2< s_j(2) <1}  \Li \big( x^{2 s_j(2)} \bigr) 
- \sum_{1/2< s_j(-2) <1}  \Li \big( x^{2 s_j(-2)} \bigr) \\
& +O(x^{3/2} /\log x) \quad \quad (x \to \infty).
\end{split}
\end{equation*}
Here, $s_j(2) \Bigl( 1-s_j(2) \Bigr)$ and $s_j(-2) \Bigl( 1-s_j(-2) \Bigr)$
are eigenvalues of the Laplacian $\Delta_0^{(1)}$
acting on $\Ker(\Lambda_2^{(2)})$ and $\Ker(K_{-2}^{(2)})$
respectively. See Theorem \ref{th:s3} for definition of
 $\Ker(\Lambda_2^{(2)})$ and $\Ker(K_{-2}^{(2)})$.
\end{theorem}

%%%%%%%%%%%%%%%%%%%%%%%%%%%%%%%%%%%%%%%%%%%%%%%%%%%%%%%%%%%%%%%%%%%%%%%%%%%
%%%%%%%%%%%%%%%%%%%%%%%%%%%%%%%%%%%%%%%%%%%%%%%%%%%%%%%%%%%%%%%%%%%%%%%%%%%
\section{The Selberg trace formula 
for Hilbert modular surfaces with non-trivial weights} 
%%%%%%%%%%%%%%%%%%%%%%%%%%%%%%%%%%%%%%%%%%%%%%%%%%%%%%%%%%%%%%%%%%%%%%%%%%%%
%%%%%%%%%%%%%%%%%%%%%%%%%%%%%%%%%%%%%%%%%%%%%%%%%%%%%%%%%%%%%%%%%%%%%%%%%%%%

\subsection{Hilbert modular group of a real quadratic field} 
Let
$K/\Q$ be a real quadratic field with class number one and
$\mathcal{O}_K$ be the ring of integers of $K$. Put $D$ be the
discriminant of $K$ and $\varepsilon > 1 $ be the fundamental
unit of $K$. 
We denote the generator of $\mathrm{Gal}(K/\Q)$
by $\sigma$ and put $a' := \sigma(a)$ and $N(a) :=a a'$ 
for $a \in K$. 
We also put
$\gamma' =  \Bigl( 
\begin{array}{cc}
a' & b' \\
c' & d'
\end{array} \Bigr)$ for 
$\gamma =  \Bigl( 
\begin{array}{cc}
a & b \\
c & d
\end{array} \Bigr) \in \mathrm{PSL}(2, \mathcal{O}_K)$.

Let $G$ be $\mathrm{PSL}(2,\R)^2 
= \Bigl( \mathrm{SL}(2,\R) / \{ \pm I \} \Bigr)^2$ and 
$\bH^2$ be the direct product of two copies of the upper half plane 
$\bH := \{ z \in \C \, | \,  \IM(z)>0  \}$.
The group $G$ acts on $\bH^2$ by 
\[ g.z = (g_1,g_2).(z_1,z_2) = 
\biggl( \frac{a_1z_1+b_1}{c_1z_1+d_1},\frac{a_2z_2+b_2}{c_2z_2+d_2} \biggr) \in \bH^2
\]
for $g=(g_1,g_2)=
( \Bigl( 
\begin{array}{cc}
a_1 & b_1 \\
c_1 & d_1
\end{array} \Bigr),
 \Bigl( 
\begin{array}{cc}
a_2 & b_2 \\
c_2 & d_2
\end{array} \Bigr)
)$ and $z=(z_1,z_2) \in \bH^2$.

A discrete subgroup $\Gamma \subset G$ is called irreducible if 
it is not commensurable with any direct product $\Gamma_1 \times \Gamma_2$ 
of two discrete subgroups of $\mathrm{PSL}(2,\R)$. 
We have classification of the elements of irreducible $\Gamma$.
\begin{proposition}[Classification of the elements]
Let $\Gamma$ be an irreducible discrete subgroup of $G$. 
Then any element of $\Gamma$ is one of the followings. 
\begin{enumerate}
\item $\gamma=(I,I)$ is the identity 
\item $\gamma=(\gamma_1,\gamma_2)$ is hyperbolic \, $\Lra \, 
|\tr(\gamma_1)| > 2$ and $|\tr(\gamma_2)| > 2$  
\item $\gamma=(\gamma_1,\gamma_2)$ is elliptic \, $\Lra \, 
|\tr(\gamma_1)| < 2$ and $|\tr(\gamma_2)| < 2$
\item $\gamma=(\gamma_1,\gamma_2)$ is hyperbolic-elliptic \, $\Lra \, 
|\tr(\gamma_1)| > 2$ and $|\tr(\gamma_2)| < 2$  
\item $\gamma=(\gamma_1,\gamma_2)$ is elliptic-hyperbolic \, $\Lra \, 
|\tr(\gamma_1)| < 2$ and $|\tr(\gamma_2)| > 2$  
\item $\gamma=(\gamma_1,\gamma_2)$ is parabolic \, $\Lra \, 
|\tr(\gamma_1)| = |\tr(\gamma_2)| = 2$ 
\end{enumerate}
\end{proposition}

Note that there are no other types in $\Gamma$. (parabolic-elliptic etc.)
(Cf. Shimizu \cite{Sh})

Let us consider the Hilbert modular group of the real quadratic field $K$,
\[ \Gamma_{K} := 
\Bigl\{ (\gamma,\gamma') =  \Bigl( 
\Bigl( 
\begin{array}{cc}
a & b \\
c & d
\end{array} \Bigr), 
\, 
\Bigl( 
\begin{array}{cc}
a' & b' \\
c' & d'
\end{array} \Bigr)
\Bigr)
\Big| \, 
\Bigl( 
\begin{array}{cc}
a & b \\
c & d
\end{array} \Bigr) \in 
\mathrm{PSL}(2,\mathcal{O}_K)
\Big\}. 
\]

It is known that $\Gamma_K$ is an irreducible discrete subgroup 
of $G=\mathrm{PSL}(2,\R)^2$ with the only one cusp 
$\infty := (\infty, \infty)$,
i.e. $\Gamma_K$-inequivalent parabolic fixed point.
$X_K = \Gamma_K \backslash \bH^2$ 
is called the Hilbert modular surface.

We can easily see that
\begin{lemma}[Stabilizer of the cusp $\infty=(\infty,\infty)$]
\label{lem:gamma-inf}
The stabilizer of $\infty=(\infty,\infty)$ in $\Gamma_{K}$ is
given by
\[ \Gamma_{\infty} := 
\Bigl\{ \Bigl( 
\Bigl( 
\begin{array}{cc}
u & \alpha \\
0 & u^{-1}
\end{array} \Bigr), 
\, 
\Bigl( 
\begin{array}{cc}
u' & \alpha' \\
0 & u'^{-1}
\end{array} \Bigr)
\Bigr)
\Big| \, 
u \in \mathcal{O}_{K}^{\times}, \, \alpha \in \mathcal{O}_{K}
\Big\}. 
\]
\end{lemma}

\begin{definition}[Types of hyperbolic elements] \label{d:type}
For a hyperbolic element $\gamma$, we define that
\begin{enumerate}
\item $\gamma$ is type 1 hyperbolic $\Lra$ whose all fixed points are 
not fixed by parabolic elements.
\item $\gamma$ is type 2 hyperbolic $\Lra$ not type 1 hyperbolic.
\end{enumerate}  
\end{definition}

\begin{lemma} \label{lem:hyp2}
Any type 2 hyperbolic elements of $\Gamma_{K}$ are conjugate
to an element of 
\[
\Bigl\{ \gamma_{k,\alpha} =  
\Bigl( 
\begin{array}{cc}
\varepsilon^{k} & \alpha \\
0 & \varepsilon^{-k}
\end{array} \Bigr)
\, \Big| \, k \in \N, \, \alpha \in 
\mathcal{O}_{K} \Bigr\}
\] 
in $\Gamma_{K}$. The centralizer of $\gamma_{k,\alpha}$ in $\Gamma_{K}$
is an infinite cyclic group. 
\end{lemma}
\begin{proof}
See pp.91--93 in \cite{E}.
\end{proof}

By the above lemma, we may take a generator of the centralizer
$Z_{\Gamma_{K}}(\gamma_{k,\alpha})$ as $\gamma_{k_0,\beta}$
with $k_0 \in \N$ and $\beta \in \mathcal{O}_{K}$.
We also write $k_0$ as $k_0(\gamma_{k,\alpha})$.

Let $R_1,R_2,\cdots,R_{N}$ be a complete system of representatives
of the $\Gamma_K$-conjugacy classes of primitive elliptic elements of
$\Gamma_K$. $\nu_1,\nu_2, \cdots, \nu_N$ $(\nu \in \N, \, \nu \ge 2)$
denote the orders of $R_1,R_2,\cdots,R_{N}$.
We may assume that $R_j$ is conjugate 
in $\mathrm{PSL}(2,\R)^2$ to
\[ R_j \sim \Bigl( 
\Bigl( 
\begin{array}{cc}
\cos \frac{\pi}{\nu_j} & - \sin \frac{\pi}{\nu_j} \\
\sin \frac{\pi}{\nu_j} & \cos \frac{\pi}{\nu_j}
\end{array} \Bigr), 
\, 
\Bigl( 
\begin{array}{cc}
\cos \frac{t_j \pi}{\nu_j} & - \sin \frac{t_j \pi}{\nu_j} \\
\sin \frac{t_j \pi}{\nu_j} & \cos \frac{t_j \pi}{\nu_j}
\end{array} \Bigr)\Bigr), 
\quad (t_j,\nu_j)=1. 
\]
For even natural number $m \ge 2$ and 
$l \in \{0,1,\cdots, \nu_j-1\}$, we define the integers
$\alpha_l(m,j), \, \overline{\alpha_l}(m,j) \in \{0,1, \cdots, \nu_j -1 \}$
by
\begin{equation} \label{eq:alpha}
\begin{split}
& l + t_j (\frac{m-2}{2}) \equiv \alpha_l(m,j) \pmod{\nu_j} \\
& l - t_j (\frac{m-2}{2}) \equiv \overline{\alpha_l}(m,j) 
\pmod{\nu_j}\\
\end{split}
\end{equation}

We denote by $\Gamma_{\mathrm{H1}}$, 
$\Gamma_{\mathrm{E}}$, 
$\Gamma_{\mathrm{HE}}$,
$\Gamma_{\mathrm{EH}}$ and
$\Gamma_{\mathrm{H2}}$,  
type 1 hyperbolic $\Gamma_K$-conjugacy classes,
elliptic $\Gamma_K$-conjugacy classes,
hyperbolic-elliptic $\Gamma_K$-conjugacy classes, 
elliptic-hyperbolic $\Gamma_K$-conjugacy classes and
type 2 hyperbolic $\Gamma_K$-conjugacy classes
of $\Gamma_K$ respectively.

\subsection{Preliminaries for the Selberg trace formula}

Fix the weight $(m_1,m_2) \in (2 \Z)^2$. 
Set the automorphic factor $j_{\gamma}(z_j) = \frac{cz_j+d}{|cz_j+d|}$ for 
$\gamma \in \mathrm{PSL}(2,\R)$ $(j=1,2)$. 

Let $\Delta_{m_j}^{(j)} := -y_j^2(\frac{\partial^2}{\partial x_j^2}
   +\frac{\partial^2}{\partial y_j^2}) 
+ i  m_j \, y_j \frac{\partial}{\partial x_j} \quad (j=1,2)$
be the Laplacians of weight $m_j$ for the variable $z_j$.

Let us define the $L^2$-space of 
automorphic forms of weight $(m_1,m_2)$ with respect to the 
Hilbert modular group $\Gamma_{K}$.
\begin{definition}[$L^2$-space of automorphic forms of weight $(m_1,m_2)$]
\begin{eqnarray*}
 &&  L^2(\Gamma_K \backslash \bH^2 \, ; \, (m_1,m_2)) :=  
\Bigl\{ f \colon \bH^2 \to \C, \, C^{\infty} \, \Big| \, \\
&&  (i) \, 
f((\gamma,\gamma')(z_1,z_2)) 
= j_{\gamma}(z_1)^{m_1} j_{\gamma'}(z_2)^{m_2}f(z_1,z_2) 
\quad \forall (\gamma,\gamma') \in \Gamma_K \\
&&  (ii) \, 
\exists (\lambda^{(1)},\lambda^{(2)}) \in \R^2 \quad 
\Delta_{m_1}^{(1)} \, f(z_1,z_2) = \lambda^{(1)} f(z_1,z_2), \quad 
\Delta_{m_2}^{(2)} \, f(z_1,z_2) = \lambda^{(2)} f(z_1,z_2) \\
&& (iii) \, ||f ||^2 = \int_{\Gamma_K \backslash \bH^2} f(z) 
\overline{f(z)} \, 
d \mu (z) < \infty.
\Bigr\}
\end{eqnarray*}
Here, $d \mu(z) = \frac{dx_1dy_1}{y_1^2} \frac{dx_2dy_2}{y_2^2}$
for $z=(z_1,z_2) \in \bH^2$.
\end{definition}

We denote $C_{c}^{\infty}(\R^2)$ the space of 
compactly supported smooth functions on $\R^2$.  

Take $\Phi \in C_{c}^{\infty}(\R^2)$ and introduce the point-pair invariant
kernel $k(z,w)$ of weight $(m_1,m_2)$ 
for $\Phi$ (as (6.3) on \cite[p.386]{H2}):
\begin{equation} \label{def:k}
k(z,w) := \Phi \biggl[ \frac{|z_1-w_1|^2}{\IM z_1 \IM w_1},  
\frac{|z_2-w_2|^2}{\IM z_2 \IM w_2}
 \biggr]  H_{(m_1,m_2)}(z,w) 
\end{equation}
for $(z,w) = ((z_1,z_2),(w_1,w_2)) \in \bH^2 \times \bH^2$. Here, 
\[ 
H_{(m_1,m_2)}(z,w) := H_{m_1}(z_1,w_1) \, H_{m_2}(z_2,w_2) \]
with
\[
H_{m_j}(z_j,w_j) := i^{m_j} \frac{(w_j-\bar{z_j})^{m_j}}{|w_j-\bar{z_j}|^{m_j}}
=  i^{m_j} \frac{|z_j-\bar{w_j}|^{m_j}}{(z_j-\bar{w_j})^{m_j}}
\]
for $j=1,2$. The reason of the last equality is that  $m_1,m_2$ are even integers.
(See \cite[Definition 2.1, p.359]{H1} and \cite[(5.1), p.349]{H2}).

\begin{definition} \label{def:Qgh}
For $\Phi \in C_{c}^{\infty}(\R^2)$, 
define
\begin{align}
Q(w_1,w_2) :=& 
\int \! \! \! \int_{\R^2}  \Phi(w_1+v_1^2,w_2+v_2^2)
\prod_{j=1}^2 \biggl[ \frac{\sqrt{w_j+4}+iv_j}{\sqrt{w_j+4}-iv_j} \biggr] ^{m_j/2}
dv_1 dv_2 \\ 
(w_1, w_2 \ge 0)&, \nonumber \\
g(u_1, u_2) := & Q(e^{u_1}+e^{-u_1}-2, e^{u_2}+e^{-u_2}-2), \\
h(r_1, r_2)  := & \int \! \! \! \int_{\R^2}
g(u_1,u_2) \,e^{i(r_1u_1+r_2u_2)} \, du_1 du_2.
\end{align}
\end{definition}
We can easily check that 
$Q(w_1,w_2) \in C_c^{\infty}([0,\infty)^2)$, 
$g(u_1,u_2) \in C_c^{\infty}(\R^2)$ is an even function  
and $h(r_1,r_2) \in C^{\infty}(\R^2)$ is 
an even and rapidly decreasing function. 

\begin{proposition} \label{prop:phi}
\begin{align}
\Phi(x_1,x_2) =& \Bigl( -\frac{1}{\pi} \Bigr)^2 \int \! \! \! \int_{\R^2} 
\frac{\partial^2 Q}{\partial w_1\partial w_2}(x_1+t_1^2, x_2+t_2^2) 
\prod_{j=1}^2 \biggl[ \frac{\sqrt{x_j+4+t_j^2}-t_j}{\sqrt{x_j+4+t_j^2}+t_j} \biggr] ^{m_j/2}
dt_1 dt_2 \\
(x_1, x_2 \ge 0)&, \nonumber \\
g(u_1, u_2) =&  \Bigl( \frac{1}{2 \pi} \Bigr)^2 \int \! \! \! \int_{\R^2}
h(r_1,r_2) e^{-i(r_1u_1+r_2u_2)} \, dr_1 dr_2.
\end{align}
\end{proposition}
\begin{proof}
See \cite[p.386]{H2}, 
\cite[Proposition 2.2]{E} and 
\cite[(1.1.1)]{Z}.
\end{proof}

\subsection{Eisenstein series}

Let $(m_1,m_2) \in (2 \Z)^2$, $z=(z_1,z_2) = (x_1+i y_1, x_2+i y_2) 
\in \bH^2$, and $(s_1,s_2) \in \C^2$ with $\RE(s_1), \RE(s_2) \gg 0$.   
We define, 
\[E_{(m_1,m_2)}(z,s_1,s_2)
:= \sum_{\gamma \in \Gamma_{\infty} \backslash \Gamma_{K}}
\frac{y_1^{s_1}}{|cz_1+d|^{2s_1}}
\frac{y_2^{s_2}}{|c'z_2+d'|^{2s_2}}
\frac{|cz_1+d|^{m_1}}{(cz_1+d)^{m_1}}
\frac{|c'z_2+d'|^{m_2}}{(c'z_2+d')^{m_2}}
.\]

\begin{definition}[Family of Eisenstein series]
For $(m_1,m_2) \in (2 \Z)^2$, $z=(z_1,z_2) = (x_1+i y_1, x_2+i y_2) 
\in \bH^2$, $s \in \C$ with $\RE(s) >1$ and $k \in \Z$ we define 
\begin{equation}
E_{(m_1,m_2)}(z,s;k) := 
E_{(m_1,m_2)} \Bigl( z,s +\frac{\pi i k}{2 \log \varepsilon}, 
s-\frac{\pi i k}{2 \log \varepsilon} \Bigr). 
\end{equation}
\end{definition}

\begin{proposition}
For $\RE(s)>1$, 
the Eisenstein series
$E_{(m_1,m_2)}(z,s;k)$ is absolutely convergent and
\[ E_{(m_1,m_2)}(\gamma z,s;k) = 
j_{\gamma}(z_1)^{m_1} j_{\gamma'}(z_2)^{m_2} \, E_{(m_1,m_2)}(z,s;k) \]
for any $\gamma \in \Gamma_K$. $E_{(m_1,m_2)}(z,s;k)$
is a common eigenfunction of $\Delta_{m_1}^{(1)}$
and $\Delta_{m_2}^{(2)}$. 
\end{proposition}
\begin{proof}
See pp.38--44 in \cite{E}. 
\end{proof}

\begin{proposition}[Fourier expansion of Eisenstein series]
Put
\[ L:=\{ l = (l_1,l_2) \in K^2 \, | \,  l_1 \alpha + l_2 \alpha' \in \Z 
\quad \forall \alpha \in \mathcal{O}_K \} 
\]
and $\langle l, x \rangle :=  l_1x_1 + l_2x_2$ for $x=(x_1,x_2) \in \R^2$.
We write the Fourier coefficients as $a_l(y,s;k)$ for $l \in L$: 
\[ E_{(m_1,m_2)}(z,s;k)
 = \sum_{l \in L} a_l(y,s;k)  \, e^{2 \pi i  \langle l, x \rangle}
.\]
Then the constant term
$a_{0}(y,s;k)$ is given by
\[ y_1^{s+\frac{\pi i k}{2 \log \varepsilon}}
   y_2^{s-\frac{\pi i k}{2 \log \varepsilon}} 
 + \varphi_{(m_1,m_2)}(s,k) \, y_1^{1-s-\frac{\pi i k}{2 \log \varepsilon}}
 y_2^{1-s+\frac{\pi i k}{2 \log \varepsilon}}
\]
with 
\begin{equation} \label{eq:scattering}
\begin{split} 
\varphi_{(m_1,m_2)}(s,k) = & \frac{(-1)^{\frac{m_1+m_2}{2}} \pi}{ \sqrt{D}} \frac{L(2s-1,\chi_{-k})}{L(2s,\chi_{-k})} \frac{\Gamma(s+\frac{\pi i k}{2 \log \varepsilon}-\frac{1}{2})\Gamma(s+\frac{\pi i k}{2 \log \varepsilon})}{\Gamma(s+\frac{\pi i k}{2 \log \varepsilon}+\frac{m_1}{2})\Gamma(s+\frac{\pi i k}{2 \log \varepsilon}-\frac{m_1}{2})} 
\\
& \quad \times \frac{\Gamma(s-\frac{\pi i k}{2 \log \varepsilon}-\frac{1}{2}) \Gamma(s-\frac{\pi i k}{2 \log \varepsilon})}{\Gamma(s-\frac{\pi i k}{2 \log \varepsilon}+\frac{m_2}{2}) \Gamma(s-\frac{\pi i k}{2 \log \varepsilon} - \frac{m_2}{2})}. 
\end{split}
\end{equation}
Here, the Hecke $L$-function $L(s,\chi_{-k})$ is defined by 
$\displaystyle{L(s,\chi_{-k}) := \sum_{0 \ne (c) \subset \mathcal{O}_{K}} 
\Bigl| \frac{c}{c'} \Bigr|^{-\frac{i \pi k}{\log \varepsilon}} |N(c)|^{-s}}$
for $k \in \Z$.

For $l \ne (0,0)$, $a_{l}(y,s;k)$ is given by
\begin{align*}
&  \frac{ (-1)^{\frac{m_1+m_2}{2}} }{ \sqrt{D}} \frac{\sigma_{1-2s,-k}(l)}{L(2s,\chi_{-k})} 
\frac{ \pi^{2s}  |l_1|^{s +\frac{\pi i k}{2 \log \varepsilon} -1}  
\, |l_2|^{s -\frac{\pi i k}{2 \log \varepsilon} -1}}{\Gamma(s+\frac{\pi i k}{2 \log \varepsilon}+\sgn(l_1)\frac{m_1}{2}) \, \Gamma(s-\frac{\pi i k}{2 \log \varepsilon}+\sgn(l_2)\frac{m_2}{2})} 
\\
& \quad \times W_{\sgn(l_1) \cdot \frac{m_1}{2}, \, s+\frac{\pi i k}{2 \log \varepsilon}-\frac{1}{2}}(4 \pi |l_1|y_1) \, 
W_{\sgn(l_2) \cdot \frac{m_2}{2}, \, s-\frac{\pi i k}{2 \log \varepsilon}-\frac{1}{2}}(4 \pi |l_2| y_2).
\end{align*}

Here, $W_{\kappa,\mu}(z)$ is Whittaker's confluent hypergeometric function
$($see \cite[Chapter 16]{WW} for definition$)$. 
and 
$\displaystyle{\sigma_{1-2s,-k}(l) = \sum_{\{c\}, \, \frac{l}{c} \in \mathcal{D}_{k}^{-1}} \frac{\chi_{-k}(c)}{|N(c)|^{2s-1}}}$,
where $\mathcal{D}_{k}^{-1}$ is the inverse different of  $K$. $($See \cite[p.50]{E}$)$.
\end{proposition}
\begin{proof}
The case of $(m_1,m_2)=(0,0)$ is proved in \cite{E}. 
For general case, 
we use the formulas (see \cite[p.55]{Fis} and \cite[3.384 (9)]{GR}):
\[ \int_{-\infty}^{\infty} \frac{dx}{|x+i|^{2s-m} (x+i)^m} 
 =  (-1)^{m/2} \pi^{1/2} 
   \frac{\Gamma(s-\frac{1}{2}) \Gamma(s)}{\Gamma(s+\frac{m}{2}) \Gamma(s-\frac{m}{2})  }
,\]
and 
\[ \int_{-\infty}^{\infty} \frac{e^{- 2 \pi i pxy}}{|x+i|^{2s-m} (x+i)^m} \, dx
 =  (-1)^{m/2}  \frac{\pi^s (|p|y)^{s-1}}{\Gamma(s +\sgn(p) \frac{m}{2})}
 W_{\sgn(p) \cdot \frac{m}{2}, s-\frac{1}{2}} (4 \pi |p|y)
 \]
for $0 \ne p \in \R$ .
The rest of the proof is quite the same as in pp.47--50 in \cite{E}.
\end{proof}

We can prove the following theorem and
proposition by the similar method in pp.58--64 in \cite{E}.

\begin{theorem}[Functional equation]
For any $k \in \Z$,
$E_{(m_1,m_2)}(z,s;k)$ and $\varphi(s,k)$ can be continued meromorphically
to all of $s \in \C$. Moreover, 
we have 
\[ E_{(m_1,m_2)}(z,1-s;-k)
 = \varphi_{(m_1,m_2)}(1-s,-k) \, E_{(m_1,m_2)}(z,s;k)
, \]
and
\[ \varphi_{(m_1,m_2)}(s,k) \, \varphi_{(m_1,m_2)}(1-s,-k) = 1
.\]
\end{theorem}

\begin{proposition}
$E_{(m_1,m_2)}(z,s;k)$ and $\varphi_{(m_1,m_2)}(s,k)$ have no poles
in $\RE(s)>\frac{1}{2}$, 
except for finitely many in $(\frac{1}{2},1]$ when $k=0$.
\end{proposition}

For $Y>1$ and $\mm=(m_1,m_2) \in (2 \Z)^2$, define
\begin{equation}
E_{\mm}^Y(z,s;k) := 
\begin{cases}
 E_{\mm}(z,s;k)  -a_0(y,s;k)  \quad \mbox{if $y_1y_2 \ge Y$,} \\
 E_{\mm}(z,s;k)  \quad \mbox{if $y_1y_2 < Y.$}
\end{cases}
\end{equation}
For $\RE(s)>1$, we note that 
$E_{\mm}^Y(z,s;k)$ is a square-integrable function on the fundamental domain 
for $\Gamma_K$.

\begin{theorem}[Maass-Selberg relation] \label{th:maass}
Let $\mm, \mm' \in (2 \Z)^2$ with $\mm+\mm'=(0,0)$. 
For $(s,k) \ne (s',k')$ and $(s,k)+(s',k') \ne (1,0)$, we have 
\begin{equation}
\begin{split}
& \int_{\Gamma_K \backslash \bH^2} E_{\mm}^Y(z,s;k) E_{\mm'}^Y(z,s';k') \, d \mu (z) \\
& = 2 \sqrt{D} \log \varepsilon \biggl[ \delta_{k,-k'}   \,  \frac{Y^{s+s'-1}  - \varphi_{\mm}(s,k) \, \varphi_{\mm'}(s',k') \, Y^{-s-s'+1}}{s+s'-1}
\\
& \qquad + \delta_{k,k'} \, 
\frac{\varphi_{\mm'}(s',k') \, Y^{s-s'}  - \varphi_{\mm}(s,k) \, Y^{s'-s}}{s-s'}
\biggr].   
\end{split}
\end{equation}
\end{theorem}
\begin{proof}
See pp.66--69 in \cite{E}.
\end{proof}

\subsection{Selberg trace formula for Hilbert modular surfaces}

We consider a certain integral operator $\mathcal{K}_{\Gamma}$
acting on $ L^2(\Gamma_K \backslash \bH^2 \, ; \, (m_1,m_2)) $.
The kernel of this integral operator is given as follows.

\begin{definition}[Automorphic kernel function]
For $(z,w) \in \bH^2 \times \bH^2$ and $\mm = (m_1,m_2) \in (2 \Z)^2$, define
\begin{equation}
\begin{split} 
K_{\Gamma}(z,w) :=& \sum_{\gamma \in \Gamma_{K}}
k(z, \gamma w) \, j_{\gamma}(w) \\
=& \sum_{(\gamma, \gamma') \in \Gamma_{K}}
k \Bigl( (z_1, z_2), (\gamma w_1, \gamma' w_2)  \Bigr)
\cdot \biggl( \frac{c w_1 + d}{|cw_1+d|} \biggr)^{m_1}
\biggl( \frac{c' w_2 + d'}{|c'w_2+d'|} \biggr)^{m_2}.
\end{split}
\end{equation}
Here, the point-pair invariant kernel $k(z,w)$ is defined in (\ref{def:k}).
\end{definition}

It is known that
\begin{proposition}
Let $L^2_{\text{dis}}(\Gamma_K \backslash \bH^2 \, ; \, (m_1,m_2))$
be the subspace of the discrete spectrum 
of $\mathcal{K}_{\Gamma}$ and 
$L^2_{\text{con}}(\Gamma_K \backslash \bH^2 \, ; \, (m_1,m_2))$
be the subspace of the continuous spectrum.
Then, 
we have a direct sum decomposition of $\mathcal{K}_{\Gamma}$-invariant subspaces :  
\[  L^2(\Gamma_K \backslash \bH^2 \, ; \, (m_1,m_2)) 
 = L^2_{\text{dis}}(\Gamma_K \backslash \bH^2 \, ; \, (m_1,m_2))
   \oplus L^2_{\text{con}}(\Gamma_K \backslash \bH^2 \, ; \, (m_1,m_2))  \]
and there is an orthonormal basis $\{ \phi_j \}_{j=0}^{\infty}$
of $L^2_{\text{dis}}(\Gamma_K \backslash \bH^2 \, ; \, (m_1,m_2))$.
\end{proposition}

\begin{definition}[Hilbert Maass forms of weight $(m_1,m_2)$]
\label{def:HM}
Let $(m_1,m_2) \in (2\Z)^2$.
We call   
\[ L^2_{\text{dis}}(\Gamma_K \backslash \bH^2 \, ; \, (m_1,m_2)) \]
the space of Hilbert Maass forms for $\Gamma_K$ of weight $(m_1,m_2)$.
\end{definition}

To subtract continuous spectrum on 
$L^2_{\text{con}}(\Gamma_K \backslash \bH^2 \, ; \, (m_1,m_2))$, 
we introduce (see \cite[p.79]{E} or \cite[p.1644]{Z})
\begin{definition}
For $(z,w) \in \bH^2 \times \bH^2$ and $\mm = (m_1,m_2) \in (2 \Z)^2$, define
\begin{equation}
\begin{split} 
H_{\Gamma}(z,w) := &
  \quad \frac{1}{8 \pi \sqrt{D} \, \log \ve}
\sum_{k \in \Z} \int_{- \infty}^{\infty} h \Bigl( r+\frac{\pi k}{2 \log \ve}, r-\frac{\pi k}{2 \log \ve} \Bigr)
\\ 
& \quad \quad \quad \times 
E_{\mm} \Bigl( z,\frac{1}{2}+ir;k \Bigr) \, E_{-\mm} \Bigl( w,\frac{1}{2}-ir;-k \Bigr) \, dr.
\end{split}
\end{equation}
\end{definition}

%the scattering determinant $\varphi_{(m_1,m_2)}(s,k)$
%defined in (\ref{eq:scattering}).

Let $\{ \phi_{j} \}_{j=0}^{\infty}$ be an orthonormal basis
of $L^2_{\text{dis}}(\Gamma_{K} \backslash \bH^2 
\, ; \, (m_1,m_2))$ and 
$(\lambda_j^{(1)},\lambda_j^{(2)}) \in \R^2$ such that
\[ \Delta_{m_1}^{(1)} \phi_j = \lambda_j^{(1)} \phi_j \quad \mbox{and} \quad
   \Delta_{m_2}^{(2)} \phi_j = \lambda_j^{(2)} \phi_j. \]

\begin{lemma} \label{lem:least-eigen}
For any $j$,
\begin{equation}  \label{eq:least-eigen}
\lambda_j^{(1)}   \ge  \frac{|m_1|}{2} \Bigl( 1-\frac{|m_1|}{2} \Bigr), \quad 
\lambda_j^{(2)}   \ge  \frac{|m_2|}{2} \Bigl( 1-\frac{|m_2|}{2} \Bigr).
\end{equation}
\end{lemma}   
\begin{proof}
See \cite[(6.1), p.385]{H2}.
\end{proof}
   
Let us define the set of spectral parameters
\[ 
\mathrm{Spec}(m_1,m_2) := 
\bigl\{ (r^{(1)}_j,r^{(2)}_j) \bigr\}_{j=0}^{\infty},
\]
which is a discrete subset of 
\[ \biggl( \R \cup i\Bigl[-\frac{||m_1|-1|}{2}, \frac{||m_1|-1|}{2}\Bigr] \biggr)  
\times \biggl( \R \cup i\Bigl[-\frac{||m_2|-1|}{2}, \frac{||m_2|-1|}{2}\Bigr] \biggr) 
. \]

Here, we write $\lambda^{(l)}_j = \tfrac{1}{4}+(r^{(l)}_j)^2$ and 
$r_j^{(i)}$ are defined by
\begin{equation} \label{eq:rj}
r_j^{(l)} := 
\begin{cases}
  \sqrt{\lambda_j^{(l)}- \frac{1}{4}}     \quad \: \, \mbox{if } \lambda_j^{(l)} \ge \frac{1}{4},  \\
 i \sqrt{\frac{1}{4} - \lambda_j^{(l)}}   \quad \mbox{if } \lambda_j^{(l)} < \frac{1}{4}, 
\end{cases}
\end{equation}
for $l=1,2$.

\begin{theorem}
$K_{\Gamma}(z,w) - H_{\Gamma}(z,w)$ is a Hilbert-Schmidt integral kernel, that is
\[ \int \! \! \! \int_{(\Gamma_{K} \backslash \bH^2)^2} |K_{\Gamma}(z,w) - H_{\Gamma}(z,w)|^2 \, 
d \mu (z)  \, d \mu (w)  < \infty.  \]
\end{theorem}
\begin{proof}
Similar with the proof of Theorem 9.7 in \cite{E} or
p.1644 in \cite{Z}.
\end{proof}

We shall assume that
\[ k(z,w) = \int_{\bH^2} k^{(1)}(z,v) \, k^{(2)}(v,w) \, d\mu(v)
\]
where, $k^{(1)}$ and $k^{(2)}$ are defined as (\ref{def:k}). 
Then $K_{\Gamma} - H_{\Gamma}$ defines a integral operator of 
trace class. So, we have

\begin{theorem} \label{th:pre-trf}
\begin{equation} \label{eq:pre}
\sum_{j=0}^{\infty} \, 
h(r_j^{(1)},r_j^{(2)})
= \int_{\Gamma_{K} \backslash \bH^2}
\Bigl[ K_{\Gamma}(z,z) - H_{\Gamma}(z,z)  \Bigr] \, d \mu (z),
\end{equation}
where the left hand side is absolutely convergent.
\end{theorem}
Our next task is to evaluate the right hand side of 
(\ref{eq:pre}) explicitly.

Hereafter, we assume that the test functions are written as follows.
\begin{assumption} \label{assump:test}
We shall assume that the test functions are products of two separate functions 
that each involve only one independent variable. 
That is
\begin{equation}
\begin{split}
& h(r_1,r_2) = h_1(r_1) \, h_2(r_2),  \quad g(u_1,u_2) =g_1(u_1) \, g_2(u_2), \\
& \Phi(x_1,x_2) = \Phi_1(x_1) \, \Phi_2(x_2), \quad Q(w_1,w_2)=Q_1(w_1) \, Q_2(w_2). 
\end{split}
\end{equation}
Without loss of generality 
we may assume that $\Phi_1$ and $\Phi_2$ are real valued.
\end{assumption}

Now we can state the Selberg trace formula for our cases.
We give a proof of this theorem in the next section.
 
\begin{theorem}[Selberg trace formula for $L^2(\Gamma_{K} \backslash \bH^2 
\, ; \, (0,m))$ with $m \in 2 \Z$] \label{th:trf}
Let $g(u_1,u_2)$ be an even function in $C_{c}^{\infty}(\R^2)$ 
and put $h(r_1,r_2) :=
\int_{-\infty}^{\infty} \int_{-\infty}^{\infty}
g(u_1,u_2)
e^{i(r_1u_1+r_2u_2)} \, du_1 du_2$, so that $h$ is even, rapidly decreasing
and analytic.

Then we have,
\begin{equation} 
\sum_{j=0}^{\infty} \, 
h(r_j^{(1)},r_j^{(2)}) = \mathbf{I}(h)+\mathbf{II_a}(h)+\mathbf{II_b}(h)+\mathbf{III}(h).
\end{equation}

Here, 
\begin{align*}
& \mathbf{I}(h) 
:= \ \frac{\mathrm{vol}(\Gamma_K \backslash \bH^2)}{16 \pi^2}
\int \! \! \! \! \int_{\R^2}
\frac{\frac{\partial^2}{\partial u_1 \partial u_2} g(u_1,u_2)}
{\sinh(u_1/2) \sinh(u_2/2)}
e^{-\frac{{m}}{2}u_2} \, du_1 du_2 \\
& \quad + \sum_{(\gamma,\gamma') \in \Gamma_{\mathrm{H1}}}
\frac{\mathrm{vol}(\Gamma_{\gamma} \backslash G_{\gamma}) \, g(\log N(\gamma),\log N(\gamma'))}
{(N(\gamma)^{1/2}-N(\gamma)^{-1/2})(N(\gamma')^{1/2}-N(\gamma')^{-1/2})} \\
& \quad + \sum_{R(\theta_1,\theta_2) \in \Gamma_{\mathrm{E}}} 
\frac{-e^{-i \theta_1+i(m-1)\theta_2}}{16 \nu_{R} \sin \theta_1 \sin \theta_2}  \! \! \int \! \! \! \! \int_{\R^2} \! \! g(u_1,u_2) 
e^{-\frac{u_1}{2}+\frac{(m-1)}{2}u_2}
 \prod _{j=1}^{2} \Bigl[ \frac{e^{u_j}-e^{2 i \theta_j}}{\cosh u_j - \cos 2 \theta_j} \Bigr]
du_1 du_2
\\
& \quad + \sum_{(\gamma,\omega) \in \Gamma_{\mathrm{HE}}} \! \! \!
\frac{\log N(\gamma_0)}{N(\gamma)^{1/2}-N(\gamma)^{-1/2}}
\frac{ie^{i(m-1)\omega}}{4 \sin \omega}
\int_{-\infty}^{\infty} \! \! g(\log N(\gamma),u)
e^{\frac{m-1}{2}u}
\Bigl[ \frac{e^{u}-e^{2 i \omega}}{\cosh u - \cos 2 \omega} \Bigr]
du \\    
& \quad + \sum_{(\omega', \gamma') \in \Gamma_{\mathrm{EH}}} \! \! \!
\frac{\log N(\gamma'_0)}{N(\gamma')^{1/2}-N(\gamma')^{-1/2}}
\frac{ie^{-i\omega'}}{4 \sin \omega'}
\int_{-\infty}^{\infty} \! \! g(u, \log N(\gamma'))
e^{\frac{-1}{2}u}
\Bigl[ \frac{e^{u}-e^{2 i \omega'}}{\cosh u - \cos 2 \omega'} \Bigr]
du,  
\end{align*}

\begin{align*}
& \mathbf{II_a}(h) := 
\Bigl[ \sqrt{D} A_0 - 4 \log \varepsilon (\log 2 +C_{E}) \Bigr] g(0,0)
+ \log \varepsilon \int_{0}^{\infty} \! \! [g(u,0)+g(0,u)] du  \\
& \quad - \frac{\log \varepsilon}{2 \pi^2} \int \! \! \! \! \int_{\R^2} 
\Bigl[ \frac{\Gamma'}{\Gamma}(1+ir_1) + \frac{\Gamma'}{\Gamma}(1+ir_2)
\Bigr] h(r_1,r_2) \, dr_1 dr_2 \\
& \quad + 2 \log \varepsilon 
\int_{0}^{\infty} \frac{g(0,u)}{e^{u/2}-e^{-u/2}}
\Bigl[1-\cosh \frac{m}{2}u \Bigr] du 
\end{align*}

\begin{align*}
& \mathbf{II_b}(h) :=
- 4 \log \varepsilon \sum_{k=1}^{\infty} 
\sum_{\gamma_{k,\alpha} \in \Gamma_{\mathrm{H2}}}
\frac{k_0(\gamma_{k,\alpha}) \log(N(\alpha, \varepsilon^{k}-\varepsilon^{-k}))}{|N(\varepsilon^{k}-\varepsilon^{-k})|} g(2k \log \varepsilon, 2k \log \varepsilon) \\
& \quad + 4 \log \varepsilon \sum_{k=1}^{\infty}
\log(\varepsilon^k -\varepsilon^{-k} ) \, g(2k \log \varepsilon, 2k \log \varepsilon) \\
& \quad + 2 \log \varepsilon \sum_{k=1}^{\infty}
\int_{2k \log \varepsilon}^{\infty} 
\Bigl[ g(u,2k \log \varepsilon)+g(2k \log \varepsilon,u) \Bigr] 
\frac{\cosh(u/2)}{\sinh(u/2)+\sinh(k \log \varepsilon)} \, du \\
& \quad + 2 \log \varepsilon \sum_{k=1}^{\infty}
\int_{2k \log \varepsilon}^{\infty} g(2k \log \varepsilon,u) 
\frac{1 - \cosh \bigl( m( u/2 - k \log \varepsilon) \bigr)}{\sinh(u/2 - k \log \varepsilon)} \, du,
\end{align*}
and
\begin{align*}
& \mathbf{III}(h) :=
\frac{1}{4 \pi} \sum_{k \in \Z} \int_{\R} 
h \Bigl( r+\frac{\pi k}{2 \log \varepsilon},r-\frac{\pi k}{2 \log \varepsilon} 
\Bigr) \,
\frac{\varphi'_{(0,m)}}{\varphi_{(0,m)}} \bigl( \frac{1}{2}+ir, k \bigr) \, dr \\
& - \frac{1}{4} h(0,0) \, \varphi_{(0,m)} \bigl( \frac{1}{2},0 \bigr).
\end{align*}
\end{theorem} 
The series and integrals converges absolutely.
Here, $A_0$ is the constant term of the Laurent expansion of $\zeta_{K}(s)$
at $s=1$ and $C_{E}$ is the Euler constant. 
The case of $(0,m)=(0,0)$ is proved by Zograf \cite{Z} and Efrat \cite{E}.

%%%%%%%%%%%%%%%%%%%%%%%%%%%%%%%%%%%%%%%%%%%%%%%%%%%%%%%%%%%%
%%%%%%%%%%%%%%%%%%%%%%%%%%%%%%%%%%%%%%%%%%%%%%%%%%%%%%%%%%%%
\section{Proof of the Selberg trace formula for Hilbert modular surfaces}
%%%%%%%%%%%%%%%%%%%%%%%%%%%%%%%%%%%%%%%%%%%%%%%%%%%%%%%%%%%%
%%%%%%%%%%%%%%%%%%%%%%%%%%%%%%%%%%%%%%%%%%%%%%%%%%%%%%%%%%%$

\subsection{Orbital integrals and the fundamental domain}
In this section we prove Theorem \ref{th:trf}. 
We recall Theorem \ref{th:pre-trf}: 
\[
\sum_{j=0}^{\infty} \, 
h(r_j^{(1)},r_j^{(2)})
= \int_{\Gamma_{K} \backslash \bH^2}
\Bigl[ K_{\Gamma}(z,z) - H_{\Gamma}(z,z)  \Bigr] \, d \mu (z).
\]
Formally, we have
\[ \int_{\Gamma_K \backslash \bH^2}
K_{\Gamma}(z,z) \, d \mu(z) 
= \sum_{\gamma \in \mathrm{Conj}(\Gamma_K)} 
\int_{\Gamma_\gamma \backslash \bH^2} 
k(z, \gamma z) \, j_{\gamma}(z) \, d \mu(z).
\]
Here, we put $\Gamma_{\gamma}$ be the centralizer of $\gamma$ in $\Gamma_{K}$. 
To prove Theorem \ref{th:trf}, we calculate the orbital integral
\[  \int_{\Gamma_\gamma \backslash \bH^2} 
k(z, \gamma z) \, j_{\gamma}(z) \, d \mu(z)
\]
explicitly for each $\Gamma_{K}$-conjugacy classes $[\gamma]$ of $\Gamma_{K}$.
However, $\gamma$ is parabolic or type 2 hyperbolic, the above
orbital integral does not converge.  

Therefore, we introduce the truncated fundamental domain 
for $\Gamma_{K}$. 
First we construct the fundamental domain $F_{\infty}$ of the group
$\Gamma_{\infty}$. (See Lemma \ref{lem:gamma-inf}).
By direct calculation, we have,
\begin{lemma}[Fundamental domain of $\Gamma_{\infty}$]
Let $D$ be the discriminant of the quadratic field $K$. 
We write $(z_1,z_2)=(x_1+iy_1,x_2+iy_2) \in \bH^2$. 
\begin{enumerate}
\item If $D \equiv 1 \pmod{4}$, put
\begin{equation*}
\begin{split}
F_{\infty} :=   \bigl\{ & (x_1+iy_1, x_2+iy_2) \, | \,  
0 \le \bigl( 1-\tfrac{1}{\sqrt{D}} \bigr) \, x_1+ \bigl( 1+\tfrac{1}{\sqrt{D}} \bigr) \, x_2 < 2, \,   0 \le x_1 - x_2 < 2 \sqrt{D}, \,  \\
& \, \ve^{-2} \le y_1/y_2  < \ve^2  \bigr\}.  
\end{split}
\end{equation*}
\item Otherwise, put
\[  F_{\infty} :=  \bigl\{ (x_1+iy_1, x_2+iy_2) \, | \,  
0 \le x_1+x_2 < 2, \,   0 \le x_1 - x_2 < 2 \sqrt{D}, \,  \ve^{-2} \le y_1/y_2  < \ve^2  \bigr\}.  \]
\end{enumerate}
Then $F_{\infty}$ is a fundamental domain
for the group $\Gamma_{\infty}$ acting on $\bH^2$.
\end{lemma}

We define the standard truncated fundamental domain for $\Gamma_K$.
\begin{definition}[Standard fundamental domain] \label{def:std}
Let $Y>1$. 
\begin{enumerate}
\item The fundamental domain $F$ of  $\Gamma_{K}$, which is contained in 
$F_{\infty}$, is called the standard fundamental domain for $\Gamma_{K}$.
\item $F^{Y}:=\{ (z_1,z_2) \in F \, |\,  y_1y_2 < Y\}$ is called the 
truncated standard fundamental domain for $\Gamma_{K}$.
\item Let $\gamma$ be a parabolic or type 2 hyperbolic element of $\Gamma_{K}$.
\[ F_{\gamma}^Y := \bigcup_{\delta \in \Gamma_{\gamma} \backslash \Gamma} \delta(F^Y)
\]
is called the truncated standard fundamental domain for the centralizer of $\gamma$
in $\Gamma_{K}$. 
\end{enumerate}
\end{definition}

\subsection{Contribution of the identity, type 1 hyperbolic, elliptic 
and mixed elements}
In this subsection, we compute the orbital integral
\[  \int_{\Gamma_\gamma \backslash \bH^2} 
k(z, \gamma z) \, j_{\gamma}(z) \,  d \mu(z)
\]
explicitly, when $\gamma$ is the identity, 
an elliptic, a type 1 hyperbolic, a hyperbolic-elliptic, 
or an elliptic-hyperbolic element. 
We note that all the integrals are convergent
for these elements.
Let $(m_1, m_2) \in (2 \Z)^2$. 

$\bullet$ Identity term: By definition, we have
\begin{align}
& I(m_1,m_2) := \int_{\Gamma_K \backslash \bH^2}
k(z,z) \, d \mu (z) 
=   \int_{\Gamma_K \backslash \bH^2}
H_{(m_1,m_2)}(z,z) \, \Phi(0,0) \, d \mu (z)  \label{term:id} \\
& = (-1)^{m_1+m_2} \mathrm{vol}(\Gamma_K \backslash \bH^2) \, \Phi(0,0).
\nonumber
\end{align}
And $\Phi(0,0)$ is given by (see p.396 in \cite{H2})
\begin{align*}
\Phi(0,0) &= \frac{1}{\pi^2} \int \! \! \! \int_{\R^2} 
\frac{\partial^2 Q}{\partial w_1 \partial w_2}
(t_1^2,t_2^2)
\biggl[ \frac{\sqrt{4+t_1^2}-t_1}{\sqrt{4+t_1^2}+t_1} \biggr]^{m_1/2}
\biggl[ \frac{\sqrt{4+t_2^2}-t_2}{\sqrt{4+t_2^2}+t_2} \biggr]^{m_2/2}
 dt_1 dt_2 \\
&=\frac{1}{4 \pi^2} \int \! \! \! \int_{\R^2}
\frac{\frac{\partial^2}{\partial u_1 \partial u_2} g(u_1,u_2)}
{(e^{u_1/2} - e^{-u_1/2}) (e^{u_2/2} - e^{-u_2/2}) }
e^{-\frac{{m_1}}{2}u_1} e^{-\frac{{m_2}}{2}u_2} 
\, du_1 du_2. 
\end{align*}

$\bullet$ Type 1 hyperbolic terms: 
For type 1 hyperbolic element $(\gamma,\gamma') \in \Gamma_{K}$, 
we denote it by $\gamma$ for simplicity. 
It is known that the centralizer of $\gamma$ in $\Gamma_{K}$ 
is a free abelian group of rank two. (See Theorem 5.7 in \cite[p.26]{E}).
We can easily compute (see also \cite[p.31]{E})
\begin{align}
& H_{1}(m_1,m_2;\gamma)  := \int_{\Gamma_{\gamma} \backslash \bH^2}
k(z,\gamma z) j_{\gamma}(z) \, d \mu(z) \label{term:hyp1} \\
&=\frac{\mathrm{vol}(\Gamma_{\gamma} \backslash \bH^2) \, g(\log N(\gamma),\log N(\gamma'))}
{(N(\gamma)^{1/2}-N(\gamma)^{-1/2})(N(\gamma')^{1/2}-N(\gamma')^{-1/2})}. 
\nonumber
\end{align}

$\bullet$ Elliptic terms: 
Let $R \in \Gamma_K$ be an elliptic element. 
We may assume that
$R$ is conjugate in $G$ to the element
\[ 
R(\theta_1,\theta_2)
=
\Bigl( 
\Bigl( 
\begin{array}{cc}
\cos \theta_1 & - \sin \theta_1 \\
\sin \theta_1 &   \cos \theta_1
\end{array} \Bigr), 
\, 
\Bigl( 
\begin{array}{cc}
\cos \theta_2 & - \sin \theta_2 \\
\sin \theta_2 &   \cos \theta_2
\end{array} \Bigr) \Bigr) .
\]

Let $R_0$ be a generator 
of the centralizer of $R$ in $\Gamma_{K}$ and denote the order of $R_0$ by $\nu_R$.
Then $R_0$ is conjugate in $G$ to the element
\[  \Bigl( 
\Bigl( 
\begin{array}{cc}
\cos (\pi / \nu_R) & - \sin (\pi / \nu_R) \\
\sin (\pi / \nu_R) &   \cos (\pi / \nu_R)
\end{array} \Bigr), 
\, 
\Bigl( 
\begin{array}{cc}
\cos (t \pi / \nu_R) & - \sin (t \pi / \nu_R) \\
\sin (t \pi / \nu_R) &   \cos (t \pi / \nu_R)
\end{array} \Bigr) \Bigr) ,  
\quad (\nu, \, \exists t) = 1.
\]

We write $R=R_0^k$ with $(1 \le k \le \nu_R-1)$, 
and put $(\alpha_j, \, \beta_j) = (\cos \theta_j, \, \sin \theta_j)$
for $j=1,2$.
Using the formulas at pp.389--394 in \cite{H2} 
(see also p.1647 in \cite{Z}), we have
\begin{align}
& E(m_1,m_2;R)  := \int_{<R_0> \backslash \bH^2}
k(z,Rz) j_{R}(z) \, d \mu(z)  \label{term:ell} \\
& = \frac{1}{\nu_{R}}
\int_{\bH^2} k \bigl( (z_1,z_2),(r(\theta_1)z_1,r(\theta_2)z_2) \bigr)
\frac{(\beta_1 z_1+\alpha_1)^{m_1}}{|\beta_1 z_1+\alpha_1|^{m_1}}
\frac{(\beta_2 z_2+\alpha_2)^{m_2}}{|\beta_2 z_2+\alpha_2|^{m_2}}
\, \frac{dx_1dy_1}{y_1^2} \frac{dx_2dy_2}{y_2^2}  \nonumber \\
& = \frac{\pi^2}{4 \nu_{R} \beta_1 \beta_2}
\int_{0}^{\infty} \! \! \int_{0}^{\infty}
e^{i m_1 \arg [2 \alpha_1+i\sqrt{t_1+4 \beta_1^2}]} 
e^{i m_2 \arg [2 \alpha_2+i\sqrt{t_2+4 \beta_2^2}]} 
\frac{\Phi(t_1,t_2)}{\sqrt{t_1^2+4 \beta_1^2} \sqrt{t_2^2+4 \beta_2^2}} 
\, dt_1 dt_2  \nonumber \\
& = \frac{1}{16 \nu_R \beta_1 \beta_2}
\bigl\{ i e^{i(m_1-1)\theta_1} \bigr\}
\bigl\{ i e^{i(m_2-1)\theta_2} \bigr\}
\int \! \! \! \! \int_{\R^2}
g(u_1,u_2) \prod_{j=1}^{2}
\Bigl[ e^{\frac{(m_j-1)u_j}{2}} \frac{e^{u_j}-e^{2 i \theta_j}}{\cosh u_j - \cos 2 \theta_j} \Bigr]
du_1 du_2. \nonumber
\end{align}

$\bullet$ Hyperbolic-elliptic terms: 
Let $\gamma = (\gamma,\gamma') \in G$ be a hyperbolic-elliptic element. 
The group $\Gamma_{\gamma}$ is infinite cyclic and there exists
a generator $\gamma_0 = (\gamma_0, \gamma_0')$ such that 
$ \gamma = \gamma_0^{k}$ with $k \ge 1$. 
We may assume that $(\gamma_0,\gamma_0')$ is conjugate in $G$ 
to the element
\[
\Bigl( 
\Bigl( 
\begin{array}{cc}
N(\gamma_0)^{1/2} & 0 \\
0 &  N(\gamma_0)^{-1/2}
\end{array} \Bigr), 
\, 
\Bigl( 
\begin{array}{cc}
\cos \omega_0 & - \sin \omega_0 \\
\sin \omega_0 &   \cos \omega_0
\end{array} \Bigr) \Bigr). 
\]
Here, $N(\gamma_0)>1$, $\omega_0 \in (0,\pi)$ and $\omega_0 \notin \pi \Q$.
Using the formulas at pp.389 -- 394 in \cite{H2} 
(see also p.1647 in \cite{Z}), we have
\begin{align}
& HE(m_1,m_2;\gamma)  := \int_{\Gamma_{\gamma} \backslash \bH^2}
k(z,\gamma z) j_{\gamma}(z) \, d \mu(z) \label{term:hyp-ell} \\
& = 
\int_{<\gamma_0> \backslash \bH^2} k \bigl( (z_1,z_2),(N(\gamma) z_1,r(\omega)z_2) \bigr)
\frac{(N(\gamma)^{-1/2})^{m_1}}{|N(\gamma)^{-1/2}|^{m_1}}
\frac{(z_2 \sin \omega + \cos \omega)^{m_2}}{|z_2 \sin \omega +\cos \omega|^{m_2}}
\, \frac{dx_1dy_1}{y_1^2} \frac{dx_2dy_2}{y_2^2} \nonumber \\
& = \frac{\log N(\gamma_0)}{N(\gamma)^{1/2}-N(\gamma)^{-1/2}}
\frac{\pi}{2 \sin \omega}
\int_{- \infty}^{\infty}  \Phi_1\Bigl( N(\gamma)+N(\gamma)^{-1}-2+v_1^2 
\Bigr)  \nonumber \\ 
& \quad \Bigl[ 
\frac{N(\gamma)^{1/2}+N(\gamma)^{-1/2}+iv_1}{N(\gamma)^{1/2}+N(\gamma)^{-1/2}
-iv_1} \Bigr]^{m_1/2} \, dv_1
\int_{0}^{\infty}  e^{i m_2 \arg [2 \cos \omega+i\sqrt{t_2+4 \sin^2 \omega}]} \frac{\Phi_2(t_2)}{\sqrt{t_2^2+4 \sin^2 \omega}} \, dt_2
\nonumber \\
&=\frac{\log N(\gamma_0)}{N(\gamma)^{1/2}-N(\gamma)^{-1/2}}
\frac{ie^{i(m_2-1)\omega}}{4 \sin \omega}
\int_{-\infty}^{\infty} \! \! g(\log N(\gamma),u_2)
e^{\frac{m_2-1}{2}u_2}
\Bigl[ \frac{e^{u_2}-e^{2 i \omega}}{\cosh u_2 - \cos 2 \omega} \Bigr]
du_2. \nonumber   
\end{align}

$\bullet$ Elliptic-hyperbolic terms: 
Let $\gamma = (\gamma,\gamma') \in G$ be an elliptic-hyperbolic element. 
The group $\Gamma_{\gamma}$ is infinite cyclic and there exists
a generator $\gamma_0 = (\gamma_0, \gamma_0')$ such that 
$ \gamma = \gamma_0^{l}$ with $l \ge 1$. 
We may assume that $(\gamma_0,\gamma_0')$ is conjugate in $G$ 
to the element
\[
\Bigl( 
\Bigl( 
\begin{array}{cc}
\cos \omega_0' & - \sin \omega_0' \\
\sin \omega_0' &   \cos \omega_0'
\end{array} \Bigr), 
\, 
\Bigl( 
\begin{array}{cc}
N(\gamma_0')^{1/2} & 0 \\
0 &  N(\gamma_0')^{-1/2}
\end{array} \Bigr) \Bigr). 
\]
Here, $N(\gamma_0')>1$, $\omega_0' \in (0,\pi)$ and $\omega_0' \notin \pi \Q$.
Then we have
\begin{align}
& EH(m_1,m_2;\gamma) := \int_{\Gamma_{\gamma} \backslash \bH^2}
k(z,\gamma z) j_{\gamma}(z) \, d \mu(z) \label{term:ell-hyp} \\
&=\frac{\log N(\gamma_0')}{N(\gamma')^{1/2}-N(\gamma')^{-1/2}}
\frac{ie^{i(m_1-1)\omega'}}{4 \sin \omega'}
\int_{-\infty}^{\infty} \! \! g(u_1,\log N(\gamma'))
e^{\frac{m_1-1}{2}u_1}
\Bigl[ \frac{e^{u_1}-e^{2 i \omega'}}{\cosh u_1 - \cos 2 \omega'} \Bigr]
du_1. \nonumber     
\end{align}

Putting together with the all results in this subsection,
we obtain the term $\mathbf{I}(h)$ in Theorem \ref{th:trf}.

\subsection{Parabolic contribution}

Let $\Gamma_{\text{P}}$ be the set of $\Gamma_K$-conjugacy classes of
parabolic elements in $\Gamma_K$. Let $(m_1,m_2) \in (2 \Z)^2$ and $Y>1$.
We consider the parabolic contribution to the trace formula
with the truncation parameter $Y$:
\[ P^{Y}(m_1,m_2) := \sum_{\gamma \in \Gamma_{\text{P}}}
\int_{F_{\gamma}^{Y}} k(z,\gamma z) j_{\gamma}(z) \, d \mu(z).
\]
Here, $F_{\gamma}^{Y} = \bigcup_{\delta \in \Gamma_{\gamma} \backslash \Gamma}
\delta(F^{Y})$ and $F^{Y} = \{(z_1,z_2) \in F \, | \, 
\IM(z_1) \IM(z_2) < Y \}$ is the truncated standard fundamental domain
for $\Gamma_K$, which is defined in Definition \ref{def:std}. 

Then we have
\begin{proposition} \label{prop:py}
For $m \in 2 \Z$ and $Y>1$, we have
\begin{align*}
P^{Y}(0,m) =& 2 \log \ve \, \log Y \, g(0,0) \\ 
& +\Bigl[ \sqrt{D} A_0 - 4 \log \varepsilon (\log 2 +C_{E}) \Bigr] g(0,0)
+ \log \varepsilon \int_{0}^{\infty} \! \! [g(u,0)+g(0,u)] du  \\
& - \frac{\log \varepsilon}{2 \pi^2} \int \! \! \! \! \int_{\R^2} 
\Bigl[ \frac{\Gamma'}{\Gamma}(1+ir_1) + \frac{\Gamma'}{\Gamma}(1+ir_2)
\Bigr] h(r_1,r_2) \, dr_1 dr_2 \\
& + 2 \log \varepsilon 
\int_{0}^{\infty} \frac{g(0,u)}{e^{u/2}-e^{-u/2}}
\Bigl[1-\cosh \frac{m}{2}u \Bigr] du + o(1) \quad (Y \to \infty).
\end{align*}
Here, $A_0$ is the constant term of the Laurent
expansion of $\zeta_{K}(s)$, the Dedekind zeta function of $K$, 
at $s=1$ and $C_{E}:=\lim_{n \to \infty}(1+\frac{1}{2}+\dots+\frac{1}{n} -\log n)$ is the Euler constant.
\end{proposition}

\begin{proof}
We recall that the test function $\Phi$ is written as 
$\Phi(x_1,x_2)=\Phi_1(x_1) \, \Phi_2(x_2)$ 
with real valued $\Phi_1$ and $\Phi_2$ 
by Assumption \ref{assump:test}.
By course of the same procedure at p.1648 in \cite{Z}
(note that Zograf's $2 \sqrt{D}$ in \cite{Z} is $\sqrt{D}$ in our notation),
we have
\begin{align*}
 & P^{Y}(m_1,m_2) \\
 & = 4 \sqrt{D}(A_{-1} \log Y + A_{0})
\RE \biggl\{ \int_{0}^{\infty} \! \! \int_{0}^{\infty}
     \Bigl[  \frac{(2+iu_1)}{|2+iu_1|} \Bigr]^{m_1}
     \Bigl[  \frac{(2+iu_2)}{|2+iu_2|} \Bigr]^{m_2}
\Phi(u_1^2,u_2^2) \, du_1 du_2 \biggr\} \\
& \quad + 4 \sqrt{D} A_{-1} 
\RE \biggl\{ \int_{0}^{\infty} \! \! \int_{0}^{\infty} 
     \log (u_1 u_2)
     \Bigl[  \frac{(2+iu_1)}{|2+iu_1|} \Bigr]^{m_1}
     \Bigl[  \frac{(2+iu_2)}{|2+iu_2|} \Bigr]^{m_2}
\Phi(u_1^2,u_2^2) \, du_1 du_2 \biggr\} \\
& \quad + o(1) \quad (Y \to \infty). 
\end{align*}
Here, $A_{-1},A_{0}$ are the coefficients of the Laurent
expansion of $\zeta_{K}(s)$, the Dedekind zeta function of $K$, 
at $s=1$. In particular, $A_{-1} = \frac{2 \log \varepsilon}{\sqrt{D}}$. 
Put 
\begin{align*}
 & P_0(m_1,m_2) \\
 & := 4  \sqrt{D}(A_{-1} \log Y + A_{0})
\RE \biggl\{ \int_{0}^{\infty} \! \! \int_{0}^{\infty}
     \Bigl[  \frac{(2+iu_1)}{|2+iu_1|} \Bigr]^{m_1}
     \Bigl[  \frac{(2+iu_2)}{|2+iu_2|} \Bigr]^{m_2}
\Phi(u_1^2,u_2^2) \, du_1 du_2 \biggr\} \\
& = \sqrt{D}(A_{-1} \log Y + A_{0}) \, g(0,0).  
\end{align*}
Here, the last equality is derived from Definition \ref{def:Qgh}.

For $j=1,2$, put
\[ P_{j}(m_1,m_2) :=  
\RE \biggl\{ \int_{0}^{\infty} \! \! \int_{0}^{\infty} 
     \log (u_j)
     \Bigl[  \frac{(2+iu_1)}{|2+iu_1|} \Bigr]^{m_1}
     \Bigl[  \frac{(2+iu_2)}{|2+iu_2|} \Bigr]^{m_2}
\Phi(u_1^2,u_2^2) \, du_1 du_2 \biggr\}.
\]
We note that
\[ P^{Y}(m_1,m_2) = P_0(m_1,m_2) +8 \log \ve \Bigl\{
 P_1(m_1,m_2)+P_2(m_1,m_2) \Bigr\} +o(1). 
\]
We calculate the case of $(m_1,m_2)=(0,m)$.
\begin{align*}
P_{2}(0,m) &=
\RE \biggl\{ \int_{0}^{\infty} \! \! \int_{0}^{\infty} 
     \log (u_2)
     \Bigl[  \frac{(2+iu_2)}{|2+iu_2|} \Bigr]^{m}
\Phi(u_1^2,u_2^2) \, du_1 du_2 \biggr\} \\
 &=\RE \biggl\{
\int_{0}^{\infty} \Phi_1(u_1^2) \, du_1
\int_{0}^{\infty} \log (u_2) 
\Bigl[  \frac{(2+iu_2)}{|2+iu_2|} \Bigr]^{m}
\Phi_2(u_2^2) \, du_2
\biggr\} \\
& = \frac{1}{2}g_{1}(0) 
\biggl\{
- \frac{1}{2}(\log 2 +C_{E}) \, g_2(0) + \frac{1}{8} h_2(0)
- \frac{1}{4 \pi} \int_{\R} h_2(r_2) \, \frac{\Gamma'}{\Gamma}
(1+ir_2) \, dr_2 \\
& \quad + \frac{1}{2} \int_{0}^{\infty} 
\frac{g_2(u_2)}{e^{u_2/2}-e^{-u_2/2}} 
\Bigl[ 1 - \cosh \frac{m}{2}u_2 \Bigr]
du_2
\biggr\}.
\end{align*}
We refer to pp.406 -- 411 in \cite{H2} for the last equality.
Thus, we obtain
\begin{align*}
P_{2}(0,m) & =
- \frac{1}{4}(\log 2 +C_{E}) \, g(0,0) + \frac{1}{8} \int_{0}^{\infty} g(0,u) \, du
- \frac{1}{16 \pi^2} \int \! \! \! \int_{\R^2} h(r_1,r_2) \, \frac{\Gamma'}{\Gamma}
(1+ir_2) \, dr_2 \\
& \quad + \frac{1}{4} \int_{0}^{\infty} 
\frac{g(0, u)}{e^{u/2}-e^{-u/2}} 
\Bigl[ 1 - \cosh \frac{m}{2}u \Bigr]
du.
\end{align*}
Similarly, we obtain
\[
P_{1}(0,m)  =
- \frac{1}{4}(\log 2 +C_{E}) \, g(0,0) + \frac{1}{8} \int_{0}^{\infty} g(u,0) \, du
- \frac{1}{16 \pi^2} \int \! \! \! \int_{\R^2} h(r_1,r_2) \, \frac{\Gamma'}{\Gamma}
(1+ir_1) \, dr_1.
\]
The proof is finished.
\end{proof}

\subsection{Type 2 hyperbolic contribution}
Let $(m_1,m_2) \in (2 \Z)^2$ and $Y>1$. 
We consider the type 2 hyperbolic contribution to the trace formula 
with the truncation parameter $Y$:
\[ H_2^{Y}(m_1,m_2) := \sum_{k=1}^{\infty} \sum_{\gamma_{k,\alpha} \in \Gamma_{\text{H2}}}
\int_{S^{Y}} k(z, \gamma_{k,\alpha} z) \, j_{\gamma_{k,\alpha}}(z) \, d \mu(z).
\]
Here, 
$\gamma_{k,\alpha} =  
\Bigl( 
\begin{array}{cc}
\varepsilon^{k} & \alpha \\
0 & \varepsilon^{-k}
\end{array} \Bigr)$ with  $k \in \N, \, \alpha \in 
\mathcal{O}_{K}$ are representatives of type 2 hyperbolic conjugacy classes 
of $\Gamma_k$, 
given in Lemma \ref{lem:hyp2}, 
\[ S^{Y} := \{(z_1,z_2) \in F_{\gamma_{k,\alpha}} \, | \, 
\IM(z_1) \IM(z_2) < Y, \,  \IM( \tau(z_1)) \IM( \tau'(z_2)) < Y \}
\] and 
$(\tau,\tau')$ is a element of $\Gamma_K$ such that
\[ (\tau,\tau') \Bigl( \alpha/(\ve^k - \ve^{-k}), \, \alpha'/((\ve')^k - (\ve')^{-k}) \Bigr)=(\infty,\infty)
.\]
We can show that (see \cite[p.1650]{Z})
\[   \int_{F_{\gamma_{k,\alpha}}^{Y}} k(z, \gamma_{k,\alpha} z) j_{\gamma_{k,\alpha}}(z) \, d \mu(z)  
    = \int_{S^{Y}} k(z, \gamma_{k,\alpha} z) j_{\gamma_{k, \alpha}}(z) 
\, d \mu(z) +o(1)   \quad  (Y \to \infty). 
\]

Put $\eta_k := \ve^{2k}+\ve^{-2k}-2$ and 
recall that $k_0=k_0(\gamma_{k,\alpha})$
was defined after Lemma \ref{lem:hyp2}. 
We can compute (see  \cite[pp.91--97]{E} or \cite[(4.5.1), p.1650]{Z})
\begin{align*}
& \int_{S^{Y}} k(z, \gamma_{k,\alpha} z) \, j_{\gamma}(z) \, d \mu(z) \\
= & \frac{k_0 \cdot 2 \log \varepsilon}{|N(\varepsilon^k-\varepsilon^{-k})|}
\int \! \! \! \int_{\R^2} \Phi(x_1^2+\eta_k,  x_2^2+\eta_k)  \\
& \quad \quad \times
\Bigl[  2 \log Y  - \log N(\Lambda)^2 + \log(x_1^2 + \eta_k)   + \log(x_2^2 + \eta_k)   \Bigr] \\
& \quad \quad \quad \times
\Big(  \frac{\sqrt{\eta_k +4} +i x_1}{\sqrt{\eta_k +4} -i x_1}  \Big)^{m_1/2}
\Big(  \frac{\sqrt{\eta_k +4} +i x_2}{\sqrt{\eta_k +4} -i x_2}  \Big)^{m_2/2}
d x_1 d x_2  \\
= & \frac{k_0 \cdot 2 \log \varepsilon}{|N(\varepsilon^k-\varepsilon^{-k})|}
\Bigl\{  R_0(m_1,m_2)  + R_1(m_1,m_2) + R_2(m_1, m_2) \Bigr\}.
\end{align*}

Here, we put
\begin{align*}
& R_0(m_1,m_2) := \int \! \! \! \int_{\R^2} \Phi(x_1^2+\eta_k,  x_2^2+\eta_k)
\Bigl[  2 \log Y  - \log N(\Lambda)^2 \Bigr] \\
& \quad \quad \times \Big(  \frac{\sqrt{\eta_k +4} +i x_1}{\sqrt{\eta_k +4} -i x_1}  \Big)^{m_1/2}
\Big(  \frac{\sqrt{\eta_k +4} +i x_2}{\sqrt{\eta_k +4} -i x_2}  \Big)^{m_2/2}
d x_1 d x_2,  \\
& R_j (m_1,m_2) := \int \! \! \! \int_{\R^2} \Phi(x_1^2+\eta_k,  x_2^2+\eta_k)
\log (x_j^2 +\eta_k) \\
& \quad \quad \times \Big(  \frac{\sqrt{\eta_k +4} +i x_1}{\sqrt{\eta_k +4} -i x_1}  \Big)^{m_1/2}
\Big(  \frac{\sqrt{\eta_k +4} +i x_2}{\sqrt{\eta_k +4} -i x_2}  \Big)^{m_2/2}
d x_1 d x_2  \quad (j=1,2),
\end{align*}
and $\Lambda \in \mathcal{O}_{K}$ such that
the ideal $(\Lambda)=(\alpha, \ve^{k}-\ve^{-k})$.

Firstly we note that
\begin{equation}
\begin{split}
& R_0(m_1,m_2) = \Bigl[  2 \log Y  - \log N(\Lambda)^2 \Bigr] Q(\eta_k, \eta_k)  \nonumber \\
& = 2 \Bigl(  2 \log Y  - \log N(\Lambda) \Bigr) g(2k \log \varepsilon, 2k \log \varepsilon)
\end{split}
\end{equation}
by using Proposition \ref{prop:phi} and the formula 
$g(u_1,u_2) = Q(e^{u_1}+e^{-u_1}-2, e^{u_2}+e^{-u_2}-2)$.

Hereafter let us compute the case of $(m_1,m_2)=(0,m)$.

\begin{proposition}
Let $m \in 2 \Z$. We have
\begin{equation}
\begin{split}
R_1(0,m) & = 2 \log (\varepsilon^k - \varepsilon^{-k} ) \, g(2k \log \varepsilon, 2k \log \varepsilon) 
\\
            & + \int_{2k \log \varepsilon}^{\infty} g(u, 2k \log \varepsilon) 
           \frac{\cosh (u/2)}{\sinh(u/2)  + \sinh(k \log \varepsilon)} \, du.  
\end{split}
\end{equation}
\end{proposition}

\begin{proof}
We recall that the test function $\Phi$ is written as 
$\Phi(x_1,x_2)=\Phi_1(x_1) \, \Phi_2(x_2)$ 
by Assumption \ref{assump:test}.
Therefore, 
\begin{align*}
& R_1(0,m) = \int_{\R} \Phi_1(x_1^2 + \eta_k) \log(x_1^2 + \eta_k) \, dx_1 
\int_{\R} \Phi_2(x_2^2 + \eta_k) 
\Big(  \frac{\sqrt{\eta_k +4} +i x_2}{\sqrt{\eta_k +4} -i x_2}  \Big)^{m/2} dx_2 \\
& = I_1(\eta_k) \cdot Q_2(\eta_k)
   = I_1(\eta_k) \cdot g_2(2k \log \varepsilon).
\end{align*}
Here, $I_1(\eta) = I_1(\eta_k)$ is given by
\begin{align*}
I_1(\eta) &= 
- \frac{2}{\pi}
\int_{0}^{\infty} \! \! \int_{-\infty}^{\infty} \log (x^2 +\eta) \, Q_1'(x^2+\eta+t^2) \, dt dx \\
& = - \frac{1}{\pi}
\int_{\eta}^{\infty}  \Bigl( \int_{-\sqrt{y-\eta}}^{\sqrt{y-\eta}}
\frac{\log(y-t^2)}{\sqrt{y-\eta-t^2}}  \, dt \Bigr) \, 
Q_1'(y) \, dy.
\end{align*}
The inner integral is evaluated as
\begin{align*}
& \int_{-\sqrt{y-\eta}}^{\sqrt{y-\eta}}
\frac{\log(y-t^2)}{\sqrt{y-\eta-t^2}}  \, dt 
= 2 \int_{0}^{\pi/2} \log (y \cos^2 \theta + \eta \sin^2 \theta) \, d \theta \\
&  =   2 \int_{0}^{\pi/2} \log \eta \, d \theta
      + 2 \int_{0}^{\pi/2} \log \Bigl( 1 + \frac{y-\eta}{\eta} \cos^2 \theta \Bigr) \, d \theta \\
& = 2 \pi \log \sqrt{\eta} + 2 \pi \log \frac{1+\sqrt{(y-\eta)/\eta+1}}{2} \\
& = 2 \pi \log \Bigl(  \frac{\sqrt{\eta}+\sqrt{y}}{2} \Bigr).
\end{align*}
Here, we used the formula: (see \cite[4.399]{GR})
\[ \int_{0}^{\pi/2} \log (1+ a \sin^2 x) \, dx 
  = \int_{0}^{\pi/2} \log (1+ a \cos^2 x) \, dx
  = \pi \log \Bigl( \frac{1+\sqrt{1+a}}{2} \Bigr)
\quad \mbox{for $a >-1$.} 
\]
Thus, we have
\begin{align*}
I_1(\eta) &= -2 \int_{\eta}^{\infty}  
\Bigl\{   \log(\sqrt{y}+\sqrt{\eta})  - \log 2 \Bigr\}  
Q_1'(y) \, dy \\
&= 2 \log 2 \int_{2k \log \varepsilon}^{\infty}  g_1'(u) \, du
   -2  \int_{2k \log \varepsilon}^{\infty}  
   \log(e^{u/2} - e^{-u/2}   + \varepsilon^k  -\varepsilon^{-k})
\, g_1'(u) \, du \\
& = - 2 \log 2 \, g_1(2k \log \ve)  +2 g_1(2k \log \ve) \log(2(\ve^k -\ve^{-k})) \\
& \quad  + \int_{2k \log \ve}^{\infty} g_1(u) \frac{\cosh(u/2)}{\sinh(u/2)+\sinh(k \log \ve)} \, du \\
& = 2 \log(\ve^k -\ve^{-k}) \, g_1(2k \log \ve)
+ \int_{2k \log \ve}^{\infty} g_1(u) \frac{\cosh(u/2)}{\sinh(u/2)+\sinh(k \log \ve)} \, du. 
\end{align*}
The rest is clear.
\end{proof}

\begin{proposition}
Let $m \in 2 \Z$. We have
\begin{equation}
\begin{split}
R_2(0,m) & = 2 \log (\varepsilon^k - \varepsilon^{-k} ) \, g(2k \log \varepsilon, 2k \log \varepsilon) 
+ \int_{2k \log \varepsilon}^{\infty} 
       \Bigl[        
     \frac{\cosh (u/2)}{\sinh(u/2)+ \sinh(k \log \varepsilon)} \\ 
& \quad \quad  + \frac{1-\cosh(m(u/2 - k \log \varepsilon))}{\sinh(u/2 - k \log \varepsilon)}
     \Bigr]
g(2k \log \varepsilon,u) \, du.
\end{split}  
\end{equation}
\end{proposition}

\begin{proof}
We recall that the test function $\Phi$ is written as 
$\Phi(x_1,x_2)=\Phi_1(x_1) \, \Phi_2(x_2)$ 
with real valued $\Phi_1$ and $\Phi_2$ 
by Assumption \ref{assump:test}.
Therefore, 
\begin{align*}
& R_2(0,m) = \int_{\R} \Phi_1(x_1^2 + \eta_k) \, dx_1 
\int_{\R} \Phi_2(x_2^2 + \eta_k) \log(x_2^2 + \eta_k)
\Big(  \frac{\sqrt{\eta_k +4} +i x_2}{\sqrt{\eta_k +4} -i x_2}  \Big)^{m/2} dx_2 \\
& = Q_1(\eta_k) \cdot I_2(\eta_k)
   = g_1(2k \log \varepsilon) \cdot I_2(\eta_k).
\end{align*}
Here, $I_2(\eta) = I_2(\eta_k)$ is given by
\begin{align*}
I_2(\eta) &= 
- \frac{2}{\pi} \RE \biggl[ 
\int_{0}^{\infty} \! \! \int_{-\infty}^{\infty} Q_2'(x_2^2+\eta+t^2) \, \log (x_2^2 +\eta)\\
& \times \biggl(  \frac{\sqrt{x_2^2+\eta+4+t^2}-t}{\sqrt{x_2^2+\eta+4+t^2}+t} \biggr)^{m/2}
\biggl(  \frac{\sqrt{\eta+4}+ix}{\sqrt{\eta+4}-ix} \biggr)^{m/2}
\, dt dx_2 \biggr]\\
& =  - \frac{1}{\pi} \RE \biggl[ 
\int_{\eta}^{\infty}  
\biggl( 
\int_{-\sqrt{y-\eta}}^{\sqrt{y-\eta}}
\frac{\log(y-\xi^2)}{\sqrt{y-\eta-\xi^2}}
\biggl(  \frac{\sqrt{\eta+4} + i \sqrt{y-\eta-\xi^2}}{\sqrt{y+4}+\xi} \biggr)^{m}
\, d \xi
\biggr)
Q_2'(y) \, dy \biggr],
\end{align*}
by changing the variables
$y=x_2^2+\eta+t^2$ and $\xi=t$.
Next changing the variable 
$\xi = \sqrt{y-\xi} \sin \varphi$, we have
\begin{align*}
& I_2(\eta) \\
& =  
- \frac{1}{\pi} \RE \biggl[ 
\int_{\eta}^{\infty}  \! \! 
\int_{- \pi/2}^{\pi/2}
\log \bigl( (y-\eta) \cos^2 \varphi  + \eta \bigr)
\biggl(  \frac{\sqrt{\eta+4} + i \sqrt{y-\eta} \cos \varphi}{\sqrt{y+4}+ \sqrt{y-\eta} \sin \varphi} \biggr)^{m}
\, d \varphi \,
Q_2'(y) \, dy \biggr] \\
& = 
- \frac{1}{\pi} \RE \biggl[ 
\int_{\eta}^{\infty}  \! \! 
\int_{- \pi/2}^{\pi/2}
\log \bigl( (y-\eta) \cos^2 \varphi  + \eta \bigr)
\biggl(  \frac{\sinh w + i \cos \varphi}{ \cosh w + \sin \varphi} \biggr)^{m}
\, d \varphi \, 
Q_2'(y) \, dy \biggr] 
\end{align*}
with 
\begin{equation}
\cosh w =  \sqrt{\frac{y+4}{y-\xi}}, \quad 
\sinh w =  \sqrt{\frac{\eta+4}{y-\xi}}.
\end{equation}

Let us consider the integral 
\[ J(y) :=  \RE \biggl[ 
\int_{- \pi/2}^{\pi/2}
\log \bigl( (y-\eta) \cos^2 \varphi  + \eta \bigr)
\biggl(  \frac{\sinh w + i \cos \varphi}{ \cosh w + \sin \varphi} \biggr)^{m}
\, d \varphi \biggr].
\]

Then we see that

\begin{align*}
I_2 &= -\frac{1}{\pi} \int_{\eta}^{\infty} J(y) Q_2'(y) \, dy 
= - \frac{1}{\pi} [J(y)Q_2(y)]_{\eta}^{\infty}  +\frac{1}{\pi} \int_{\eta}^{\infty}  J'(y) Q_2(y) \, dy \\
& =  \frac{1}{\pi} Q_2(\eta) \int_{- \pi/2}^{\pi/2} \log \eta \, d \varphi 
 + \frac{1}{\pi} \int_{\eta}^{\infty}  J'(y) Q_2(y) \, dy \\
& =  2 \log (\ve^k - \ve^{-k}) \, g_2(2k \log \ve) 
+\frac{1}{\pi} \int_{\eta}^{\infty}  J'(y) Q_2(y) \, dy. 
\end{align*}

Let us consider the function $f(z)$ defined by
\[ f(z) := \biggl(  \frac{i e^{w}-z}{i e^{w}+z} \biggr)^m  \, 
    \frac{\Log \Bigl( (\xi-\zeta) \bigl( \frac{z+z^{-1}}{2} \bigr)^2 +\eta \Bigr)}{z}
\]
to evaluate the derivative of $J(\xi)$.
Here, $\Log(z)$ is the principal value logarithm whose imaginary part lies in 
$(-\pi, \pi]$.

Let $\epsilon, \,  \delta > 0$ be two sufficiently small real numbers, 
and define the closed curve $C$ in the complex plane, which is
made up of two semi-circular arcs starting from $\varphi = -\frac{\pi}{2}+\delta$ to
$\varphi=\frac{\pi}{2}-\delta$ of the radii $1$ and $\epsilon$, 
and besides they are 
joined along by the straight lines $\varphi=\pm(\frac{\pi}{2} - \delta)$. 

Considering the counterclockwise contour integral of $f(z)$ along the curve $C$,
by Cauchy's integral theorem, we have

\begin{align}
&  \int_{-\pi/2}^{\pi/2}  \biggl( \frac{\sinh w + i \cos \varphi}{\cosh w + \sin \varphi } \biggr)^m
    \log \Bigl( (\xi - \eta) \cos^2 \varphi +\eta \Bigr) \, i d \varphi  \label{eq:j1} \\
& + \int_{i}^{i \epsilon}  \biggl( \frac{i e^w -z}{i e^w +z} \biggr)^m
    \Log \Bigl( (\xi - \eta) \Bigl( \frac{z+z^{-1}}{2} \Bigr)^2 +\eta \Bigr) \, \frac{dz}{z}  \label{eq:j2} \\
& +  \int_{- \pi/2}^{\pi/2}  \biggl( \frac{i e^w - \epsilon e^{i \varphi} }{i e^w + \epsilon e^{i \varphi}} \biggr)^m
    \Log \Bigl( (\xi - \eta) \Bigl( \frac{ \epsilon e^{i \varphi} + \epsilon^{-1} e^{- i \varphi} }{2} \Bigr)^2 +\eta \Bigr) 
\, i d \varphi  \label{eq:j3} \\
& +  \int_{-i \epsilon}^{-i }  \biggl( \frac{i e^w -z}{i e^w +z} \biggr)^m
    \Log \Bigl( (\xi - \eta) \Bigl( \frac{z+z^{-1}}{2} \Bigr)^2 +\eta \Bigr) \, \frac{dz}{z} \label{eq:j4} \\
& = 0 \nonumber
\end{align}

Put $\epsilon_0 :=  \sqrt{\frac{\xi}{\xi-\eta}} -\sqrt{\frac{\eta}{\xi-\eta}}$, then
we see that $\epsilon_0$ satisfies $(\xi-\eta)((\epsilon_0^{-1}-\epsilon_0)/2)^2=\eta$
and  (\ref{eq:j2}) and (\ref{eq:j4}) are written as follows.
\begin{align*}
(\ref{eq:j2}) 
=& \int_{\epsilon_0}^{\epsilon}  \biggl( \frac{e^w -y}{e^w +y} \biggr)^m
\biggl\{ \log \Bigl( (\xi - \eta) \Bigl( \frac{y^{-1}-y}{2}  \Bigr)^2  
- \eta \Bigr)  - i \pi  \biggr\} \frac{dy}{y} \\
& + \int_{1}^{\epsilon_0} \biggl( \frac{e^w -y}{e^w +y} \biggr)^m
\biggl\{ \log \Bigl( \eta -  (\xi - \eta) \Bigl( \frac{y^{-1}-y}{2}  \Bigr)^2  
\Bigr)  \biggr\} \frac{dy}{y}, 
\end{align*}
and
\begin{align*}
(\ref{eq:j4}) 
=& \int_{\epsilon_0}^{1} \biggl( \frac{e^w +y}{e^w -y} \biggr)^m
\biggl\{ \log \Bigl( \eta -  (\xi - \eta) \Bigl( \frac{y^{-1}-y}{2}  \Bigr)^2  
\Bigr)  \biggr\} \frac{dy}{y} \\
&+
\int_{\epsilon}^{\epsilon_0}  \biggl( \frac{e^w +y}{e^w -y} \biggr)^m
\biggl\{ \log \Bigl( (\xi - \eta) \Bigl( \frac{y^{-1}-y}{2}  \Bigr)^2  
- \eta \Bigr)  + i \pi  \biggr\} \frac{dy}{y}.
\end{align*}

While, (\ref{eq:j3}) is evaluated as
\[ 
(\ref{eq:j3})
=
\int_{-\pi/2}^{\pi/2}
\Bigl[ 1+O(\epsilon) \Bigr]
\Bigl[ \log(1/\epsilon^2) + \log \Bigl(  \frac{\xi-\eta}{4} \Bigr) 
  -2 i \varphi +O(\epsilon^2) \Bigr] \, i d \varphi
.\]

Take the real part of 
\[ (-i) \times \Bigl\{ (\ref{eq:j1})  +  (\ref{eq:j2})  +  (\ref{eq:j3}) +  (\ref{eq:j4})  \Bigr\}, 
\] 
and we obtain, 
\begin{align*}
& J(\xi) + \pi  \int_{\epsilon}^{\epsilon_0}   \biggl( \frac{e^w-y}{e^w+y}  \biggr)^m \frac{dy}{y}  
+ \pi  \int_{\epsilon}^{\epsilon_0}   \biggl( \frac{e^w+y}{e^w-y}  \biggr)^m \frac{dy}{y} 
- \pi \biggl\{ \log \Bigl( \frac{1}{\epsilon^2} \Bigr)  
               + \log \Bigl( \frac{\xi-\eta}{4} \Bigr)  \biggr\} \\
&=  O \biggl( \epsilon \log  \Bigl( \frac{1}{\epsilon^2} \Bigr) \biggr).             
\end{align*}

Therefore, we can rewrite the above formula as follows.
\begin{align*}
& J(\xi) + \pi  \int_{\epsilon}^{\epsilon_0}  \biggl[ \biggl( \frac{e^w-y}{e^w+y}  \biggr)^m -1 \biggr]  \frac{dy}{y}  
+ \pi  \int_{\epsilon}^{\epsilon_0}  \biggl[ \biggl( \frac{e^w+y}{e^w-y}  \biggr)^m  -1 \biggr]  \frac{dy}{y} \\
& \quad - \pi \log \Bigl( \frac{\xi-\eta}{4} \Bigr)  + 2 \pi \log \epsilon_0 \\
&=  O \biggl( \epsilon \log  \Bigl( \frac{1}{\epsilon^2} \Bigr) \biggr).             
\end{align*}

Letting $\epsilon \to +0$, 
we have an expression for $J(\xi)$.
Changing the variable $y=e^{-u}$ in the integral, we have
\begin{align*}
J(\xi) =& - \pi \int_{u_0}^{\infty} \biggl[ \biggl( \frac{e^w-e^{-u}}{e^w+e^{-u}} \biggr)^m  -1 \biggr] du
- \pi \int_{u_0}^{\infty} \biggl[ \biggl( \frac{e^w+e^{-u}}{e^w-e^{-u}} \biggr)^m  -1 \biggr] du \\
& + 2 \pi \log \Bigl( \frac{\sqrt{\xi}+\sqrt{\eta}}{2} \Bigr)
\end{align*}
with $u_0 := \log \epsilon_0^{-1}$.
Then we obtain an explicit formula for the derivative of $J(\xi)$.

\begin{align*}
\frac{d J(\xi)}{d \xi} &=  
\frac{\pi}{\sqrt{\xi}(\sqrt{\xi}+\sqrt{\eta})}  
- 2 \pi  \biggl(  \frac{\partial w}{\partial \xi}  + \frac{\partial u_0}{\partial \xi}  \biggr)  \\
& \quad + \pi \biggl(  \frac{\partial w}{\partial \xi}  + \frac{\partial u_0}{\partial \xi}  \biggr)
\biggl[  \biggl( \frac{e^w-e^{-u_0}}{e^w+e^{-u_0}} \biggr)^m  + \biggl( \frac{e^w+e^{-u_0}}{e^w-e^{-u_0}} \biggr)^m  \biggr].
\end{align*}

By noting that
\begin{align*}
& \frac{e^w-e^{-u_0}}{e^w+e^{-u_0}} 
= \frac{\sqrt{\xi+4}+\sqrt{\eta+4}-\sqrt{\xi}+\sqrt{\eta}}{\sqrt{\xi+4}+\sqrt{\eta+4}+\sqrt{\xi}-\sqrt{\eta}}
= \frac{\sqrt{\eta+4}+\sqrt{\eta}}{\sqrt{\xi+4}+\sqrt{\xi}} 
= \frac{\ve^{k}}{e^{u/2}}, \\
& \frac{\partial w}{\partial \xi}
= - \frac{1}{2(\xi-\eta)} \frac{\sqrt{\eta+4}}{\sqrt{\xi+4}}, \quad \quad 
\frac{\partial u_0}{\partial \xi}
= - \frac{1}{2(\xi-\eta)} \frac{\sqrt{\eta}}{\sqrt{\xi}}, \\
&  \frac{\partial w}{\partial \xi} + \frac{\partial u_0}{\partial \xi}
= - \frac{1}{\xi-\eta} \frac{\ve^k e^{u/2}  -\ve^{-k}e^{-u/2}}{e^u-e^{-u}},
\end{align*}
with $\xi = e^u+e^{-u}-2$.

We obtain
\begin{equation*} \label{eq:dj} \\
\begin{split}
\frac{d J(\xi)}{d \xi} &=  
\frac{\pi}{\sqrt{\xi}(\sqrt{\xi}+\sqrt{\eta})}  
+ \frac{2 \pi}{\xi-\eta} \frac{\ve^k e^{u/2} - \ve^{-k}e^{-u/2}}{e^u-e^{-u}}  \\
& \quad - \frac{\pi}{\xi-\eta} \frac{\ve^k e^{u/2} - \ve^{-k}e^{-u/2}}{e^u-e^{-u}} 
\Bigl[ (e^{u/2} \ve^{-k}) ^m + (e^{u/2} \ve^{-k})^{-m}  \Bigr]. 
\end{split}
\end{equation*}

Then, we have
\begin{align*}
& J'(e^u+e^{-u}-2) \cdot (e^u-e^{-u}) \\
& = \frac{\pi(e^{u/2}+e^{-u/2})}{e^{u/2}-e^{-u/2}+\ve^k-\ve^{-k}}
   +\frac{\pi (\ve^k e^{u/2} - \ve^{-k} e^{-u/2} )}{\xi-\eta}
   \Bigl\{ 2 -  \bigl( e^{u/2} \ve^{-k} \bigr)^m   - \bigl(e^{u/2} \ve^{-k} \bigr)^{-m}  \Bigr\} \\
& = \frac{\pi \cosh(u/2)}{\sinh(u/2)+\sinh(k \log \ve)}
   +\frac{\pi}{\ve^{-k}e^{u/2}  - \ve^ke^{-u/2}} 
\Bigl\{ 2 -  \bigl( e^{u/2} \ve^{-k} \bigr)^m   - \bigl(e^{u/2} \ve^{-k} \bigr)^{-m}  \Bigr\} \\
& = \frac{\pi \cosh(u/2)}{\sinh(u/2)+\sinh(k \log \ve)}
   +\frac{\pi}{ \sinh(u/2-k \log \ve)} 
\Bigl\{ 1 -  \cosh \bigl( m(u/2 -k \log \ve) \bigr) \Bigr\}. 
\end{align*}

Substituting the above equality into the following, the proof is completed.
\begin{align*}
I_2 &= 2 \log (\ve^k - \ve^{-k}) \, g_2(2k \log \ve)
       + \frac{1}{\pi} \int_{\eta}^{\infty} J'(\xi) Q_2(\xi) \, d \xi \\
&= 2 \log (\ve^k - \ve^{-k}) \, g_2(2k \log \ve)
       + \frac{1}{\pi} \int_{2k \log \ve}^{\infty} J'(e^u+e^{-u}-2) 
\, g_2(u) \, (e^u-e^{-u}) \, du. 
\end{align*}

\end{proof}

Putting together with the results in this subsection,
we obtain
\begin{proposition} \label{prop:h2y}
For $m \in 2 \Z$ and $Y>1$, we have
\begin{align*}
  H_2^{Y}(0, m) 
& = 4 \log \varepsilon \, \log Y 
\sum_{k=1}^{\infty} g(2k \log \varepsilon, 2k \log \varepsilon) \\
&  - 4 \log \varepsilon \sum_{k=1}^{\infty} 
\sum_{\gamma_{k,\alpha} \in \Gamma_{\mathrm{H2}}}
\frac{k_0(\gamma_{k,\alpha}) \log(N(\alpha, \varepsilon^{k}-\varepsilon^{-k}))}{|N(\varepsilon^{k}-\varepsilon^{-k})|} g(2k \log \varepsilon, 2k \log \varepsilon) \\
&  + 4 \log \varepsilon \sum_{k=1}^{\infty}
\log(\varepsilon^k -\varepsilon^{-k} ) \, g(2k \log \varepsilon, 2k \log \varepsilon) \\
&  + 2 \log \varepsilon \sum_{k=1}^{\infty}
\int_{2k \log \varepsilon}^{\infty} 
\Bigl[ g(u,2k \log \varepsilon)+g(2k \log \varepsilon,u) \Bigr] 
\frac{\cosh(u/2)}{\sinh(u/2)+\sinh(k \log \varepsilon)} \, du \\
&  + 2 \log \varepsilon \sum_{k=1}^{\infty}
\int_{2k \log \varepsilon}^{\infty} g(2k \log \varepsilon,u) 
\frac{1 - \cosh \bigl( m( u/2 - k \log \varepsilon) \bigr)}{\sinh(u/2 - k \log \varepsilon)} \, du
\\
& + o(1) \quad (Y \to \infty).
\end{align*}
\end{proposition}
\begin{proof}
By noting the fact (see  \cite[Proposition 3.3, p.97]{E} or \cite[p.1650]{Z}) 
\[ \sum_{ \gamma \in \Gamma_\text{H2},  \, N(\gamma) = \ve^{2k}} k_0(\gamma)
 = |N(\ve^{k}-\ve^{-k})|
,\]
the rest is clear.
\end{proof}

\subsection{Contribution from Eisenstein series}
Let $m \in 2\Z$ and $Y>1$. 
Define the contribution from the Eisenstein series 
with the truncation parameter $Y$ by
\[ EI^{Y}(0,m) := \int_{F^{Y}} H_{\Gamma}(z,z) \, d \mu(z)
.\]

By using the Maass-Selberg relation (Theorem \ref{th:maass} ), we obtain
\begin{proposition} \label{prop:eiy}
For $m \in 2 \Z$, we have
\begin{align*}
EI^{Y}(0,m) =&  
2 \log \varepsilon \, \log Y 
\sum_{k \in \Z}  g(2k \log \varepsilon, 2k \log \varepsilon) 
\\
& - \frac{1}{4 \pi} \sum_{k \in \Z} \int_{\R} 
h \Bigl( r+\frac{\pi k}{2 \log \varepsilon},r-\frac{\pi k}{2 \log \varepsilon} \Bigr) \,
\frac{\varphi'_{(0,m)}}{\varphi_{(0,m)}} \Bigl( \frac{1}{2}+ir,k \Bigr) \, dr \\
& + \frac{1}{4} h(0,0) \, \varphi_{(0,m)} \Bigl( \frac{1}{2},0 \Bigr) + o(1) \quad (Y \to \infty).
\end{align*}
\end{proposition}
\begin{proof}
By definition of the kernel function $H_{\Gamma}(z,z)$, we can check that
\begin{align*}
& \int_{F^{Y}}  H_{\Gamma}(z,z) \, d \mu(z)  \\
&= \frac{1}{8 \pi \sqrt{D} \log \ve} \int_{F} d \mu(z)
\biggl[  \sum_{k \in \Z}   \int_{\R}    
h \Bigl( r+\frac{\pi k}{2 \log \varepsilon}, r-\frac{\pi k}{2 \log \varepsilon} \Bigr) 
\biggl| E^{Y}_{(0,m)}  \Bigl( z,\frac{1}{2}+ir, k \Bigr)   \biggr|^2
dr
\biggr] \\
& \quad +o(1)  \quad  (Y \to \infty).
\end{align*}
Next we use the following special case of Theorem \ref{th:maass}:
\begin{align*}
& \int_{F} \biggl| E^{Y}_{(0,m)}  \Bigl( z,\frac{1}{2}+ir, k \Bigr)   \biggr|^2
\, d \mu(z) \\
&    = 
2 \sqrt{D} \log \ve
\biggl[ 2 \log Y 
-  \frac{\varphi'_{(0,m)}}{\varphi_{(0,m)}} \Bigl( \frac{1}{2}+ir,k \Bigr) \\
& \quad \quad \quad \quad \quad \quad 
+\delta_{0,k}  \frac{\varphi_{(0,m)}(\frac{1}{2}-ir, 0) \, Y^{2 ir} 
 - \varphi_{(0,m)}(\frac{1}{2}+ir, 0) \, Y^{-2ir} }{2 i r}
\biggr].
\end{align*}
Finally we obtain the desired formula
as in the proof of Proposition 1.1 in \cite[p.85]{E}.
\end{proof}

\subsection{Cancellation of the $\log Y$ terms}
Let us complete the proof of Theorem \ref{th:trf}.
Let $m \in 2 \Z$. 
By Propositions \ref{prop:py}, \ref{prop:h2y} and \ref{prop:eiy}, 
the $\log Y$ terms are canceled out  and
we have 
\begin{equation} \label{eq:ph2sc}
\begin{split}
& \lim_{Y \to \infty}
\biggl\{ 
\sum_{\gamma \in \Gamma_\text{P} \cup \Gamma_\text{H2}}
\int_{F_{\gamma}^{Y}} k(z,\gamma z) \, j_{\gamma}(z) \, d \mu(z)
-\int_{F^{Y}} H_{\Gamma}(z,z) \, d \mu(z) 
\biggr\} \\
& = \lim_{Y \to \infty} 
\bigg\{ P^{Y}(0,m) + H_2^{Y}(0,m) -EI^Y(0,m) + o(1)  \biggr\} \\
& = P^Y(0,m) \Big|_{Y=1} + H_2^Y(0,m) \Big|_{Y=1} - EI^{Y}(0,m) \Big|_{Y=1} \\
& =: P(0,m) + H_2(0,m) + SC(0,m). 
\end{split}
\end{equation}
We see that $P(0,m)$, $H_2(0,m)$ and $SC(0,m)$ are identified with 
$\mathbf{II}_a(h)$, $\mathbf{II}_b(h)$ and $\mathbf{III}(h)$
in Theorem \ref{th:trf} respectively.
The series and integrals appearing in these terms are absolutely 
convergent by the assumption on the test functions $h$ in this theorem.
Thus we complete the proof.  

%%%%%%%%%%%%%%%%%%%%%%%%%%%%%%%%%%%%%%%%%%%%%%%%%%%%%%%%%%%%%%%%%%%%%%%%%%%
%%%%%%%%%%%%%%%%%%%%%%%%%%%%%%%%%%%%%%%%%%%%%%%%%%%%%%%%%%%%%%%%%%%%%%%%%%%
\section{Differences of the Selberg trace formula 
for Hilbert modular surfaces} 
%%%%%%%%%%%%%%%%%%%%%%%%%%%%%%%%%%%%%%%%%%%%%%%%%%%%%%%%%%%%%%%%%%%%%%%%%%%%
%%%%%%%%%%%%%%%%%%%%%%%%%%%%%%%%%%%%%%%%%%%%%%%%%%%%%%%%%%%%%%%%%%%%%%%%%%%%

\subsection{Differences of the Selberg trace formula}

Let $m \in 2 \Z$. We introduce Maass operators 
$\Lambda_{m}^{(2)}$ and $K_{m}^{(2)}$, which play important roles in
considering the ``differences" of the Selberg trace formulas.
We refer to \cite[Proposition 5.13, p.381]{H2} and \cite[pp.305--307]{Ro}
for basic properties of these Maass operators.

Firstly we consider the following ``weight down" Maass operator
\[ \Lambda_{m}^{(2)} := iy_2 \frac{\partial}{\partial x_2} - y_2 \frac{\partial}{\partial y_2} +\frac{m}{2} \colon
L^2_\text{dis}(\Gamma_{K} \backslash \bH^2 
\, ; \, (0,m)) \to L^2_\text{dis}(\Gamma_{K} \backslash \bH^2 
\, ; \, (0,m-2)).\]
Recall that
\[ \Ker(\Lambda_m^{(2)}) =
\biggl\{  f \in  L^2_\text{dis} \Bigl( \Gamma_{K} \backslash \bH^2 
\, ; \, (0,m) \Bigr)  \, \Bigl| \,  \Delta_{m}^{(2)} f  = \frac{m}{2} \Bigl( 1-\frac{m}{2} \Bigr) \, f  
\biggr\}, \]
i.e. $\lambda^{(2)}=\frac{m}{2}(1-\frac{m}{2})$-eigenspace.

Let $\{\frac{1}{4}+\rho_j(m)^2 \}_{j=0}^{\infty}$ be the set of eigenvalues of 
$\Delta_{0}^{(1)}$ acting on $\Ker(\Lambda_m^{(2)})$, 
then we have a direct sum decomposition into eigenspaces of the Laplacians
\begin{equation} \label{eq:ker1}
\Ker(\Lambda_m^{(2)}) = 
\bigoplus_{j=0}^{\infty}
L^2_\text{dis} \Bigl( \Gamma_{K} \backslash \bH^2 
\, ; \, \Bigl( \frac{1}{4}+\rho_j(m)^2, \, \frac{m}{2} \bigl( 1-\frac{m}{2} \bigr) \Bigr),(0,m) \Bigr)
.
\end{equation}

Secondly we consider the following ``weight up" Maass operator
\[ K_{m-2}^{(2)} := iy_2 \frac{\partial}{\partial x_2} + y_2 \frac{\partial}{\partial y_2} +\frac{m-2}{2} \colon
L^2_\text{dis}(\Gamma_{K} \backslash \bH^2 
\, ; \, (0,m-2)) \to L^2_\text{dis}(\Gamma_{K} \backslash \bH^2 
\, ; \, (0,m)).\]
Recall that
\[ 
\Ker(K_{m-2}^{(2)}) =
\biggl\{  f \in  L^2_\text{dis} \Bigl( \Gamma_{K} \backslash \bH^2 
\, ; \, (0,m-2) \Bigr)  \, \Bigl| \,  \Delta_{m}^{(2)} f  = \frac{m}{2} \Bigl( 1-\frac{m}{2} \Bigr) \, f  
\biggr\}, 
\]
i.e. $\lambda^{(2)}=\frac{m}{2}(1-\frac{m}{2})$-eigenspace.

Let $\{\frac{1}{4}+\mu_j(m-2)^2 \}_{j=0}^{\infty}$ be the set of eigenvalues of 
$\Delta_{0}^{(1)}$ acting on $\Ker(K_{m-2}^{(2)})$, 
then we have a direct sum decomposition into eigenspaces of the Laplacian
\begin{equation} \label{eq:ker2}
\Ker(K_{m-2}^{(2)}) = 
\bigoplus_{j=0}^{\infty}
L^2_\text{dis} \Bigl( \Gamma_{K} \backslash \bH^2 
\, ; \, \Bigl( \frac{1}{4}+\mu_j(m-2)^2, \, \frac{m}{2} \bigl( 1-\frac{m}{2} \bigr) \Bigr),(0,m-2) \Bigr)
.
\end{equation}

By considering the two kernel spaces (\ref{eq:ker1}) and (\ref{eq:ker2}), 
we {\it subtract} the Selberg trace formula for $L^2(\Gamma_{K} \backslash \bH^2 
\, ; \, (0,m-2))$ from the one associated 
with $L^2(\Gamma_{K} \backslash \bH^2 
\, ; \, (0,m))$. Then we obtain (we give a proof in the next subsection)

\begin{theorem}[Differences of STF 
for 
$L^2(\Gamma_{K} \backslash \bH^2 
\, ; \, (0,m)) - L^2(\Gamma_{K} \backslash \bH^2 
\, ; \, (0,m-2))$] \label{th:dtrf}
Let $m \in 2 \Z$. We have
\begin{align*}
& \sum_{j=0}^{\infty} 
h_1 \Bigl( \rho_j(m) \Bigr) \,  h_2(\tfrac{i(m-1)}{2}) 
 - \sum_{j=0}^{\infty} 
h_1 \Bigl( \mu_j(m-2) \Bigr) \,  h_2(\tfrac{i(m-1)}{2}) 
% \: - \delta_{m,2} \, h_1(\tfrac{i}{2}) \, h_2(\tfrac{i}{2}) 
\\
&= \ (m-1) h_2(\tfrac{i(m-1)}{2}) 
\frac{\mathrm{vol}(\Gamma_K \backslash \bH^2)}{16 \pi^2}
\int_{-\infty}^{\infty}
  r_1 h_1(r_1) \tanh (\pi r_1)  \, dr_1 \\
& \quad + \sum_{R(\theta_1,\theta_2) \in \Gamma_{\mathrm{E}}} 
\frac{ - e^{-i \theta_1 + i(m-1) \theta_2} }{8 \nu_R \sin \theta_1 \sin \theta_2}
\, h_2(\tfrac{i(m-1)}{2})
\int_{\R} g_1(u_1) \, e^{-u_1/2} 
\Bigl[
\frac{e^{u_1}-e^{2 i \theta_1}}{\cosh u_1 - \cos 2 \theta_1} \Bigr] du_1
\\
& \quad 
 + \sum_{(\gamma,\omega) \in \Gamma_{\mathrm{HE}}} 
\frac{\log N(\gamma_0)}{N(\gamma)^{1/2}-N(\gamma)^{-1/2}} g_1(\log N(\gamma)) 
\, \frac{i e^{i(m-1)\omega}}{2 \sin \omega} h_2(\tfrac{i(m-1)}{2})
\\
& \quad - \sgn (m-1)  \log \varepsilon \, g_1(0) \, h_2(\tfrac{i(m-1)}{2})
\\
& \quad 
- 2  \sgn(m-1)  \log \varepsilon \, h_2(\tfrac{i(m-1)}{2}) 
\sum_{k=1}^{\infty} g_1(2k \log \varepsilon) \, \varepsilon^{-k|m-1|}.
\end{align*}
\end{theorem} 

\subsection{Proof of the differences of the Selberg trace formula}
We prove Theorem \ref{th:dtrf} in this subsection. 
(Basic strategy is the same as the case of the trace formulas for $\mathrm{PSL}(2,\R)$, 
see \cite[pp.481--485]{H2}).
\\
$\bullet$ Spectral side: \\
By (\ref{eq:ker1}) and (\ref{eq:ker2}), the difference between
the spectral sides of $L^2(\Gamma_{K} \backslash \bH^2 
\, ; \, (0,m))$ and $L^2(\Gamma_{K} \backslash \bH^2 
\, ; \, (0,m-2))$ is given by
\[ 
\sum_{j=0}^{\infty} 
h_1 \Bigl( \rho_j(m) \Bigr) \,  h_2(\tfrac{i(m-1)}{2}) 
 - \sum_{j=0}^{\infty} 
h_1 \Bigl( \mu_j(m-2) \Bigr) \,  h_2(\tfrac{i(m-1)}{2}) 
.\]
$\bullet$ Identity term: \\
Put $\overline{I(m)}:= I(0,m) - I(0,m-2)$. Here, $I(m_1,m_2)$ is defined in (\ref{term:id}).
Then, we have (see pp.396--397 in \cite{H2})
\begin{align*}
& \overline{I(m)} \\
& =  \frac{ \mathrm{vol}(\Gamma_K \backslash \bH^2)}{4 \pi^2} \int_{\R} \! \int_{\R} 
\frac{g_1'(u_1) \, g_2'(u_2)}{(e^{u_1/2} - e^{-u_1/2}) (e^{u_2/2} - e^{-u_2/2})}
\bigl\{ e^{-mu_2/2}  - e^{-(m-2)u_2/2} \bigr\}  \, du_1 du_2 \\
& = - \frac{ \mathrm{vol}(\Gamma_K \backslash \bH^2)}{4 \pi^2} \int_{\R}
  \frac{g_1'(u_1)}{e^{u_1/2} - e^{-u_1/2}} \, du_1 
  \int_{\R}  g_2'(u_2) \, e^{-(m-1)u_2/2} \, du_2 \\
& = \frac{ \mathrm{vol}(\Gamma_K \backslash \bH^2)}{8 \pi^2} 
\int_{\R} r_1 h_1(r_1) \tanh (\pi r_1) \, dr_1 
\int_{\R} g_2(u_2) e^{-(m-1)u_2/2} \, du_2 \times \frac{(m-1)}{2} \\
& = (m-1) \frac{ \mathrm{vol}(\Gamma_K \backslash \bH^2)}{16 \pi^2}
\, h_2(\tfrac{i(m-1)}{2})
 \int_{\R} r_1 h_1(r_1) \tanh (\pi r_1) \, dr_1. 
\end{align*}

$\bullet$ Elliptic terms: \\
Let $R$ be an elliptic element.
Put $\overline{E(m;R)}:= E(0,m;R) - E(0,m-2;R)$. Here, $E(m_1,m_2;R)$ is defined in (\ref{term:ell}).
Then, we have 
\begin{align*}
& \overline{E(m;R)} 
  = \frac{ - e^{-i \theta_1 + i(m-1) \theta_2}}{16 \nu_R \sin \theta_1 \sin \theta_2}
\int \! \! \! \! \int_{\R^2}
g(u_1,u_2) 
e^{\frac{-u_1}{2}} 
\Bigl\{ e^{\frac{(m-1)u_2}{2}}   -  e^{-2i \theta_2} \, e^{\frac{(m-3)u_2}{2}} \Bigr\} \\
& \quad \quad \quad \quad \quad \quad 
\times \prod_{j=1}^{2}
\Bigl[
\frac{e^{u_j}-e^{2 i \theta_j}}{\cosh u_j - \cos 2 \theta_j} \Bigr]
du_1 du_2 \\
& =  \frac{ - e^{-i \theta_1 + i(m-1) \theta_2} }{8 \nu_R \sin \theta_1 \sin \theta_2}
\int_{\R} g_1(u_1) \, e^{-u_1/2} 
\Bigl[
\frac{e^{u_1}-e^{2 i \theta_1}}{\cosh u_1 - \cos 2 \theta_1} \Bigr]
du_1
\int_{\R} g_2(u_2) e^{\frac{m-1}{2} u_2} \, du_2 \\
& =  \frac{ - e^{-i \theta_1 + i(m-1) \theta_2} }{8 \nu_R \sin \theta_1 \sin \theta_2}
\, h_2(\tfrac{i(m-1)}{2})
\int_{\R} g_1(u_1) \, e^{-u/2} 
\Bigl[
\frac{e^{u_1}-e^{2 i \theta_1}}{\cosh u_1 - \cos 2 \theta_1} \Bigr] du_1. 
\end{align*}

$\bullet$ Hyperbolic-elliptic terms: \\ 
Let $\gamma$ be a hyperbolic-elliptic element.
Put $\overline{HE(m;\gamma)}:= HE(0,m;\gamma) - HE(0,m-2;\gamma)$. 
Here, $HE(m_1,m_2;\gamma)$ is defined in (\ref{term:hyp-ell}).
Then, we obtain, 

\begin{align*}
& \overline{HE(m;\gamma)} 
=\frac{\log N(\gamma_0)}{N(\gamma)^{1/2}-N(\gamma)^{-1/2}}
\frac{ie^{i(m-1)\omega}}{4 \sin \omega}
\int_{-\infty}^{\infty} \! \! g(\log N(\gamma),u)
\, e^{\frac{m-1}{2}u} \Bigl\{ 1- e^{-2i \omega}  e^{-u} \Bigr\} \\
& \quad \quad \quad \quad \quad \quad \quad \times 
\Bigl[ \frac{e^{u}-e^{2 i \omega}}{\cosh u - \cos 2 \omega} \Bigr]
du \\
& = \frac{\log N(\gamma_0)}{N(\gamma)^{1/2}-N(\gamma)^{-1/2}}
\frac{ie^{i(m-1)\omega}}{2 \sin \omega}
\int_{-\infty}^{\infty} \! \! g(\log N(\gamma),u) \, 
e^{\frac{m-1}{2}u} \, du \\
& = \frac{\log N(\gamma_0)}{N(\gamma)^{1/2}-N(\gamma)^{-1/2}}
\, g_1(\log N(\gamma)) \, 
\frac{ie^{i(m-1)\omega}}{2 \sin \omega} \, 
h_2(\tfrac{i(m-1)}{2}).
\end{align*}

$\bullet$ Parabolic contribution: \\ 
Put $\overline{P(m)}:= P(0,m) - P(0,m-2)$. 
Here, $P(0,m)$ is defined in (\ref{eq:ph2sc}).
Then we have,
\begin{equation} \label{d:pm}
\overline{P(m)} 
= - \log \ve \, g_1(0)
\biggl[ \int_{0}^{\infty} g_2(u) \, e^{\frac{m-1}{2}u}  \, du 
-\int_{0}^{\infty} g_2(u) \, e^{-\frac{m-1}{2}u}  \, du 
\biggr].
\end{equation}

$\bullet$ Type 2 hyperbolic contribution: \\ 
Put $\overline{H_2(m)}:= H_2(0,m) - H_2(0,m-2)$. 
Here, $H_2(0,m)$ is defined in (\ref{eq:ph2sc}).
Then we have,
\begin{equation*} 
\begin{split}
\overline{H_2(m)}  &= 2 \log \ve \sum_{k=1}^{\infty} \int_{2k \log \ve}^{\infty}
g(2k \log \ve, u) \\
& \quad  
\times \frac{\cosh \bigl( (m-2)(u/2-k \log \ve) \bigr)  
- \cosh \bigl( m(u/2-k \log \ve) \bigr) }{\sinh(u/2-k \log \ve)}
\, du \\
& = 
\log \ve \sum_{k=1}^{\infty} \int_{2k \log \ve}^{\infty}
g(2k \log \ve, u) 
\Bigl\{  
e^{-(m-1)(u/2-k \log \ve)}   - e^{(m-1)(u/2-k \log \ve)} 
\Bigr\}
\, du. 
\end{split}
\end{equation*}
Therefore, we have
\begin{equation} \label{d:h2m}
\begin{split}
& \overline{H_2(m)} = 2 \log \ve 
\sum_{k=1}^{\infty} 
g_1(2k \log \ve)
\biggl[
    \ve^{k(m-1)} \int_{2k \log \ve}^{\infty} g_2(u) \, e^{-\frac{m-1}{2} u} \, du
    \\
& \quad \quad \quad \quad 
-  \ve^{-k(m-1)} \int_{2k \log \ve}^{\infty} g_2(u) \, e^{\frac{m-1}{2} u} \, du
\biggr].
\end{split}
\end{equation}

Finally, we calculate the scattering contribution to the differences of the trace formula 
for Hilbert modular surfaces. Put $\overline{SC(m)} := SC(0,m)-SC(0,m-2)$.
Here, $SC(0,m)$ is defined in (\ref{eq:ph2sc}).

\begin{proposition}[Scattering contribution]
\begin{equation} \label{d:eim}
\begin{split}
& \overline{SC(m)} = -2 \sgn(m-1) \log \ve 
\sum_{k=1}^{\infty} g_1(2k \log \ve)
\biggl[ \ve^{-k|m-1|}  \int_{0}^{\infty} g_2(u) \, e^{-\frac{|m-1|}{2}u} \, du    
\\ 
& \quad \quad \quad \quad 
+ \ve^{k|m-1|}  \int_{2k \log \ve}^{\infty} g_2(u) \, e^{-\frac{|m-1|}{2}u} \, du 
+ \ve^{-k|m-1|}  \int_{0}^{2k \log \ve} g_2(u) \, e^{\frac{|m-1|}{2}u} \, du 
\biggr] \\
& \quad \quad \quad \quad 
- 2 \sgn(m-1) \log \ve \, g_1(0) \int_{0}^{\infty} g_2(u) \, e^{-\frac{|m-1|}{2}u} \, du.
\end{split}
\end{equation}
\end{proposition}
\begin{proof}
Firstly, we can easily check that
\[
\varphi_{(0,m)}(s,0) 
=  \frac{(-1)^{m/2}}{D^{1-2s} \, \pi^{4s-2}} \, 
\frac{\Gamma(s)^2 \, \Gamma(s-\frac{1}{2})^2}{\Gamma(s+\frac{m}{2}) \, \Gamma(s-\frac{m}{2})
\, \Gamma(\frac{1}{2}-s)^2} \, \frac{\zeta_{K}(2s-1)}{\zeta_{K}(1-2s)},
\]
by using the functional equation of the Dedekind zeta function $\zeta_{K}(s)$. 
Thus we have $\varphi_{(0,m)}(\frac{1}{2},0)=1$. 

Secondly, by the explicit formula for $\varphi_{(0,m)}(s,k)$, (see (\ref{eq:scattering}))
we see that
\[ \varphi_{(0,m)}(s,k) \, \varphi_{(0,m-2)}(s,k)^{-1}
  = \Bigl( s - \tfrac{\pi i k}{2 \log \ve} - \tfrac{m}{2} \Bigr)
 \Bigl( s - \tfrac{\pi i k}{2 \log \ve} + \tfrac{m}{2} -1\Bigr)^{-1}.   
\]
So we have
\begin{align*}
&   \biggl( \frac{\varphi_{(0,m)}'}{\varphi_{(0,m)}}  
- \frac{\varphi_{(0,m-2)}'}{\varphi_{(0,m-2)}}  
\biggr) \bigl( \tfrac{1}{2}+ir, k \bigr) = 
\frac{1}{\frac{1}{2}-\frac{m}{2} + i (r - \frac{\pi k }{2 \log \ve})}   
- \frac{1}{-\frac{1}{2}+\frac{m}{2} + i (r - \frac{\pi k }{2 \log \ve})} \\
& \quad \quad = - \frac{m-1}{(r-\frac{\pi k}{2 \log \ve})^2+(\frac{m-1}{2})^2}.
\end{align*}
Therefore, we have
\begin{equation} \label{eq:eim2}
\begin{split}
& \overline{SC(m)}
 = - \frac{1}{4 \pi}  \sum_{k \in \Z}
\int_{-\infty}^{\infty} h \Bigl( r+\frac{\pi k}{2 \log \ve}, r-\frac{\pi k}{2 \log \ve} \Bigr)
\frac{m-1}{(r-\frac{\pi k}{2 \log \ve})^2+(\frac{m-1}{2})^2} \, dr \\
& = - \frac{1}{4 \pi}  \sum_{k \in \Z}
\int_{-\infty}^{\infty} h \Bigl( r+\frac{\pi k}{ \log \ve}, r \Bigr)
\frac{m-1}{r^2+(\frac{m-1}{2})^2} \, dr. 
\end{split}
\end{equation}
Thirdly, we use the Poisson summation formula to calculate 
(\ref{eq:eim2}) further. 
Let us determine the sequence $\{ a_k \}$ such that
\[ \sum_{k \in \Z} h_1\Bigl( r+\frac{\pi k}{ \log \ve} \Bigr)
 = \sum_{k \in \Z} a_k \, \exp \Bigl( 2 \pi i k r \cdot \frac{\log \ve }{\pi} \Bigr).
 \]
Then 
\begin{align*}
a_k =& \int_{0}^{\pi /\log \ve}
\sum_{k \in \Z} h_1\Bigl( r+\frac{\pi k}{ \log \ve} \Bigr) e^{- 2 k \log \ve \cdot ir}
\, dr 
= \frac{\log \ve}{\pi} \int_{- \infty}^{\infty} h_1(r) \, e^{- 2 k \log \ve \cdot ir} 
\, dr \\
= & \frac{\log \ve}{\pi} \, (2 \pi) \, g_1(2k \log \ve)
= 2 \log \ve \,  g_1(2k \log \ve).
\end{align*}
So (\ref{eq:eim2}) is written as
\begin{equation} \label{eq:eim3}
\overline{SC(m)}
 = - \frac{2 \log \ve}{4 \pi}  \sum_{k \in \Z} g_1(2k \log \ve) 
\int_{-\infty}^{\infty} h_2(r) \, e^{2k \log \ve \, ir}
\frac{m-1}{r^2+(\frac{m-1}{2})^2} \, dr. 
\end{equation}
Finally, let us evaluate the following integral
\[ I_0 := \frac{1}{4 \pi} \int_{-\infty}^{\infty} h_2(r) \, \frac{m-1}{r^2+(\frac{m-1}{2})^2} \, dr
, \]
and 
\[ I_{k} := \frac{1}{4 \pi} \int_{-\infty}^{\infty} h_2(r) 
\Bigl(  
e^{2k \log \ve \, ir}   
+e^{-2k \log \ve \, ir}
\Bigr)
\frac{m-1}{r^2+(\frac{m-1}{2})^2} \, dr
\]
for $k \in \Z$ and $k>0$.
Recall that
\[ h_2(r) = 2 \int_{0}^{\infty} g_2(u) \cos(ru) \, du \]
and (see \cite[3.723 (2)]{GR})
\[ \int_{0}^{\infty} \frac{\cos(ru)}{r^2+(\frac{m-1}{2})^2} \, dr
=  \frac{\pi }{|m-1|} e^{-\frac{|m-1|}{2}|u|}.
\] 
for $m \ne 1$ (we assumed that $m \in 2 \Z$).
Then we obtain
\[ I_0 = \sgn(m-1) \int_{0}^{\infty} g_2(u) \, e^{-\frac{|m-1|}{2} u} \, du
.\]
While, we have
\begin{align*}
I_k =& \frac{1}{\pi } \int_{0} ^{\infty} \! \! \int_{0}^{\infty}
g_2(u) \cdot 2 \cos (r u) \, \cos(r \cdot  2k \log \ve)
\frac{m-1}{r^2+(\frac{m-1}{2})^2} \, dr du \\
=&
\frac{1}{\pi } \int_{0} ^{\infty} \! \! \int_{0}^{\infty}
g_2(u) \, \cos \bigl( r (u+2k \log \ve) \bigr)
\frac{m-1}{r^2+(\frac{m-1}{2})^2} \, dr du \\
& +\frac{1}{\pi } \int_{0} ^{\infty} \! \! \int_{0}^{\infty}
g_2(u) \, \cos \bigl( r (u-2k \log \ve) \bigr)
\frac{m-1}{r^2+(\frac{m-1}{2})^2} \, dr du. 
\end{align*}
Therefore, we have
\begin{align*}
I_k = & \sgn(m-1) \int_{0}^{\infty} g_2(u) e^{-\frac{|m-1|}{2}(u+2k \log \ve)}
\, du 
+ 
\sgn(m-1) \int_{0}^{\infty} g_2(u) e^{-\frac{|m-1|}{2}|u-2k \log \ve|}
\, du \\
=& \sgn(m-1) \ve^{-k|m-1|} \int_{0}^{\infty} g_2(u) \, e^{-\frac{|m-1|}{2}u} \, du
+ \sgn(m-1) \ve^{k|m-1|} \int_{2k \log \ve}^{\infty} g_2(u) \, e^{-\frac{|m-1|}{2}u} \, du
\\
& \quad 
+ \sgn(m-1) \ve^{-k|m-1|} \int_{0}^{2k \log \ve} g_2(u) \, e^{\frac{|m-1|}{2}u} \, du
\end{align*}
for $k \in \N$.
We complete the proof.
\end{proof}

We can now put together with, the parabolic contribution
(\ref{d:pm}),  the type 2 hyperbolic contribution (\ref{d:h2m}) and 
the scattering contribution (\ref{d:eim}), then we obtain
\begin{proposition}
\begin{equation}
\begin{split}
\overline{P(m)}+\overline{H_2(m)}+\overline{SC(m)}  
=& 
- \sgn(m-1) \log \ve \, g_1(0) \, h_2(\tfrac{i(m-1)}{2}) \\
& \quad -2 \sgn(m-1) \log \ve \sum_{k=1}^{\infty} g_1(2k \log \ve) \, \ve^{-k|m-1|}
\, h_2(\tfrac{i(m-1)}{2}).
\end{split}
\end{equation}
\end{proposition}
\begin{proof}
By (\ref{d:pm}) and (\ref{d:h2m}), we see that
\[ \overline{P(m)} 
= - \sgn(m-1) \log \ve \, g_1(0)
\biggl[ \int_{0}^{\infty} g_2(u) \, e^{\frac{|m-1|}{2}u}  \, du 
-\int_{0}^{\infty} g_2(u) \, e^{-\frac{|m-1|}{2}u}  \, du 
\biggr],
\]
and 
\begin{align*}
& \overline{H_2(m)} = 2 \sgn(m-1) \log \ve 
\sum_{k=1}^{\infty} 
g_1(2k \log \ve)
\biggl[
    \ve^{k|m-1|} \int_{2k \log \ve}^{\infty} g_2(u) \, e^{-\frac{|m-1|}{2} u} \, du
    \\
& \quad \quad \quad \quad 
-  \ve^{-k|m-1|} \int_{2k \log \ve}^{\infty} g_2(u) \, e^{\frac{|m-1|}{2} u} \, du
\biggr].
\end{align*}
Thus we have
\begin{align*}
& \overline{P(m)}+\overline{H_2(m)}+\overline{SC(m)}  \\
& = 
- \sgn(m-1) \log \ve \, g_1(0) \, 
\biggl[  
\int_{0}^{\infty} g_2(u) e^{\frac{|m-1|}{2}u} \, du
+\int_{0}^{\infty} g_2(u) e^{-\frac{|m-1|}{2}u} \, du
\biggr]
\\
& \quad -2 \sgn(m-1) \log \ve \sum_{k=1}^{\infty} g_1(2k \log \ve) 
\biggl[
    \ve^{-k|m-1|} \int_{0}^{\infty} g_2(u) \, e^{-\frac{|m-1|}{2} u} \, du
    \\
& \quad \quad \quad \quad 
+  \ve^{-k|m-1|} \int_{2k \log \ve}^{\infty} g_2(u) \, e^{\frac{|m-1|}{2} u} \, du
+  \ve^{-k|m-1|} \int_{0}^{2k \log \ve} g_2(u) \, e^{\frac{|m-1|}{2} u} \, du
\biggr].
\end{align*}
The rest is clear.
\end{proof}

By using the above proposition, we complete the proof
of Theorem \ref{th:dtrf}.

\subsection{Double differences of the Selberg trace formula}

We wrote down the differences of the Selberg trace formulas 
for $L^2(\Gamma_{K} \backslash \bH^2 \, ; \, (0,m))$
and $L^2(\Gamma_{K} \backslash \bH^2 \, ; \, (0,m-2))$
in Theorem \ref{th:dtrf}. 
Let us denote the above differences formulas  
as $L(m) - L(m-2)$. 
Next we assume that $h_2(\frac{i(m-1)}{2}) \ne 0$ and 
$h_2(\frac{i(m-3)}{2}) \ne 0$, then consider
the ``double differences":
\[ \bigl( L(m) - L(m-2) \bigr)  h_2(\tfrac{i(m-1)}{2}) ^{-1}
   - \bigl( L(m-2) - L(m-4) \bigr)  h_2(\tfrac{i(m-3)}{2}) ^{-1}.
\]
Then we have, 

\begin{theorem}[Double differences of STF 
for $L^2 \bigl( \Gamma_{K} \backslash \bH^2 \, ; \, (0,m) \bigr)$]
\label{th:ddtrf}
Let $m \in 2 \Z$. We have 
\begin{align*}
& \sum_{j=0}^{\infty} h_1 \Bigl( \rho_j(m) \Bigr)  
   - \sum_{j=0}^{\infty} h_1 \Bigl( \mu_j(m-2) \Bigr) 
   - \sum_{j=0}^{\infty} h_1 \Bigl( \rho_j(m-2) \Bigr)  
   + \sum_{j=0}^{\infty} h_1 \Bigl( \mu_j(m-4) \Bigr) \\
% \: + \delta_{m,4} \, h_1(\tfrac{i}{2}) 
 & \quad =  \:
\frac{\mathrm{vol}(\Gamma_K \backslash \bH^2)}{8 \pi^2}
\int_{-\infty}^{\infty}
  r h_1(r) \tanh (\pi r)  \, dr \\
& \quad \quad \quad - \sum_{R(\theta_1,\theta_2) \in 
\Gamma_{\mathrm{E}}} 
\frac{i e^{ - i \theta_1} \, e^{i(m-2) \theta_2}}{4 \nu_{R} \sin \theta_1}
\int_{-\infty}^{\infty} g_1(u) \, e^{-u/2} \biggl[  \frac{e^u - e^{2 i \theta_1}}{\cosh u - \cos 2 \theta_1} \biggr] 
du \\
& \quad \quad \quad
 - \sum_{(\gamma,\omega) \in \Gamma_{\mathrm{HE}}} 
\frac{\log N(\gamma_0)}{N(\gamma)^{1/2}-N(\gamma)^{-1/2}} \, g_1(\log N(\gamma)) 
\, e^{i(m-2)\omega}
\\
& \quad \quad \quad - \log \varepsilon \, g_1(0)
 \Bigl( \sgn (m-1)  - \sgn (m-3) \Bigr) 
\\
& \quad \quad \quad - 2 \log \varepsilon \, 
\sum_{k=1}^{\infty} g_1(2k \log \varepsilon) \, 
 \Bigl( \sgn (m-1) \, \varepsilon^{-k|m-1|}  - \sgn (m-3) \, \varepsilon^{-k|m-3|} \Bigr). 
\end{align*}
\end{theorem}
\begin{proof}
By direct computation.
\end{proof}

%%%%%%%%%%%%%%%%%%%%%%%%%%%%%%%%%%%%%%%%%%%%%%%%%%%%%%%%%%%%%%%%%%%%%%%%%%%%%%%%%%%%%%%%%%%%%
%%%%%%%%%%%%%%%%%%%%%%%%%%%%%%%%%%%%%%%%%%%%%%%%%%%%%%%%%%%%%%%%%%%%%%%%%%%%%%%%%%%%%%%%%%%%%
\section{Selberg type zeta functions for Hilbert modular surfaces} 
%%%%%%%%%%%%%%%%%%%%%%%%%%%%%%%%%%%%%%%%%%%%%%%%%%%%%%%%%%%%%%%%%%%%%%%%%%%%%%%%%%%%%%%%%%%%%
%%%%%%%%%%%%%%%%%%%%%%%%%%%%%%%%%%%%%%%%%%%%%%%%%%%%%%%%%%%%%%%%%%%%%%%%%%%%%%%%%%%%%%%%%%%%%

\subsection{Selberg type zeta functions}

Let $(\gamma, \gamma') \in \Gamma_{K}$ be hyperbolic-elliptic,
i.e, $|\tr(\gamma)|>2$ and $|\tr(\gamma')|<2$. 
Then the centralizer of hyperbolic-elliptic 
$(\gamma, \gamma')$ in $\Gamma_K$ is infinite cyclic.

\begin{definition}[Selberg type zeta function for $\Gamma_{K}$ with the weight $(0,m)$]
Let $m \ge 2$ be an even integer.
The Selberg type zeta function for $\Gamma_K$ with the weight 
$(0,m)$ is defined by the following Euler product:
\label{def:szeta}
\[
Z_{K}(s;m) := \prod_{(p,p')} 
\prod_{k=0}^{\infty} \Bigl( 1-e^{i (m-2) \omega_0} \, N(p)^{-(k+s)} 
\Bigr)^{-1} 
\quad \mbox{for $\RE(s) > 1$} .
\]
Here, $(p,p')$ run through the set of primitive hyperbolic-elliptic 
$\Gamma_K$-conjugacy classes of $\Gamma_K$, and $(p,p')$ is conjugate 
in $\mathrm{PSL}(2,\R)^2$ to
\[ (p, p') \sim \Bigl( 
\Bigl( 
\begin{array}{cc}
N(p)^{1/2} & 0 \\
0 & N(p)^{-1/2}
\end{array} \Bigr), 
\, 
\Bigl( 
\begin{array}{cc}
\cos \omega_0 & - \sin \omega_0 \\
\sin \omega_0 & \cos \omega_0
\end{array} \Bigr)\Bigr),
\]
where, $N(p)>1$, $\omega_0 \in (0,\pi)$ and $\omega_0 \notin \pi \Q$.
\end{definition}
Theorem \ref{th:pgt}, which we prove in the next section by using Theorem 
\ref{th:ddtrf}, ensures that the Euler product is absolutely convergent
for $\RE(s)>1$. Therefore, $Z_{K}(s;m)$ represents a  holomorphic
function on the half plane $\RE(s)>1$.
We remark that the exponent is $-1$ in the definition,
which differs from the original one.

For an even integer $m \le 2$, we see that
\[ Z_{K}(s;m) =  \overline{Z_{K}(\overline{s};4-m)}.
\]
Thus, it is sufficient to consider $Z_{K}(s;m)$ for an even integer $m \ge 2$
by the above relation.

We show that $Z_{K}(s;m)$ has a meromorphic extension to the whole complex plane
by using Theorem \ref{th:ddtrf} (double differences of the Selberg trace formula).

\subsection{Test functions}
Let us consider the logarithmic derivative of $Z_{K}(s;m)$. For $\RE(s)>1$, 
\begin{equation} \label{eq:logzeta}
\begin{split}
\frac{d}{ds} \log Z_{K}(s;m)
=& - \sum_{(p,p')} \sum_{k=0}^{\infty} \log N(p) \frac{e^{i(m-2)\omega_0} N(p)^{-(k+s)}}{1-e^{i(m-2) \omega_0} N(p)^{-(k+s)}} 
\\
=& - \sum_{(p,p')} \sum_{k=0}^{\infty} \sum_{l=1}^{\infty}
\log N(p) \, N(p)^{-kl} N(p)^{-ls} e^{i(m-2) l\omega_0} \\
=&   - \sum_{(p,p')} \sum_{l=1}^{\infty}
\frac{\log N(p)}{1-N(p^l)^{-1}} N(p^l)^{-s} e^{i(m-2) l \omega_0}.
\end{split}
\end{equation}
Usually, we introduce a certain test function $h(r_1,r_2)$ 
to get a meromorphic extension of the logarithmic derivative
of the Selberg type zeta functions.

We can check that the Selberg trace formula (Theorem \ref{th:trf})
%, the Selberg trace formula for 
%$L^2(\Gamma_{K} \backslash \bH^2 
%\, ; \, (0,m))$, 
holds for the test function $h(r_1,r_2)$ which satisfies the following 
condition (see \cite[p.105]{E} or \cite[p.1651]{Z}):
\begin{enumerate}
\item $h(\pm r_1, \pm r_2)=h(r_1,r_2)$, 
\item $h$ is analytic in the domain $|\IM (r_1)| < \frac{1}{2} +\delta$, 
$|\IM (r_2)| < \frac{||m|-1|}{2} +\delta$ for some $\delta>0$, 
\item $h(r_1,r_2) = O \bigl( (1+|r_1|^2+|r_2|^2 \bigr)^{-2-\delta})$ for 
some $\delta>0$ in this domain.
\item $g_2(u_2) \in C_c^{\infty}(\R)$. 
\end{enumerate}
We remark that the last condition assures the
absolute convergence of the geometric side of  Theorem \ref{th:trf},
in particular that of $\mathbf{II}_b(h)$.

Let us consider the following test function:
Firstly, we fix real numbers $\beta_1, \beta_2 \ge 2$, $\beta_1 \ne \beta_2$. 
For $s \in \C$, $\RE(s)>1$, 
We set
\begin{equation} \label{eq:test}
h_1(r) : = \frac{\bigl( (\beta_1^2-(s-\frac{1}{2})^2 \bigr) \bigl( (\beta_2^2-(s-\frac{1}{2})^2 \bigr)}{\bigl( r^2+(s-\frac{1}{2})^2 \bigr) (r^2+\beta_1^2)(r^2+\beta_2^2)} \\
 = \frac{1}{r^2+(s-\frac{1}{2})^2} 
    +\frac{c_1(s)}{r^2+\beta_1^2}
    +\frac{c_2(s)}{r^2+\beta_2^2}
\end{equation}
with
\[ c_1(s) = \frac{(s-\frac{1}{2})^2-\beta_2^2}{\beta_2^2 - \beta_1^2}, 
\quad c_2(s) = - \frac{(s-\frac{1}{2})^2-\beta_1^2}{\beta_2^2 - \beta_1^2}.
 \] 
(See Parnovskii \cite{Par} for this type test functions).
The Fourier transform of $h_1$ is given by
\[ g_1(u) = \frac{1}{2 \pi} \int_{-\infty}^{\infty} h_1(r) e^{-i ru} \, dr
= \frac{1}{2s-1} e^{-(s-\frac{1}{2})|u|}
+ \frac{c_1(s)}{2\beta_1} e^{-\beta_1 \, |u|}
+ \frac{c_2(s)}{2\beta_2} e^{-\beta_2 \, |u|}.
\]
Secondly, we take $g_2(u) \in C_{c}^{\infty}(\R)$ such that
its Fourier inverse transform $h_2(r)$ satisfies 
$h_2(\frac{i(m-1)}{2}) \ne 0$ and 
$h_2(\frac{i(m-3)}{2}) \ne 0$. Then we can easily check that 
our test function $h(r_1,r_2) := h_1(r_1) \, h_2(r_2)$
satisfies the above sufficient condition for Theorem \ref{th:trf}.

Thirdly, 
let us assume that $m \ge 4$ for simplicity. 
We will treat the case of $m=2$ in the next subsection.
Let $m \ge 4$ be an even integer. 
Then we have
\begin{equation} \label{eq:m=4}
\Bigl( \Ker(K_{m-2}^{(2)}), \,  \Ker(K_{m-4}^{(2)}) \Bigr)
= 
\begin{cases}
(\{ 0 \}, \, \C)   \quad \mbox{if $m = 4$}, \\
(\{ 0\}, \, \{ 0\}) \quad \mbox{if $m \ge 6$}.
\end{cases}
\end{equation}
Here, $K_{m-2}^{(2)}$, $K_{m-4}^{(2)}$
are the weight up Maass operators and
their kernel are given by
\[ \Ker(K_{q-2}^{(2)}) =
\biggl\{  f \in  L^2_\text{dis} \Bigl( \Gamma_{K} \backslash \bH^2 
\, ; \, (0,q-2) \Bigr)  \, \Bigl| \,  \Delta_{q-2}^{(2)} f  = \frac{q}{2} \Bigl( 1-\frac{q}{2} \Bigr) \, f  
\biggr\} \]
for $q=m, m-2$. The fact (\ref{eq:m=4}) is deduced from
Lemma \ref{lem:least-eigen}
and that the Hilbert modular group $\Gamma_K$ is an irreducible
discrete subgroup.
We also recall that $\{\frac{1}{4}+\rho_j(m)^2 \}_{j=0}^{\infty}$ 
and 
 $\{\frac{1}{4}+\rho_j(m-2)^2 \}_{j=0}^{\infty}$ 
are the set of eigenvalues of  $\Delta_{0}^{(1)}$ acting 
on $\Ker(\Lambda_{m}^{(2)})$
and  $\Ker(\Lambda_{m-2}^{(2)})$ respectively.
Here, $\Lambda_{m}^{(2)}$, $\Lambda_{m-2}^{(2)}$
are the weight down Maass operators and
their kernel are given by
\[ \Ker(\Lambda_{q}^{(2)}) =
\biggl\{  f \in  L^2_\text{dis} \Bigl( \Gamma_{K} \backslash \bH^2 
\, ; \, (0,q) \Bigr)  \, \Bigl| \,  \Delta_{q}^{(2)} f  = \frac{q}{2} \Bigl( 1-\frac{q}{2} \Bigr) \, f  
\biggr\} \]
for $q=m,m-2$.
If we set $\lambda_j(q) := \frac{1}{4}+\rho_j(q)^2$ for $q=m,m-2$, we note that
\begin{equation} \label{eq:non-zero}
0 < \lambda_0(q) \le \lambda_1(q) \le \lambda_2(q) \le \dots
\end{equation}
since $\Gamma_K$ is irreducible and $m \ge4$. 

Finally, we consider Theorem \ref{th:ddtrf}, the double difference of 
the Selberg trace formula, for the above test function 
(\ref{eq:test}). Then we have

\begin{theorem}[Double differences of STF for the above test function $h_1$ 
and $h_2$]  \label{th:ddtrfa}
Let $m \ge 4$ be an even integer. 
For $\RE(s)>1$, we have
\begin{align*}
& \sum_{j=0}^{\infty} \Bigl[ \frac{1}{\rho_j(m)^2+(s-\frac{1}{2})^2} 
 + \sum_{h=1}^{2}\frac{c_h(s)}{\rho_j(m)^2 + \beta_h^2} \Bigr] \\
& \quad - \sum_{j=0}^{\infty} \Bigl[ \frac{1}{\rho_j(m-2)^2+(s-\frac{1}{2})^2} 
 + \sum_{h=1}^{2} \frac{c_h(s)}{\rho_j({m-2})^2 + \beta_h^2} \Bigr] 
+ \delta_{m,4} \Bigl[ \frac{1}{s(s-1)} 
 + \sum_{h=1}^{2}\frac{c_h(s)}{\beta_h^2-\frac{1}{4}} \Bigr] \\
& = 2 \zeta_{K}(-1) \sum_{k=0}^{\infty}
\Bigl[ \frac{1}{s+k} 
+ \sum_{h=1}^{2} \frac{c_h(s)}{\beta_h+\frac{1}{2}+k} \Bigr] 
+ \frac{1}{2s-1} {\frac{Z_{K}'(s;m)}{Z_{K}(s;m)}}
+\sum_{h=1}^{2} \frac{c_h(s)}{2 \beta_h} 
\frac{Z_{K}'(\frac{1}{2}+\beta_h;m)}{Z_{K}(\frac{1}{2}+\beta_h;m)}
\\
& \quad + \frac{1}{2s-1} 
\sum_{j=1}^{N} \sum_{l=0}^{\nu_j-1} \frac{\nu_j-1-\alpha_l(m,j)-\overline{\alpha_l}(m,j)}{\nu_j^2}
\, \psi \Bigl( \frac{s+l}{\nu_j} \Bigr) \\
& \quad 
+ \sum_{h=1}^{2} \frac{c_h(s)}{2 \beta_h} 
\sum_{j=1}^{N} \sum_{l=0}^{\nu_j-1} \frac{\nu_j-1-\alpha_l(m,j)-\overline{\alpha_l}(m,j)}{\nu_j^2}
\, \psi \Bigl( \frac{\frac{1}{2}+\beta_h+l}{\nu_j} \Bigr) 
\\
& \quad + \frac{1}{2s-1} \frac{d}{ds} 
\log \biggl\{ \frac{(1-\varepsilon^{-(2s+m-4)})}
{(1-\varepsilon^{-(2s+m-2)})} \biggr\} 
+ \sum_{h=1}^{2} \frac{c_h(s)}{2 \beta_h} \frac{d}{d \beta_h} 
\log \biggl\{ 
\frac{(1-\varepsilon^{-(2 \beta _h +m-3)})}
{(1-\varepsilon^{-(2 \beta_h +m-1)})}
\biggr\}. 
\end{align*}
Here, $\psi(z)=\frac{d}{dz} \log \Gamma(z)$ is the digamma function.
\end{theorem}
\begin{proof}
We compute each terms appearing in Theorem \ref{th:ddtrf} for the test function
$h_1(r)$ given by (\ref{eq:test}). \\

$\bullet$ Discrete spectrum: 
We denote the spectral side of Theorem \ref{th:ddtrf} by $A_\text{spec}(s;m)$.
By the fact (\ref{eq:m=4}), we see that
\[ A_\text{spec}(s;m) = 
   \sum_{j=0}^{\infty} h_1 \bigl( \rho_j(m) \bigr) 
- \sum_{j=0}^{\infty} h_1 \bigl( \rho_j(m-2) \bigr) 
+ \delta_{m,4} \, h_1(i/2). \]
 
$\bullet$ Identity term: We denote the identity contribution by $A_\text{id}(s)$. 
By the proof of Proposition 4.9 in \cite{H1}, 
\begin{align*}
A_\text{id}(s)
 =& \frac{\mathrm{vol}(\Gamma_K \backslash \bH^2)}{8 \pi^2}
 \biggl\{ \pi \tan \Bigl( \pi \Bigl( s-\frac{1}{2} \Bigr) \Bigr) 
+ \sum_{h=1}^{2}  c_h(s) \, \pi \tan(\pi \beta_h)
 \\
\quad & +\sum_{k=0}^{\infty}
\Bigl[ \frac{1}{s+k} -\frac{1}{s-1-k}
+ \sum_{h=1}^{2} \frac{c_h(s)}{\beta_h+\frac{1}{2}+k} 
- \sum_{h=1}^{2} \frac{c_h(s)}{\beta_h-\frac{1}{2}-k} 
\Bigr] 
\biggr\}
\\
 =&  2 \zeta_{K}(-1) \sum_{k=0}^{\infty}
\Bigl[ \frac{1}{s+k} 
+ \sum_{h=1}^{2} \frac{c_h(s)}{\beta_h+\frac{1}{2}+k} \Bigr] .
\end{align*}
Here, we used the partial fractional expansion of $\cot(\pi z)$, 
the fact $c_1(s)+c_2(s) = -1$ and the formula by Siegel
(see Theorem (1.1) in \cite{Ge} or Proposition 5.1 in \cite{F}):
\begin{equation} \label{eq:volume-zeta}
\frac{\mathrm{vol}(\Gamma_K \backslash \bH^2)}{4 \pi^2}
   = 2 \zeta_{K}(-1).
\end{equation}

$\bullet$ Elliptic term: 
We denote the elliptic contribution by $A_\text{ell}(s;m)$. 
This is given by 
\begin{align*}
A_\text{ell}(s;m)
 =& 
- \sum_{R(\theta_1,\theta_2) \in 
\Gamma_{\mathrm{E}}} 
\frac{i e^{ - i \theta_1} \, e^{i(m-2) \theta_2}}{4 \nu_{R} \sin \theta_1}
\int_{-\infty}^{\infty} g_1(u) \, e^{-u/2} \biggl[  \frac{e^u - e^{2 i \theta_1}}{\cosh u - \cos 2 \theta_1} \biggr] 
du \\
=& 
- \sum_{R(\theta_1,\theta_2) \in 
\Gamma_{\mathrm{E}}} 
\frac{e^{i(m-2) \theta_2}}{\nu_{R}}
\int_{0}^{\infty} \biggl[  \frac{g_1(u) \, \cosh(u/2)}{\cosh u - \cos 2 \theta_1} \biggr] 
du.
\end{align*}
By noting (10.29) and (10.31) in \cite{Iw}, we have
\begin{align*}
A_\text{ell}(s;m) 
=& \frac{1}{2s-1} 
\sum_{j=1}^{N} \sum_{k=1}^{\nu_j-1} \sum_{l=0}^{\nu_j-1}
\frac{ \exp \bigl( i(m-2) (\pi i k t_j )/\nu_j \bigr)}{\nu_j^2}
\frac{\sin \bigl( (2l+1) \, \pi k /\nu_j \bigr)}{\sin (\pi k/\nu_j)}
\, \psi \Bigl( \frac{s+l}{\nu_j} \Bigr) \\
&+  \sum_{h=1}^{2} c_h(s) \cdot \biggl\{  \mbox{the same for $s=\frac{1}{2}+\beta_h$} \biggr\}
.
\end{align*}
Next we use the following equality 
\begin{align*}
&\sum_{k=1}^{\nu_j-1} 
e^{i(m-2) (\pi i k t_j )/\nu_j }
\frac{\sin \bigl( (2l+1) \, \pi k /\nu_j \bigr)}{\sin (\pi k/\nu_j)} 
\\
& = - \frac{1}{2} 
\biggl( 
\sum_{k=1}^{\nu_j-1} \frac{ i e^{i (2 \alpha_l(m,j) +1) \pi k /\nu_j }}{\sin(\pi k /\nu_j)}
- 
\sum_{k=1}^{\nu_j-1} \frac{ i e^{-i (2 \overline{\alpha_l}(m,j) +1) \pi k /\nu_j}}{\sin(\pi k /\nu_j)}
\biggr)  \\
& = - \frac{1}{2} \Bigl\{  -\bigl( \nu_j-1-2 \alpha_l(j,m) \bigr)  
  - \bigl( \nu_j-1-2 \overline{\alpha_l}(j,m) \bigr) \Bigr\}  
\\
& = \nu_j-1-\alpha_l(j,m)-\overline{\alpha_l}(j,m). 
\end{align*}
Here, the integers $\alpha_l(j,m), \, \overline{\alpha_l}(j,m) \in \{ 0,1, \dots \nu_j-1 \}$
are defined in (\ref{eq:alpha}). 
The above equality is deduced 
from (see \cite[p.67]{Fis}).
\[ \sum_{k=1}^{\nu-1} \frac{i e^{-i(2a+1)\pi k/\nu}}{\sin(\pi k/\nu)}
= \nu-1 -2a \quad (a \in \{ 0,1, \dots \nu_j-1 \}).
\]

Therefore, we have
\begin{align*}
A_\text{ell}(s;m)
=& \frac{1}{2s-1} 
\sum_{j=1}^{N} \sum_{l=0}^{\nu_j-1} \frac{\nu_j-1-\alpha_l(m,j)-\overline{\alpha_l}(m,j)}{\nu_j^2}
\, \psi \Bigl( \frac{s+l}{\nu_j} \Bigr) \\
&  
+ \sum_{h=1}^{2} \frac{c_h(s)}{2 \beta_h} 
\sum_{j=1}^{N} \sum_{l=0}^{\nu_j-1} \frac{\nu_j-1-\alpha_l(m,j)-\overline{\alpha_l}(m,j)}{\nu_j^2}
\, \psi \Bigl( \frac{\frac{1}{2}+\beta_h+l}{\nu_j} \Bigr).
\end{align*}

$\bullet$ Hyperbolic-elliptic term: 
We denote the hyperbolic-elliptic contribution by $A_\text{hyp-ell}(s)$.
This is given by
\begin{align*}
A_\text{hyp-ell}(s;m)
=& - \frac{1}{2s-1} \sum_{(\gamma,\omega) \in \Gamma_{\mathrm{HE}}} 
\frac{\log N(\gamma_0)}{N(\gamma)^{1/2}-N(\gamma)^{-1/2}} 
N(\gamma)^{-(s-1/2)}
\, e^{i(m-2)\omega} \\
& \quad - \sum_{h=1}^{2} c_h(s) \cdot \biggl\{  \mbox{the same for $s=\frac{1}{2}+\beta_h$} \biggr\} \\
=& \frac{1}{2s-1} {\frac{Z_{K}'(s;m)}{Z_{K}(s;m)}}
+\sum_{h=1}^{2} \frac{c_h(s)}{2 \beta_h} 
\frac{Z_{K}'(\frac{1}{2}+\beta_h;m)}{Z_{K}(\frac{1}{2}+\beta_h;m)}.
\end{align*}
The last equality is derived from (\ref{eq:logzeta}).

$\bullet$ Parabolic plus scattering term: 
Since $m \ge 4$, $\sgn(m-1)-\sgn(m-3)$ vanishes
and this term contributes zero. So, $A_\text{par/sct}(s;m) \equiv 0$. \\

$\bullet$ Type 2 hyperbolic plus scattering term: 
We denote the type 2 plus scattering contribution by $A_\text{hyp2/sct}(s;m)$.
Then, 
\begin{align*}
A_\text{hyp2/sct}(s;m) &= - \frac{2 \log \ve}{2s-1} \sum_{k=1}^{\infty}
\ve^{-k(2s-1)} (\ve^{-k(m-1)}  -\ve^{-k(m-3)}) \\
& \quad - \sum_{h=1}^{2} c_h(s) \cdot \biggl\{  \mbox{the same for $s=\frac{1}{2}+\beta_h$} \biggr\}
.
\end{align*}
Nothing that
\[ \sum_{k=1}^{\infty} \ve^{-k(2s-1)} \ve^{-k(m-1)}
=  \frac{\ve^{-(2s+m-2)}}{1-\ve^{-(2s+m-2)}}
= \frac{-1}{2 \log \ve} \frac{d}{ds} \log \Bigl(  1-\ve^{-(2s+m-2)}   \Bigr)^{-1}
\]
for $\RE(s) >1-m/2 $. For  $\RE(s)>2-m/2$, therefore, we have
\[ A_\text{hyp2/sct}(s;m) = 
\frac{1}{2s-1} \frac{d}{ds} 
\log \biggl\{ \frac{(1-\varepsilon^{-(2s+m-4)})}
{(1-\varepsilon^{-(2s+m-2)})} \biggr\} 
+ \sum_{h=1}^{2} \frac{c_h(s)}{2 \beta_h} \frac{d}{d \beta_h} 
\log \biggl\{ 
\frac{(1-\varepsilon^{-(2 \beta _h +m-3)})}
{(1-\varepsilon^{-(2 \beta_h +m-1)})}
\biggr\}
\]
The proof is finished.
\end{proof}

\subsection{Analytic continuation of Selberg type zeta functions}
We prove
\begin{theorem} \label{th:a1}
For an even integer $m \ge 4$, the Selberg zeta function
$Z_K(s;m)$, originally defined for 
$\RE(s) > 1$, has an analytic continuation to the whole 
complex plane as a meromorphic function.  
\begin{enumerate}
\item $Z_{K}(s;m)$ has zeros at  \\
$s=\frac{1}{2} \pm i \rho_j(m)$ of order equal to the multiplicity 
of the eigenvalue $\frac{1}{4}+\rho_j(m)^2$ of $\Delta_0^{(1)}$ acting
on $\Ker(\Lambda_m^{(2)})$, \\
$s=1-\frac{m}{2}  + \frac{\pi i k}{\log \ve}$ 
of order $1$ for $k \in \Z$. 
\item $Z_{K}(s;m)$ has poles at  \\
$s=\frac{1}{2} \pm i \rho_j(m-2)$ of order equal to the multiplicity 
of the eigenvalue $\frac{1}{4}+\rho_j(m-2)^2$ of $\Delta_0^{(1)}$ acting
on $\Ker(\Lambda_{m-2}^{(2)})$, \\
$s=2-\frac{m}{2} + \frac{\pi i k}{\log \ve}$ 
of order $1$ for $k \in \Z$. 
\item $Z_{K}(s;m)$ has zeros or poles 
(according to their orders are positive or negative) at  \\
$s=-k$ $(k \in \N\cup\{0\})$ of order $(2k+1)E(X_K)+2\sum_{j=1}^{N} [k/\nu_j] - \sum_{j=1}^{N} \beta_{k,j}(m)$. 
\item If $m=4$, $Z_{K}(s,m)$ has additional simple zeros at 
$s=0$ and $s=1$.
\end{enumerate}
Here, 
\[ \Ker(\Lambda_{q}^{(2)}) =
\biggl\{  f \in  L^2_\text{dis} \Bigl( \Gamma_{K} \backslash \bH^2 
\, ; \, (0,q) \Bigr)  \, \Bigl| \,  \Delta_{q}^{(2)} f  = \frac{q}{2} \Bigl( 1-\frac{q}{2} \Bigr) \, f  
\biggr\} \]
for $q=m$ or $m-2$, and
$E(X_{K})$ denotes the Euler characteristic of the Hilbert modular surface $X_{K}$ and
the definition of the integers $\beta_{j,k}(m)$ will be given in 
$($\ref{eq:beta}$)$. 
When the location of two zeros or poles coincide, 
the orders of them are added.
\end{theorem}
\begin{proof}
To get a meromorphic extension of  $Z_{K}(s;m)$,
we show that the logarithmic derivative of $Z_{K}(s;m)$ has a 
meromorphic extension to the whole complex plane 
and its poles are all simple with integral residues. 
By Theorem \ref{th:ddtrfa}, 
it is easy to see that $(2s-1) \, A_\text{spec}(s;m)$ and 
$-(2s-1) \, A_\text{hyp2/sct}(s;m)$ are meromorphic over the complex plane 
and their poles are all simple with integral residues.
So, we consider the function:
\[ g(s;m) := -(2s-1) \Bigl( A_\text{id}(s) + A_\text{ell}(s;m) \Bigr)
\]
We see that $g(s;m)$ is also meromorphic and only have simple poles at $s=-k$ 
for $k \in \N\cup\{0\}$.
By the identity
\[ \frac{1}{\nu} \sum_{l=0}^{\nu-1} \psi \Bigl( \frac{s+l}{\nu} \Bigr)
= \psi(s) - \log \nu,
\]
we have
\begin{align*}
E_j(s) :=& \frac{1}{2s-1} \sum_{l=0}^{\nu_j-1} \frac{\nu_j-1-\alpha_l(m,j)-\overline{\alpha}_l(m,j)}{\nu_j^2}
\psi \Bigl( \frac{s+l}{\nu_j} \Bigr) \\
=& 
\frac{1}{\nu_j} \Bigl\{  \psi(s)  - \log \nu_j \Bigr\}
+\frac{1}{2s-1} \sum_{l=0}^{\nu_j-1}
\frac{\nu_j -2s-\alpha_l(m,j) - \overline{\alpha}_l(m,j)}{\nu_j^2} \psi \Bigl( \frac{s+l}{\nu_j} \Bigr). 
\end{align*}
Thus, for $k \in \N \cup \{ 0 \}$, (we write $k= l + \nu_j n$ with $l=0,1,\dots ,\nu_j-1$)
\[ 
- \Res_{s=-k} (2s-1) \, E_j(s)  =
\frac{2k+1}{\nu_j}  + 2 \Bigl[ \frac{k}{\nu_j} \Bigr] + 1- \frac{\alpha_l(m,j) + \overline{\alpha}_l(m,j) - 2l}{\nu_j},
\]
with $l=k-\nu_j[k/\nu_j]$.
Put 
\begin{equation} \label{eq:beta}
\beta_{k,j}(m) := \frac{\alpha_l(m,j) + \overline{\alpha}_l(m,j) - 2l}{\nu_j}
\quad \mbox{ with } \quad  l =  k-\nu_j \Bigl[ \frac{k}{\nu_j} \Bigr].
\end{equation}
We see that $\beta_{k,j}(m) \in \Z$ since $\alpha_l(m,j) + \overline{\alpha}_l(m,j)  \equiv 2l \pmod{\nu_j}$
by (\ref{eq:alpha}).

Therefore, we have
\begin{align*}
\Res_{s=-k} g(s;m) =&
(2k+1) \zeta_{K}(-1) - (2k+1) \sum_{j=1}^{N} \frac{1}{\nu_j} 
+ \sum_{j=1}^{N} (2 \Bigl[ \frac{k}{\nu_j} \Bigr]+1) - \sum_{j=1}^{N} \beta_{k,j}(m) \\
=& (2k+1) \, E(X_K) + 2 \sum_{j=1}^{N} \Bigl[ \frac{k}{\nu_j} \Bigr] - 2kN - \sum_{j=1}^{N} \beta_{k,j}(m).
\end{align*}
Here, $E(X_{K})$ is the Euler characteristic of the Hilbert modular surface $X_K$ and 
we used the formula (see \cite[Theorem (1.2), p.60]{Ge}):
\[  E(X_K) = 2 \zeta_{K}(-1)+\sum_{j=1}^{N} \frac{\nu_j-1}{\nu_j}. \]
Hence the residues of $g(s;m)$ are all integers.
The rest of proof is clear.  
 \end{proof}

\subsection{Functional equation of Selberg type zeta functions}

\begin{theorem} \label{th:a2}
Let $m \ge 4$ be an even integer.
The function $Z_K(s;m)$ satisfies the 
following functional equation
\[ \hat{Z}_K(s;m) = \hat{Z}_K(1-s;m). \]

Here the completed zeta function $\hat{Z}_K(s,m)$ is given by
\[ \hat{Z}_K(s;m) := Z_K(s;m)  
\, Z_{\mathrm{id}}(s) \, Z_{\mathrm{ell}}(s;m) 
\, Z_{\mathrm{hyp2/sct}}(s;m)  \]
with
\begin{align*}
Z_{\mathrm{id}}(s) &:= \Bigl( \Gamma_2(s) \Gamma_2(s+1) \Bigr)^{2 \, \zeta_{K}(-1)} \\ 
Z_{\mathrm{ell}}(s;m) &:= \prod_{j=1}^{N} \prod_{l=0}^{\nu_j -1}
\Gamma \bigl( \tfrac{s+l}{\nu_j} \bigr)^{\frac{\nu_j -1 
- \alpha_l(m,j) - \overline{\alpha_l}(m,j) }{\nu_j}}\\ 
Z_{\mathrm{hyp2/sct}}(s;m) &:= \zeta_{\varepsilon}(s + \tfrac{m}{2}-1) \, 
\zeta_{\varepsilon}(s+\tfrac{m}{2}-2)^{-1},
\end{align*}
where, 
$\Gamma_2(z)$ is the double Gamma function 
$($for definition, we refer to \cite{KK} or
\cite[Definition 4.10, p.751]{GP2}$)$,   
$\nu_1,\nu_2,\cdots,\nu_{N} $ are the orders  
of the elliptic fixed points
in $X_K$ and the integers
$\alpha_l(m,j), \, \overline{\alpha_l}(m,j) \in \{0,1, \cdots, \nu_j -1 \}$ 
was defined in $($\ref{eq:alpha}$)$, 
$\zeta_{\varepsilon}(s) := (1 - \varepsilon^{-2s} )^{-1}$ and
$\varepsilon$ is the fundamental unit of $K$.
\end{theorem}
\begin{proof}
Starting from the formula in Theorem \ref{th:ddtrfa}, we compute the difference of the both sides 
at $s$ and $1-s$. We see that 
\begin{align*}
& 2 \zeta_{K}(-1) \cdot (2s-1) \,   \pi \cot (\pi s)
+
\biggl( {\frac{Z_{K}'(s;m)}{Z_{K}(s;m)}}
+        {\frac{Z_{K}'(1-s;m)}{Z_{K}(1-s;m)}}
\biggr) \\
& \quad + 
\biggl( {\frac{Z_\text{ell}'(s)}{Z_\text{ell}(s)}}
+        {\frac{Z_\text{ell}'(1-s)}{Z_\text{ell}(1-s)}}
\biggr) 
+  
\biggl( {\frac{Z_\text{hyp2/sct}'(s)}{Z_\text{hyp2/sct}(s)}}
+        {\frac{Z_\text{hyp2/sct}'(1-s)}{Z_\text{hyp2/sct}(1-s)}}
\biggr) 
\\
&=0,
\end{align*}
by the partial fractional expansion: 
$ \pi \cot (\pi s)  =  \sum_{k=0}^{\infty} \bigl[ \frac{1}{s+k}  - \frac{1}{1-s+k} \bigr]
$.
Let $I(s) := 2 \zeta_{K}(-1) \cdot (2s-1) \,   \pi \cot (\pi s)$.
It is known that the double sine function $S_2(z):= \Gamma_2(2-z) \, \Gamma_2(z)^{-1}$
satisfies the differential equation (see \cite[Theorem 2.15, p.860]{KK}):
\[ \frac{d}{dz} \log S_2(z) = - \pi (z-1) \cot (\pi z). \] 
Therefore, 
\begin{align*}
I(s) =& - 2\zeta_{K}(-1) \, \frac{d}{ds} \log \Bigl( S_2(s) S_2(s+1) \Bigr) 
     =  - 2\zeta_{K}(-1) \, \frac{d}{ds} \log \Bigl( 
\frac{\Gamma_2(2-s)}{\Gamma_2(s)} \frac{\Gamma_2(1-s)}{\Gamma_2(s+1)}
\Bigr) \\
=&  {\frac{Z_\text{id}'(s)}{Z_\text{id}(s)}}
+        {\frac{Z_\text{id}'(1-s)}{Z_\text{id}(1-s)}}.
\end{align*}
Integrating and exponentiating, we obtain the desired functional equation.
\end{proof}

%%%%%%%%%%%%%%%%%%%%%%%%%%%%%%%%%%%%%%%%%%%%%%%%%%%%%%%%%%%%%%%%%%%%%%%%%%%%%%%%%%%%%%%%%%%%%%%
%%%%%%%%%%%%%%%%%%%%%%%%%%%%%%%%%%%%%%%%%%%%%%%%%%%%%%%%%%%%%%%%%%%%%%%%%%%%%%%%%%%%%%%%%%%%%%%
\section{Ruelle type zeta functions and applications} 
%%%%%%%%%%%%%%%%%%%%%%%%%%%%%%%%%%%%%%%%%%%%%%%%%%%%%%%%%%%%%%%%%%%%%%%%%%%%%%%%%%%%%%%%%%%%%%%
%%%%%%%%%%%%%%%%%%%%%%%%%%%%%%%%%%%%%%%%%%%%%%%%%%%%%%%%%%%%%%%%%%%%%%%%%%%%%%%%%%%%%%%%%%%%%%%

\subsection{Ruelle type zeta functions}

We consider the following Ruelle type zeta function
\begin{definition}[Ruelle type zeta function for $\Gamma_{K}$]
\label{def:ruellezeta}
For $\RE(s)>1$, the Ruelle type zeta function for 
$\Gamma_K$ is defined by the following absolutely convergent
Euler product:
\[
R_{K}(s) := \prod_{(p,p')} 
\bigl( 1 - N(p)^{-s} \bigr)^{-1}. 
\]
Here, $(p,p')$ run through the set of primitive hyperbolic-elliptic 
$\Gamma_K$-conjugacy classes of $\Gamma_K$, and $(p,p')$ is conjugate 
in $\mathrm{PSL}(2,\R)^2$ to
\[ (p, p') \sim \Bigl( 
\Bigl( 
\begin{array}{cc}
N(p)^{1/2} & 0 \\
0 & N(p)^{-1/2}
\end{array} \Bigr), 
\, 
\Bigl( 
\begin{array}{cc}
\cos \omega & - \sin \omega \\
\sin \omega & \cos \omega
\end{array} \Bigr)\Bigr).
\]
Here, $N(p)>1$, $\omega \in (0,\pi)$ and $\omega \notin \pi \Q$.
\end{definition}

We note that the following relation between
the Ruelle type zeta function and the Selberg type zeta function
for $\Gamma_K$.
\begin{lemma} \label{lem:ruelle}
For $\RE(s)>1$, we have
\[ R_{K}(s) = \frac{Z_{K}(s;2)}{Z_{K}(s+1;2)} .\]
\end{lemma}
\begin{proof}
For $\RE(s)>1$, we have
\[ \frac{Z_{K}(s;2)}{Z_{K}(s+1;2)} 
 = \frac{\prod_{(p,p')} 
\prod_{k=0}^{\infty} ( 1 - N(p)^{-(s+k)} )^{-1} }
{\prod_{(p,p')}\prod_{k=0}^{\infty} ( 1 - N(p)^{-(s+k+1)} )^{-1}}
=R_{K}(s).
\]
\end{proof}
To get a meromorphic extension of $R_K(s)$, 
we consider meromorphic extension of the
Selberg type zeta function $Z_{K}(s;2)$.
For this, we recall Theorem \ref{th:ddtrf}, 
the double differences of the trace formula
for the weight $(0,2)$.

\begin{corollary}[Double differences of STF 
for $L^2 \bigl( \Gamma_{K} \backslash \bH^2 \, ; \, (0,2) \bigr)$]
\label{cor:ddtrf2}
Let $m = 2$. We have 
\begin{align*}
& \sum_{j=0}^{\infty} h_1 \Bigl( \rho_j(2) \Bigr)  
   + \sum_{j=0}^{\infty} h_1 \Bigl( \mu_j(-2) \Bigr)  
- 2 h_1 \Bigl( \frac{i}{2} \Bigr) \\
 & \quad =  \:
\frac{\mathrm{vol}(\Gamma_K \backslash \bH^2)}{8 \pi^2}
\int_{-\infty}^{\infty}
  r h_1(r) \tanh (\pi r)  \, dr  \\
& \quad \quad \quad - \sum_{R(\theta_1,\theta_2) \in 
\Gamma_{\mathrm{E}}} 
\frac{i e^{ - i \theta_1}}{4 \nu_{R} \sin \theta_1}
\int_{-\infty}^{\infty} g_1(u) \, e^{-u/2} \biggl[  \frac{e^u - e^{2 i \theta_1}}{\cosh u - \cos 2 \theta_1} \biggr] 
du \\
& \quad \quad \quad
 - \sum_{(\gamma,\omega) \in \Gamma_{\mathrm{HE}}} 
\frac{\log N(\gamma_0) \, g_1(\log N(\gamma))   }{N(\gamma)^{1/2}-N(\gamma)^{-1/2}} 
- 2 \log \varepsilon \, g_1(0)
\, - 4 \log \varepsilon \, 
\sum_{k=1}^{\infty} g_1(2k \log \varepsilon) \, \varepsilon^{-k}. 
\end{align*}
Here,  $\{ 1/4+ \rho_j(2)^2 \}_{j=0}^{\infty}$ and
 $\{ 1/4+ \mu_j(-2)^2 \}_{j=0}^{\infty}$ are the sets of 
eigenvalues of the Laplacian $\Delta_{0}^{(1)}$ acting 
on $\Ker(\Lambda_{2}^{(2)})$ and $\Ker(K_{-2}^{(2)})$
respectively.
These kernel spaces are given by
\[ \Ker(\Lambda_{2}^{(2)}) =
\biggl\{  f \in  L^2_\text{dis} \Bigl( \Gamma_{K} \backslash \bH^2 
\, ; \, (0,2) \Bigr)  \, \Bigl| \,  \Delta_{2}^{(2)}  f  = 0
\biggr\}, \]
\[ \Ker(K_{-2}^{(2)}) =
\biggl\{  f \in  L^2_\text{dis} \Bigl( \Gamma_{K} \backslash \bH^2 
\, ; \, (0,-2) \Bigr)  \, \Bigl| \,  \Delta_{-2}^{(2)} \, f  = 0
\biggr\}. \]
\end{corollary}
\begin{proof}
By noting that $\Ker(\Lambda_{0}^{(2)}) = \Ker(K_{0}^{(2)})=\C$, the rest is clear 
by Theorem \ref{th:ddtrf}.
\end{proof}

\begin{theorem}[Double differences of STF for the test function $h_1$ 
(see (\ref{eq:test}))
with the weight $(0,2)$] \label{th:b0}
\begin{align*}
& \sum_{j=0}^{\infty} \Bigl[ \frac{1}{\rho_j(2)^2+(s-\frac{1}{2})^2} 
 + \sum_{h=1}^{2}\frac{c_h(s)}{\rho_j(2)^2 + \beta_h^2} \Bigr]
 + \sum_{j=0}^{\infty} \Bigl[ \frac{1}{\mu_j(-2)^2+(s-\frac{1}{2})^2} 
 + \sum_{h=1}^{2} \frac{c_h(s)}{\mu_j(-2)^2 + \beta_h^2} \Bigr] \\
& \quad - 
2 \Bigl[ \frac{1}{s(s-1)} 
 + \sum_{h=1}^{2}\frac{c_h(s)}{\beta_h^2-\frac{1}{4}} \Bigr] \\
& = 2 \zeta_{K}(-1) \sum_{k=0}^{\infty}
\Bigl[ \frac{1}{s+k} 
+ \sum_{h=1}^{2} \frac{c_h(s)}{\beta_h+\frac{1}{2}+k} \Bigr]  
+ \frac{1}{2s-1} {\frac{Z_{K}'(s;2)}{Z_{K}(s;2)}}
+\sum_{h=1}^{2} \frac{c_h(s)}{2 \beta_h} 
\frac{Z_{K}'(\frac{1}{2}+\beta_h;2)}{Z_{K}(\frac{1}{2}+\beta_h;2)}
\\
& \quad + \frac{1}{2s-1} 
\sum_{j=1}^{N} \sum_{l=0}^{\nu_j-1} \frac{\nu_j-1-2l}{\nu_j^2}
\, \psi \Bigl( \frac{s+l}{\nu_j} \Bigr)
+ \sum_{h=1}^{2} \frac{c_h(s)}{2 \beta_h} 
\sum_{j=1}^{N} \sum_{l=0}^{\nu_j-1} \frac{\nu_j-1-2l}{\nu_j^2}
\, \psi \Bigl( \frac{\frac{1}{2}+\beta_h+l}{\nu_j} \Bigr)
\\
& \quad + \frac{1}{2s-1} \frac{d}{ds} 
\log \bigl(  \varepsilon^{-2s}   \bigr) 
+ \sum_{h=1}^{2} \frac{c_h(s)}{2 \beta_h} \frac{d}{d \beta_h} 
\log \bigl( 
\varepsilon^{-(2 \beta_h +1)}
\bigr) \\
& \quad + \frac{1}{2s-1} \frac{d}{ds} 
\log \biggl\{ \frac{1}
{(1-\varepsilon^{-2s})^2} \biggr\} 
+ \sum_{h=1}^{2} \frac{c_h(s)}{2 \beta_h} \frac{d}{d \beta_h} 
\log \biggl\{ 
\frac{1}
{(1-\varepsilon^{-(2 \beta_h +1)})^2}
\biggr\}.
\end{align*}
\end{theorem}
\begin{proof}
By Corollary \ref{cor:ddtrf2} and the same computation in Theorem \ref{th:ddtrfa}.
\end{proof}

\begin{theorem} \label{th:b1}
The Selberg zeta function
$Z_K(s;2)$, originally defined for 
$\RE(s) > 1$, has an analytic continuation to the whole 
complex plane as a meromorphic function. 
\begin{enumerate}
\item $Z_{K}(s;2)$ has a double pole at $s=1$. 
\item $Z_{K}(s;2)$ has zeros at  \\
$s=\frac{1}{2} \pm i \rho_j(2)$ of order equal to the multiplicity 
of the eigenvalue $\frac{1}{4}+\rho_j(2)^2$ of $\Delta_0^{(1)}$ acting
on $\Ker(\Lambda_{2}^{(2)})=
\Bigl\{  f \in  L^2_\text{dis} \Bigl( \Gamma_{K} \backslash \bH^2 
\, ; \, (0,2) \Bigr)  \, \Bigl| \,  \Delta_{2}^{(2)}  f  = 0
\Bigr\}$, \\
$s=\frac{1}{2} \pm i \mu_j(-2)$ of order equal to the multiplicity 
of the eigenvalue $\frac{1}{4}+\mu_j(-2)^2$ of $\Delta_0^{(1)}$ acting
on $\Ker(K_{-2}^{(2)})=
\Bigl\{  f \in  L^2_\text{dis} \Bigl( \Gamma_{K} \backslash \bH^2 
\, ; \, (0,-2) \Bigr)  \, \Bigl| \,  \Delta_{-2}^{(2)} \, f  = 0
\Bigr\}$.
\item $Z_{K}(s;2)$ has zeros at  
$s = \pm \frac{k \pi i}{\log \varepsilon}$ $(k \in \N)$ of order 2.
\item $Z_{K}(s;2)$ has a zero at  
$s=0$ of order $E(X_K)$.  
\item $Z_{K}(s;2)$ has zeros or poles 
(according to their orders are positive or negative) 
at  
$s=-k$ $(k \in \N)$ of order $(2k+1)E(X_K)+2\sum_{j=1}^{N}[k/\nu_j]-2kN$.
\end{enumerate}
Here, 
$E(X_{K})$ denotes the Euler characteristic of the Hilbert modular surface $X_{K}$.
When the location of two zeros or poles coincide, 
the orders of them are added.
\end{theorem}
\begin{proof}
By Theorem \ref{th:b0} and the same proof of Theorem \ref{th:a1}.
\end{proof}

\begin{theorem} \label{th:b2}
The Selberg type zeta function $Z_K(s;m)$ satisfies the 
following functional equation
\[ \hat{Z}_K(s;2) = \hat{Z}_K(1-s;2). \]
Here the completed zeta function $\hat{Z}_K(s,m)$ is given by
\[ \hat{Z}_K(s;2) := Z_K(s;2) \, 
 Z_{\mathrm{id}}(s) \, Z_{\mathrm{ell}}(s;2) 
\, Z_{\mathrm{par/sct}}(s;2) \, Z_{\mathrm{hyp2/sct}}(s;2) \]
with
\begin{align*}
& Z_{\mathrm{id}}(s) := \Bigl( \Gamma_2(s) \Gamma_2(s+1) \Bigr)^{2 \, \zeta_{K}(-1)}, 
\quad
Z_{\mathrm{ell}}(s;2) := \prod_{j=1}^{N} \prod_{l=0}^{\nu_j -1}
\Gamma \bigl( \tfrac{s+l}{\nu_j} \bigr)^{\frac{\nu_j -1 -2l }{\nu_j}}, \\ 
& \quad Z_{\mathrm{par/sct}}(s;2) := \varepsilon^{-2s}, \quad  
Z_{\mathrm{hyp2/sct}}(s;2) := \zeta_{\varepsilon}(s)^2 
= (1-\varepsilon^{-2s})^{-2}.
\end{align*}
\end{theorem}
\begin{proof}
By using Theorem \ref{th:b0}, the proof is the same 
as in Theorem \ref{th:a1}.
\end{proof}

\begin{theorem} \label{th:Ruelle}
The function $R_{K}(s)$ has a meromorphic continuation to the 
whole $\C$. $R_{K}(s)$ has double pole at $s=1$ and nonzero for 
$\RE(s) \ge 1$.
\end{theorem}
\begin{proof}
By Theorem \ref{th:b1} and Lemma \ref{lem:ruelle}.
\end{proof}

\begin{theorem} \label{th:ruelle-fe}
The function $R_{K}(s)$ has the following functional equation
\begin{equation}
\begin{split}
R_{K}(s) \, R_{K}(-s) =& (-1)^{E(X_K)} \, 
\bigl( 2 \sin (\pi s) \bigr)^{2 E(X_K)} 
\prod_{j=1}^{N} \biggl( \frac{\sin (\pi s/\nu_j)}{\sin(\pi s)}\biggr)^2 \\
& \cdot 
\biggl( \frac{\zeta_{\ve}(s-1) \, \zeta_{\ve} (s+1)}{\zeta_{\ve}(s)^2} \biggr)^2,
\end{split}
\end{equation}
where, 
$\zeta_\ve(s)=(1-\ve^{-2s})^{-1}$, $N$ is the number of elliptic fixed points 
in $X_K$.
\end{theorem}
\begin{proof}
By Theorem \ref{th:b2}, we have
\begin{align*}
R_{K}(s) R_{K}(-s)
 & = \frac{Z_{K}(s;2)}{Z_{K}(s+1;2)} \frac{Z_{K}(-s;2)}{Z_{K}(-s+1;2)} 
    = \frac{Z_{K}(s;2)}{Z_{K}(1-s;2)} \frac{Z_{K}(-s;2)}{Z_{K}(1+s;2)}  \\
 & =  B_{K}(s) \, C_{K}(s). 
\end{align*}
The functions $B_{K}(s)$ and $C_{K}(s)$ are given as follows.
\[
B_{K}(s) :=
\frac{Z_\text{id}(1+s) \, Z_\text{id}(1-s)}{ Z_\text{id}(s) \, Z_\text{id}(-s)}
\frac{Z_\text{ell}(1+s;2) \, Z_\text{ell}(1-s;2)}{ Z_\text{ell}(s;2) \, Z_\text{ell}(-s;2)},
\]
\[
C_{K}(s) := 
\frac{Z_\text{par/sct}(1+s;2) \, Z_\text{par/sct}(1-s;2)}{ Z_\text{par/sct}(s;2) \, Z_\text{par/sct}(-s;2)}
\frac{Z_\text{hyp2/sct}(1+s;2) \, Z_\text{hyp2/sct}(1-s;2)}{ Z_\text{hyp2/sct}(s;2) \, Z_\text{hyp2/sct}(-s;2)}. 
\]
We can easily check that
\[ C_{K}(s) =  \biggl( \frac{\zeta_{\ve}(s-1) \, \zeta_{\ve} (s+1)}{\zeta_{\ve}(s)^2} \biggr)^2
.\]
Let us compute $B_{K}(s)$. 
Put
\[ \Xi(s) := \frac{\Gamma_2(s+1)\Gamma_2(s+2) \Gamma_2(1-s)\Gamma_2(2-s)}
{\Gamma_2(s)\Gamma_2(s+1) \Gamma_2(-s)\Gamma_2(1-s)}
\]
and
\[ G_{\nu}(s) := \prod_{j=0}^{\nu -1} \Gamma \Bigl( \frac{s+l}{\nu} \Bigr)^{\frac{\nu-1-2l}{\nu}}.
\]
Then we see that
\begin{equation} \label{eq:bk1}
B_{K}(s) = \Xi(s)^{E(X_{K})} 
  \prod_{j=1}^{N} \bigg[ \frac{G_{\nu_j}(1+s) \, G_{\nu_j}(1-s)}{G_{\nu_j}(s) \, G_{\nu_j}(-s)}  
   \, \Xi(s)^{-\frac{\nu_j -1}{\nu_j}}  \biggr].
\end{equation}
By using $\Gamma_2(s+1)/\Gamma_2(s) = \sqrt{2 \pi} \, \Gamma(s)^{-1}$ 
(see \cite{KK} or \cite[Proposition 4.11]{GP2}), we have
\begin{equation} \label{eq:bk2}
\Xi(s) = (2 \pi)^2 \bigl( \Gamma(s) \, \Gamma(s+1) \, \Gamma(-s) \, \Gamma(-s+1) \bigr)^{-1}
= -4 \sin^2(\pi s). 
\end{equation} 
By using the multiplication formula for the Gamma function
(see \cite[8.335]{GR}), we have 
\begin{align*}
G_{\nu}(1+s) \, G_{\nu}(s)^{-1}
 &=  \Gamma \Bigl( \frac{s}{\nu} \Bigr)^{-\frac{\nu-1}{\nu}}
\biggl[  \prod_{l=1}^{\nu-1} \Gamma \Bigl( \frac{s+l}{\nu} \Bigr)^{\frac{2}{\nu}}
\biggr]  
  \Gamma \Bigl( \frac{s+\nu}{\nu} \Bigr)^{\frac{1-\nu}{\nu}}  \\
 &= \Gamma \Bigl( \frac{s}{\nu} \Bigr)^{\frac{1-\nu}{\nu}}
 \Gamma \Bigl( 1+\frac{s}{\nu} \Bigr)^{\frac{1-\nu}{\nu}}
 \biggl[ \Gamma \Bigl( \frac{s}{\nu} \Bigr)^{-1}
(2 \pi)^{\frac{\nu-1}{2}} \nu^{1/2-s} \, \Gamma (s) 
\biggr]^{\frac{2}{\nu}}.
\end{align*}
Hence, we have
\begin{equation} \label{eq:bk3}
\begin{split}
&\frac{G_{\nu}(1+s)}{G_{\nu}(s)} \frac{G_{\nu}(1-s)}{G_{\nu}(-s)}
\, \Xi(s)^{-\frac{\nu-1}{\nu}} \\
&=\Gamma \Bigl( \frac{s}{\nu} \Bigr)^{\frac{-1-\nu}{\nu}}
 \Gamma \Bigl( 1+\frac{s}{\nu} \Bigr)^{\frac{1-\nu}{\nu}}
 \Gamma \Bigl( \frac{-s}{\nu} \Bigr)^{\frac{-1-\nu}{\nu}}
 \Gamma \Bigl( 1+\frac{-s}{\nu} \Bigr)^{\frac{1-\nu}{\nu}}
 \biggl[ \nu \, \Gamma (s) \, \Gamma(-s) \biggr]^{\frac{2}{\nu}} \\
 & \quad 
 \times \bigl( \Gamma(s) \, \Gamma(s+1) \, \Gamma(-s) \, \Gamma(-s+1) \bigr)^{\frac{\nu-1}{\nu}}
 \\
 &=\biggl[ \frac{\Gamma(s) \, \Gamma(1-s)}{\Gamma(s/\nu) \, \Gamma(1-\nu/s)} \biggr]^2 
   = \biggl( \frac{\sin(\pi s/\nu)}{\sin(\pi s)} \biggr)^2.
\end{split}
\end{equation}
Substituting (\ref{eq:bk2}) and (\ref{eq:bk3}) into (\ref{eq:bk1}), 
we complete the proof.
\end{proof}

We can obtain an explicit formula of the leading term of $R_{K}(s)$
at $s=0$. Let $n_0$ denote an integer such that 
$\lim_{s \to 0} s^{-n_0}R_{K}(s)$ is a nonzero finite value and
\[ R_{K}^{*}(0) := \lim_{s \to 0} s^{-n_0}R_{K}(s).
\]

\begin{theorem}
The following equalities hold.
\[ n_0 =  E(X_K) + 2 \]  
and
\[ |R_{K}^{*}(0)| = (2 \pi)^{E(X_{K})} \, \prod_{j=1}^{N} \nu_j^{-1} \, 
\frac{(2 \ve \log \ve)^2}{(\ve^2-1)^2}.
\]
\end{theorem}
\begin{proof}
By Theorem \ref{th:ruelle-fe}, we can compute
\[
\lim_{s \to 0} \frac{R_{K}(s)R_{K}(-s)}{s^{2(E(X_K)+2)}} 
= (-1)^{E(X_{K})} (2 \pi)^{2E(X_{K})} \, \prod_{j=1}^{N}  \nu_j^{-2} \, 
\biggl( \frac{(2 \log \ve)^2}{(1-\ve^2)(1-\ve^{-2})} \biggr)^2.
\]
The rest is clear.
\end{proof}

\begin{corollary} \label{cor:ruelle}
Let $D$ be the discriminant of $K$ and $D \ge 13$. Then,   
the function $R_{K}(s)$ satisfy the functional equation
\begin{equation}
\begin{split}
R_{K}(s) \, R_{K}(-s)=&(-1)^{E(X_K)} \, 2^{2E(X_K)}
\sin (\pi s)^{2E(X_K) -2a_2(\Gamma) -2 a_3(\Gamma)} \\
& \cdot \sin \Bigl( \frac{\pi s}{2} \Bigr)^{2a_2(\Gamma)} 
 \sin \Bigl( \frac{\pi s}{3} \Bigr)^{2 a_3(\Gamma)} 
 \biggl( \frac{\zeta_{\ve}(s-1) \, \zeta_{\ve} (s+1)}{\zeta_{\ve}(s)^2} \biggr)^2
\end{split}
\end{equation}
and the absolute value of the coefficient of the leading term of $R_{K}(s)$ at $s=0$ is 
given by
\[ |R_{K}^{*}(0)|
  = \frac{(2 \pi)^{E(X_{K})}}{2^{a_2(\Gamma)} \, 3^{a_3(\Gamma)}} 
\frac{(2 \ve \log \ve)^2}{(\ve^2-1)^2}.
\]
Here, $a_r(\Gamma)$ is the number of elliptic fixed points
in $X_K$ for which  corresponding points have isotropy  groups
of order $r$.  
\end{corollary}
We remark that $a_2(\Gamma)$ and $a_3(\Gamma)$ are described by 
the class numbers of certain imaginary quadratic fields. (Cf. Prestel \cite{Pre})
\begin{proof}
By the fact that
unless $D=5,8,12$, then the $a_r(\Gamma)$ vanish for $r>3$.
(Cf. \cite[p.16]{Ge})
\end{proof}

\begin{corollary} \label{cor:5812}
Let $D=5,8$ or $12$. 
Then the class number of the real quadratic field 
$K$ with the given discriminant $D$ is one and 
the absolute value of the coefficient of the leading term of $R_{K}(s)$ at $s=0$ is 
\begin{itemize}
\item  \, $(2 \pi)^4 (2^23^25^2)^{-1} (2 \ve \log \ve)^2/(\ve^2-1)^2$ with \, $\ve=(1+\sqrt{5})/2$ \, if $D=5$,
\item  \, $(2 \pi)^4 (2^23^24^2)^{-1} (2 \ve \log \ve)^2/(\ve^2-1)^2$ with \, \, $\ve= 1+\sqrt{2}$ \, \,  if $D=8$,
\item  \, $(2 \pi)^4 (2^33^26^1)^{-1} (2 \ve \log \ve)^2/(\ve^2-1)^2$ with \, \, $\ve=2+\sqrt{3}$ \, if $D=12$.
\end{itemize}
\end{corollary}
\begin{proof}
Use the table of $a_r(\Gamma)$ for $D=5,8$ or $12$ on \cite[p.268]{Ge}.
\end{proof}

%%%%%%%%%%%%%%%%%%%%%%%%%%%%%%%%%%%%%%%%%%%%%%%%%%%%%%%%%%%%%%%%
\subsection{Weyl's law} 
%%%%%%%%%%%%%%%%%%%%%%%%%%%%%%%%%%%%%%%%%%%%%%%%%%%%%%%%%%%%%%%%

As an application of the double difference of the trace formula for $\Gamma_K$ with the weight
$(0,2)$ (Corollary \ref{cor:ddtrf2}), we have the following ``Weyl's law".

\begin{proposition}[Weyl's law I] \label{prop:weyl}
Let $T>0$. 
We consider the following counting function:
\[ N(T) := \# \bigl\{ j \, \big| \,  1/4 + \rho_j(2)^2 \le T \bigr\}
+\# \bigl\{ j \, \big| \,   1/4+ \mu_j(-2)^2 \le T  \bigr\}
.\]
Then we have
\begin{equation}
N(T) \sim \frac{\mathrm{vol}(\Gamma_K \backslash \bH^2)}{8 \pi^2}
\, T \quad (T \to \infty). 
\end{equation}
\end{proposition}
\begin{proof}
For any $\beta>0$, the test function $h_1(r) = e^{- \beta r^2}$
is admissible in Corollary \ref{cor:ddtrf2}. The Fourier transform is
\[ g_1(u) = \frac{e^{-u^2/(4 \beta)}}{\sqrt{4 \pi \beta}}
,\]
so we have
\begin{align*}
 & \sum_{j=0}^{\infty}e^{-\beta(1/4+\rho_j(2)^2)}
+\sum_{j=0}^{\infty}e^{-\beta(1/4+\mu_j(-2)^2)}
\, -2 \\
& = \frac{\mathrm{vol}(\Gamma_K \backslash \bH^2)}{8 \pi^2}
\int_{-\infty}^{\infty} e^{-\beta(1/4+r^2)} r \tanh(\pi r) \, dr \\
& \quad   -  \frac{e^{-\beta/4}}{\sqrt{4 \pi \beta}} \sum_{R(\theta_1,\theta_2) \in 
\Gamma_{\mathrm{E}}} 
\frac{i e^{ - i \theta_1}}{4 \nu_{R} \sin \theta_1}
\int_{-\infty}^{\infty} e^{-u^2/(4 \beta)} \, e^{-u/2} \biggl[  \frac{e^u - e^{2 i \theta_1}}{\cosh u - \cos 2 \theta_1} \biggr] 
du \\
& \quad 
 - \frac{e^{-\beta/4}}{\sqrt{4 \pi \beta}} \sum_{(\gamma,\omega) \in \Gamma_{\mathrm{HE}}} 
\frac{\log N(\gamma_0) }{N(\gamma)^{1/2}-N(\gamma)^{-1/2}} 
\,  e^{ - (\log N(\gamma))^2/(4 \beta) }  \\
& \quad - 
\frac{e^{-\beta/4}}{\sqrt{4 \pi \beta}}
\biggl\{
2 \log \varepsilon 
\, + 4 \log \varepsilon \, 
\sum_{k=1}^{\infty} e^{-(2k \log \varepsilon)^2/(4 \beta)} \, \varepsilon^{-k} \biggr\}. 
\end{align*}
Since $\tanh(\pi r)=1+O(e^{-2 \pi|r|})$ for any $r \in \R$,
we obtain
\[ \sum_{j=0}^{\infty}e^{-\beta(1/4+\rho_j(2)^2)}
+\sum_{j=0}^{\infty}e^{-\beta(1/4+\mu_j(-2)^2)}
= \frac{\mathrm{vol}(\Gamma_K \backslash \bH^2)}{8 \pi^2 \, \beta}
- \frac{2 \log \ve}{\sqrt{4 \pi \beta}} +O(1)
\quad (\beta \to +0).
\]
By a classical Tauberian theorem, we complete the proof.
\end{proof}
We remark that the above proposition  is enough to prove ``a prime geodesic type theorem"  
for the set of primitive hyperbolic-elliptic conjugacy classes of $\Gamma_K$
in the next subsection. 

Besides we can prove a more strong Weyl's law by using 
Theorem \ref{th:dtrf},
the differences (not double differences) of the trace formula
for $\Gamma_{K}$ with the weight $(0,m)$.

\begin{theorem}[Weyl's law II] \label{th:weyl2}
Let $m \in 2\Z$ and $T>0$.
We consider the following two counting functions:
\begin{align*}
& N_m^{+}(T) := \# \bigl\{ j \, \big| \,  1/4 + \rho_j(m)^2 \le T \bigr\}, \\
& N_m^{-}(T) := \# \bigl\{ j \, \big| \,   1/4+ \mu_j(m-2)^2 \le T  \bigr\}.
\end{align*}

Then we have
\begin{align}
\mbox{If $m \ge 2$, then} \quad & N_m^{+}(T) \sim (m-1) \frac{\mathrm{vol}(\Gamma_K \backslash \bH^2)}{16 \pi^2}
\, T \quad (T \to \infty), \\
\mbox{If $m\le 0$, then} \quad & N_m^{-}(T) \sim (1-m) \frac{\mathrm{vol}(\Gamma_K \backslash \bH^2)}{16 \pi^2}
\, T \quad (T \to \infty). 
\end{align}
\end{theorem}
\begin{proof}
We take the test function $h_1(r) = e^{- \beta r^2}$ $(\beta>0)$
in Theorem \ref{th:dtrf}. 
Then we obtain (by the same computation as in the proof of Proposition \ref{prop:weyl})
\begin{align*} 
& \sum_{j=0}^{\infty}e^{-\beta(1/4+\rho_j(m)^2)}
-\sum_{j=0}^{\infty}e^{-\beta(1/4+\mu_j(m-2)^2)} \\
& \quad = (m-1) \frac{\mathrm{vol}(\Gamma_K \backslash \bH^2)}{16 \pi^2 \, \beta}
- \sgn(m-1) \frac{\log \ve}{\sqrt{4 \pi \beta}} +O(1)
\quad (\beta \to +0).
\end{align*}
The rest is clear.
\end{proof}

%%%%%%%%%%%%%%%%%%%%%%%%%%%%%%%%%%%%%%%%%%%%%%%%%%%%%%%%%%%%%%%%
\subsection{Prime geodesic theorem} 
%%%%%%%%%%%%%%%%%%%%%%%%%%%%%%%%%%%%%%%%%%%%%%%%%%%%%%%%%%%%%%%%

We can show the following asymptotic formulas for counting
functions of  $\mathrm{P\Gamma_{HE}}$, 
the set of primitive hyperbolic-elliptic 
$\Gamma_K$-conjugacy classes of $\Gamma_K$, by
Corollary \ref{cor:ddtrf2} and Proposition \ref{prop:weyl}.

\begin{theorem}[Prime geodesic theorem] \label{th:pgt}
For $X\ge 1$, 
\begin{equation} \label{eq:pgt1}
\sum_{\substack{(p,p')  \in \mathrm{P\Gamma_{HE}} \\
    N(p) \le X}}
\log N(p) = 2X
- \sum_{1/2< s_j(2) <1} \frac{X^{s_j(2)}}{s_j(2)}
- \sum_{1/2< s_j(-2) <1}  \frac{X^{s_j(-2)}}{s_j(-2)}
+O \bigl( X^{3/4}   \bigr), 
\end{equation}

\begin{equation} \label{eq:pgt2}
\begin{split}
\sum_{\substack{(p,p')  \in \mathrm{P\Gamma_{HE}}  \\
    N(p) \le X}}
1 =& 2 \Li(X) 
- \sum_{1/2< s_j(2) <1}  \Li \big( X^{s_j(2)} \bigr)
- \sum_{1/2< s_j(-2) <1}  \Li \bigl( X^{s_j(-2)} \bigr) \\
&+O \bigl( X^{3/4} /\log X   \bigr), 
\end{split}
\end{equation}
where, $s_j(2) := 1/2 -i \rho_j(2)$, $s_j(-2) := \frac{1}{2} - i \mu_j(-2)$,
and $\Li(x) := \int_{2}^{x} 1/\log t \, dt$.
$($The condition $1/2<s_j(2), \, s_j(-2)<1$ implies that
$\rho_j(2)^2$ and $\mu_j(-2)^2$ are negative. See $($\ref{eq:rj}$)$. $)$
\end{theorem}
\begin{proof}
We follow the same procedure as in Iwaniec \cite{Iw}. 
Let us begin with the same test function on \cite[p.155]{Iw}
or \cite[p.401]{IK}, given by  
\[ g_1(u) = 2 \cosh (u/2) \, q(u), \]
where, $q(u)$ is even, smooth, supported on $|u| \le \log(X+Y)$, and such that
$0 \le q(u) \le 1$ and $q(u)=1$ if $|x| \le \log X$. 
The parameters $X \ge Y \ge 1$ will be chosen later.
For $s=1/2+ir$ in the segment $1/2<s \le 1$ we have
\[ h_1(r) = \int_{-\infty}^{\infty}
\Bigl(  e^{su} + e^{(1-s)u} \Bigr) q(u) \, du
= s^{-1} X^s + O(Y+X^{1/2}),
\]
and for $s$ on the line $\RE(s)=1/2$ we get by partial integration that
\[ h_1(r)  \ll |s|^{-1} X^{1/2} \min \bigl\{ 1, |s|^{-2} T^2 \bigr\},
\]
where $T=XY^{-1}$. By using Proposition \ref{prop:weyl},
the spectral side of Corollary \ref{cor:ddtrf2} becomes
(by the same method as in \cite[pp.305--307]{Naud})
\begin{equation*}
\begin{split}
& \sum_{j=0}^{\infty} h_1 \Bigl( \rho_j(2) \Bigr)
 +\sum_{j=0}^{\infty} h_1 \Bigl( \rho_j(-2) \Bigr)
 -2 h_1(i/2) \\
& =  -2X +\sum_{1/2< s_j(2) <1} \frac{X^{s_j(2)}}{s_j(2)}
+ \sum_{1/2< s_j(-2) <1}  \frac{X^{s_j(-2)}}{s_j(-2)}
+O(Y+X^{1/2}T).
\end{split}
\end{equation*}
On the geometric side, the identity term contributes
\[ \frac{\mathrm{vol}(\Gamma_K \backslash \bH^2)}{8 \pi^2}
\int_{-\infty}^{\infty} r h_1(r) \tanh( \pi r ) \, dr
\ll X^{1/2} T
\]
and the elliptic term, 
the parabolic plus scattering and the type 2 hyperbolic plus
scattering terms contribute no more than the above bound.
Gathering these estimates, we arrive at
\begin{equation} \label{eq:pgt3}
\begin{split}
& - \sum_{(p,p')  \in \mathrm{P\Gamma_{HE}}}
q(\log N(p) ) \\
& = - 2X
+ \sum_{1/2< s_j(2) <1} \frac{X^{s_j(2)}}{s_j(2)}
+ \sum_{1/2< s_j(-2) <1}  \frac{X^{s_j(-2)}}{s_j(-2)}
+O \bigl( Y+X^{1/2} T \bigr).
\end{split}
\end{equation}
Subtracting (\ref{eq:pgt3}) from that for $X+Y$ in place of $X$, 
we deduce that
\[ \sum_{X < N(p)< X+Y} \log N(p) \ll Y+X^{1/2} T
.\]
Hence, we can drop the excess over $N(p) \le X$ within the error term
in (\ref{eq:pgt3}).
As usual we choose $Y=X^{3/4}$ to minimize the error term.
We completes the proof.
\end{proof}

%%%%%%%%%%%%%%%%%%%%%%%%%%%%%%%%%%%%%%%%%%%%%%%%%%%%%%%%%%%%%%%%
\subsection{Binary quadratic forms over the ring 
of real quadratic integers} 
%%%%%%%%%%%%%%%%%%%%%%%%%%%%%%%%%%%%%%%%%%%%%%%%%%%%%%%%%%%%%%%%

We denote by $\mathcal{D}$ the set of discriminants of integral binary quadratic forms, 
that is,   
\[ \mathcal{D} := \{ d \in \Z \, | \, d \equiv 0, \, 1
\pmod{4}, \, d \mbox{ not a square}, \, 
d > 0 \}
.\]
For each $d \in \mathcal{D}$, let $h(d)$ denote the number of inequivalent primitive
binary quadratic forms of discriminant $d$, and let $(x_d,y_d)$ be the fundamental solution
of the Pellian equation $x^2-dy^2=4$ over $\Z$. Put 
\[ \varepsilon_d: = \frac{x_d+\sqrt{d} \, y_d}{2}.
\]
By using the prime geodesic theorem for $\mathrm{PSL}(2,\Z)$, 
Sarnak \cite{S1} deduced the following theorem
on the average behavior of $h(d)$.  

\begin{theorem}[Sarnak {\cite[Theorem 2.1]{S1}}]
\label{th:sarnak}
For $x \ge 2$, we have
\[ \sum_{\substack{d \in \mathcal{D} \\
    \varepsilon_d \le x}}
h(d) \log \varepsilon_d = \frac{x^2}{2}  
+O \bigl( x^{3/2} (\log x)^3 \bigr)
\quad \quad (x \to \infty)
. \]
\[ 
\sum_{\substack{d \in \mathcal{D} \\
    \varepsilon_d \le x}}
h(d)
= \Li(x^2)  + O \bigl( x^{3/2} (\log x)^2 \bigr)
\quad \quad (x \to \infty)
.\]
Here, $\Li(x) = \int_{2}^{x} 1/\log t \, dt$.
\end{theorem}
There are several works to improve
the error term of the prime geodesic theorem for
$\mathrm{PSL}(2,\Z)$. We refer to \cite{SY}
for this subject.
  
Let us  consider a generalization of Theorem \ref{th:sarnak}
to that for class numbers of indefinite binary quadratic forms 
over the real quadratic integer ring $\mathcal{O}_{K}$.  
Put
\[ \mathcal{D}_{+-} := \{ d \in \mathcal{O}_{K} \, | \,  
\exists b \in \mathcal{O}_K \, \mbox{ s.t. } \, d \equiv b^2
\pmod{4}, \, d \mbox{ not a square in $\mathcal{O}_{K}$}, \, 
d > 0, \, d' < 0 \}
.\]
For each $d \in \mathcal{D}_{+-}$, let $h_{K}(d)$ denote the number of 
inequivalent primitive binary quadratic forms of discriminant $d$
over $\mathcal{O}_{K}$, 
and let $(x_d,y_d) \in \mathcal{O}_{K} \times \mathcal{O}_{K}$ 
be the fundamental solution of the Pellian equation $x^2-dy^2=4$. 
Put 
\[ \varepsilon_K(d): = \frac{x_d+\sqrt{d} \, y_d}{2}.
\]
By using Theorem \ref{th:pgt}, we can deduce the 
following theorem on
the average behavior of $h_{K}(d)$. 

\begin{theorem} \label{th:class-number}
For $x \ge 2$, we have
\begin{equation}
\begin{split}
\sum_{\substack{d \in \mathcal{D}_{+-} \\
    \varepsilon_K(d) \le x}}
h_{K}(d) \log \varepsilon_K(d) = & \, x^2
- \frac{1}{2} \sum_{1/2< s_j(2) <1} \frac{X^{2 s_j(2)}}{s_j(2)} 
- \frac{1}{2} \sum_{1/2< s_j(-2) <1} \frac{X^{2 s_j(-2)}}{s_j(-2)} \\
& + O(x^{3/2}) \quad \quad (x \to \infty), 
\end{split}
\end{equation}
\begin{equation}
\begin{split}
\sum_{\substack{d \in \mathcal{D}_{+-} \\
    \varepsilon_K(d) \le x}}
h_{K}(d)
=& \, 2 \Li(x^2) 
- \sum_{1/2< s_j(2) <1}  \Li \big( x^{2 s_j(2)} \bigr) 
- \sum_{1/2< s_j(-2) <1}  \Li \big( x^{2 s_j(-2)} \bigr) \\
& +O(x^{3/2} /\log x) \quad \quad (x \to \infty)
.
\end{split}
\end{equation}
\end{theorem}
\begin{proof}
We recall the assumption that the class number of $K$ is one, 
so that $\mathcal{O}_{K}$ is a PID.
Let $Q(x,y)=ax^2+bxy+cy^2$ be a primitive indefinite 
quadratic forms of discriminant $d \in \mathcal{D}_{+-}$ over $\mathcal{O}_{K}$, 
i.e. $a,b,c \in \mathcal{O}_{K}$, the ideal $(a,b,c) = \mathcal{O}_{K}$
and $d=b^2-4ac>0$, $d'<0$.

The equation $Q(\theta,1)=0$ has two real roots, 
$\theta_1=(-b+\sqrt{d})/2a$ and $\theta_2=(-b-\sqrt{d})/2a$.
By linear change of variable, $\mathrm{SL}(2,\mathcal{O}_{K})$
acts on such forms and the number of equivalence classes is $h_{K}(d)$.
The stabilizer of $Q$ under the action of  $\mathrm{SL}(2,\mathcal{O}_{K})$
is equal to the stabilizer of $\theta_1$ or $\theta_2$ and become a
free abelian group of rank one.
And a generator of this group is given by
\[ g(Q) = 
\Bigl( 
\begin{array}{cc}
(t_0-bu_0)/2 & -c u_0 \\
a u_0 & (t_0+b u_0)/2
\end{array} \Bigr)
\]
where, $(t_0,u_0)$ is the fundamental solution to the Pellian equation
$t^2-du^2=4$ over $\mathcal{O}_{K}$. 
Moreover the norm of $g(Q)$ is $\ve_{K}(d)^2$
with $\ve_K(d)=(t_0+u_0\sqrt{d})/2$.

The map $Q \mapsto (g,g')$ sends primitive $\mathcal{O}_{K}$-integral 
quadratic forms to hyperbolic-elliptic conjugacy classes
of $\Gamma_{K}$. 
It is known that  it induces a bijection between classes of forms
and primitive hyperbolic-elliptic conjugacy classes of $\Gamma_K$.
(We refer to Efrat \cite{E} for details).
Hence, we obtain the desired formula from
Theorem \ref{th:pgt}.
\end{proof}

%%%%%%%%%%%%%%%%%%%%%%%%%%%%%%%%%%%%%%%%%%%%%%%%%%%%%%%%%%%%%%%%%%%%%%%%%
%%%%%%%%%%%%%%%%%%%%%%%%%%%%%%%%%%%%%%%%%%%%%%%%%%%%%%%%%%%%%%%%%%%%%%%%%
%%%    References
%%%%%%%%%%%%%%%%%%%%%%%%%%%%%%%%%%%%%%%%%%%%%%%%%%%%%%%%%%%%%%%%%%%%%%%%%
%%%%%%%%%%%%%%%%%%%%%%%%%%%%%%%%%%%%%%%%%%%%%%%%%%%%%%%%%%%%%%%%%%%%%%%%%
% ------------------------------------------------------------------------
\bibliographystyle{plain}
\def\cprime{$'$} \def\polhk#1{\setbox0=\hbox{#1}{\ooalign{\hidewidth
  \lower1.5ex\hbox{`}\hidewidth\crcr\unhbox0}}}
\providecommand{\bysame}{\leavevmode\hbox
to3em{\hrulefill}\thinspace}
\providecommand{\MR}{\relax\ifhmode\unskip\space\fi MR }
\providecommand{\MRhref}[2]{%
  \href{http://www.ams.org/mathscinet-getitem?mr=#1}{#2}
} \providecommand{\href}[2]{#2}

\end{document}